\documentclass[english]{article}

\newlength{\tilesize}
\newlength{\smalltilesize}
\usepackage[T1]{fontenc}
\usepackage[latin9]{inputenc}
\usepackage{float}
\usepackage{amsmath}
\usepackage{graphicx}
\usepackage{amssymb}
\usepackage{framed}
\usepackage{subfigure}
\usepackage{babel}
\usepackage[a4paper]{geometry}
\usepackage{enumerate}
\usepackage{longtable}
\usepackage{mdwlist}
\usepackage{paralist}

\usepackage{pgf,tikz}
\usetikzlibrary[patterns]
\usetikzlibrary{arrows}
\definecolor{wwwwww}{rgb}{0.4,0.4,0.4}
\definecolor{cccccc}{rgb}{0.8,0.8,0.8}

\usepackage[colorlinks=true, linkcolor=black, citecolor=black]{hyperref}

\newtheorem{theorem}{Theorem}[section]
\newtheorem{lemma}[theorem]{Lemma}

\newtheorem{definition}[theorem]{Definition}

\newtheorem{remark}[theorem]{Remark}
\newtheorem{claim}[theorem]{Claim}

\numberwithin{equation}{section}

\newcommand{\qed}{\hspace*{\fill} $\blacksquare$\medskip}

\newenvironment{proof}
{\noindent {\em Proof}.\,\,}
{\qed}

\def \Z {\mathbb Z}

\def \N {\mathbb N}

\def \cX {\mathcal X}
\def \cY {\mathcal Y}
\def \cT {\mathcal T}

\def \cC {\mathcal C}
\def \cV {\mathcal V}
\def \cI {\mathcal I}
\def \cB {\mathcal B}

\def \cA {\mathcal A}
\def \cS {\mathcal S}
\def \cW {\mathcal W}
\def \cG {\mathcal G}

\def \cP {\mathcal P}
\def \cO {\mathcal O}
\def \cD {\mathcal D}

\def \cU {\mathcal U}

\def \epsi {\varepsilon}
\def \prm 	 {^{\prime}}
\def \dprm {^{\prime\prime}}
\def \starred {^{\star}}

\def \boundary {\partial}
\def \CAPA {{\rm CAP}}
\def \CS {{\rm CS}}

\def \quasiqnumber {\ell\starred(\ell\starred - 1)}
\def \protonumber {\quasiqnumber + 1}
\def \critinumber {\quasiqnumber + 2}
\def \boub {^{\mathrm{bd2}}}

\def \supp	 {\text{supp}}
\def \subconf {\prec}

\def \Lambdaminus {\Lambda^{-}}

\def \metaset	{\cX_\mathrm{meta}}
\def \groundset 	{\cX_\mathrm{stab}}
\newcommand 	{\stablev}[1] {V_{#1}}
\def \comlev 	{\varPhi}
\newcommand 	{\lowset}[1] {\cI_{#1}}

\def \molset	 {\cV_{\star,\nb}^{4\nb}}
\def \nbset		{\cV_{\star,\nb}}
\newcommand {\nbis}[1]	 {\cV_{\star, #1}}
\newcommand {\standard}[1]	{\ensuremath{\hat{\mathcal{S}}_{#1} }}

\def \gate {\cC^{\star}}
\def \proto {\cP}
\def \entgate {\gate_\mathrm{bd}}
\def \gateatt {\gate_\mathrm{att}}

\def \g {g(\Box, \minnbis{\protonumber})}
\def \gbar {\bar{g}(\Box, \minnbis{\protonumber})}

\newcommand \minnbis[1] {\bar{\cV}_{\star, #1}}
\newcommand \entrance[1] {{\Rsh}#1}
\newcommand \optentrance[1] {{\bar\Rsh}#1}

\newcommand {\broken}[1] {B^{-}(#1)}
\newcommand {\extrapart}[1] {n_{1}^{+}(#1)}

\def \optpaths	{\ensuremath{(\Box\to\boxplus)_{\mathrm{opt}}} }

\def \setb {\varpi_{2}} 

\def \ta	{1}
\def \tb	{2}
\def \na	{n_{\ta}}
\def \nb	{n_{\tb}}

\def \abbar	{\ta\tb\text{--bar}}
\def \abbars	{\ta\tb\text{--bars}}
\def \btile		{\tb\text{--tile}}
\def \btiles		{\tb\text{--tiles}}
\def \btiled		{\tb\text{--tiled}}

\def \D	 {\Delta}
\def \Da	 {\Delta_{1}}
\def \Db	 {\Delta_{2}}

\newcommand {\tsupp}[1] {[#1]}

\def \tile {\text{t}}

\def \RT {{\bf RT }}
\def \RA {{\bf RA }}
\def \RB {{\bf RB }}
\def \RC {{\bf RC }}

\def \RF {{\bf RF }}

\def \CA {\ensuremath{\cD_{A}} }
\def \CB {\ensuremath{\cD_{B}} }
\def \CC {\ensuremath{\cD_{C}} }
\def \CF {\ensuremath{\cD_{F}} }

\newcommand{\checkref} {} 

\newcommand{\step}[1] {\medskip\noindent \underline{Step #1:} }

\newcounter{mylist1counter}
\newenvironment{mytextlikelist}
{	\begin{list}{ (\roman{mylist1counter})~}
	{\usecounter{mylist1counter} \itemsep=0em \topsep=0em \parsep=0em }
}
{	
	\end{list}
}

\newcounter{alongedgescounter}

\def \notesmessage{}

\begin{document}

\author{
\renewcommand{\thefootnote}{\arabic{footnote}}
F.\ den Hollander \footnotemark[1] \, \footnotemark[2]
\\
\renewcommand{\thefootnote}{\arabic{footnote}}
F.R.\ Nardi \footnotemark[3] \, \footnotemark[2]
\\
\renewcommand{\thefootnote}{\arabic{footnote}}
A.\ Troiani \footnotemark[1]
}

\title{Kawasaki dynamics with two types of particles:\\
critical droplets}

\footnotetext[1]{
Mathematical Institute, Leiden University, P.O.\ Box 9512,
2300 RA Leiden, The Netherlands
}
\footnotetext[2]{
EURANDOM, P.O.\ Box 513, 5600 MB Eindhoven, The Netherlands
}
\footnotetext[3]{
Technische Universiteit Eindhoven, P.O.\ Box 513, 5600 MB Eindhoven, 
The Netherlands
}

\maketitle

\marginpar{\footnotesize{\notesmessage}}


\begin{abstract}
This is the third in a series of three papers in which we study a 
two-dimensional lattice gas consisting of two types of particles
subject to Kawasaki dynamics at low temperature in a large finite 
box with an open boundary. Each pair of particles occupying neighboring 
sites has a negative binding energy provided their types are different, 
while each particle has a positive activation energy that depends on 
its type. There is no binding energy between particles of the same 
type. At the boundary of the box particles are created and annihilated 
in a way that represents the presence of an infinite gas reservoir. 
We start the dynamics from the empty box and are interested in the 
transition time to the full box. This transition is triggered by 
a critical droplet appearing somewhere in the box.

In the first paper we identified the parameter range for which the
system is metastable, showed that the first entrance distribution on 
the set of critical droplets is uniform, computed the expected transition 
time up to and including a multiplicative factor of order one, and proved 
that the nucleation time divided by its expectation is exponentially 
distributed, all in the limit of low temperature. These results were proved 
under \emph{three hypotheses}, and involved \emph{three model-dependent quantities}: 
the energy, the shape and the number of critical droplets. In the second paper 
we proved the first and the second hypothesis and identified the energy of 
critical droplets. In the third paper we prove the third hypothesis and identify 
the shape and the number of critical droplets, thereby completing our analysis.

Both the second and the third paper deal with understanding the \emph{geometric 
properties} of subcritical, critical and supercritical droplets, which are 
crucial in determining the metastable behavior of the system, as explained in 
the first paper. The geometry turns out to be considerably more complex than for 
Kawasaki dynamics with one type of particle, for which an extensive literature 
exists. The main motivation behind our work is to understand metastability of 
multi-type particle systems. 

\end{abstract}


\pagebreak


\section{Introduction}
\label{sec introduction}

\paragraph{Motivation.}
The main motivation behind the present work is to understand metastability of \emph{multi-type
particle systems} subject to \emph{conservative stochastic dynamics}. In the past ten
years, a good understanding has been achieved of the metastable behavior of the lattice 
gas subject to Kawasaki dynamics, i.e., a conservative dynamics characterized by random 
hopping of particles of a single type with hardcore repulsion and nearest-neighbor 
attraction. The analysis was based on a combination of techniques from large deviation 
theory, potential theory, geometry and combinatorics. In particular, a precise description 
has been obtained of the time to nucleation (from the ``gas phase'' to the ``liquid 
phase''), the shape of the critical droplet triggering the nucleation, and the typical 
nucleation path, i.e., the typical growing and shrinking of droplets. For an overview 
we refer the reader to two recent papers presented at the 12th Brazilian School of 
Probability: Gaudilli\`ere and Scoppola~\cite{GSnotes} and Gaudilli\`ere~\cite{Gnotes}. 
For an overview on metastability and droplet growth in a broader context, we refer the 
reader to the monograph by Olivieri and Vares~\cite{OV04}, and the review papers by 
Bovier~\cite{B09}, \cite{B11}, den Hollander~\cite{dH09}, Olivieri and Scoppola~\cite{OS10}. 
 
The model we study constitutes a first attempt to generalize the results in Bovier, den 
Hollander and Nardi~\cite{BdHN06} for two-dimensional Kawasaki dynamics with one type of 
particle to \emph{multi-type particle systems}. We take a large finite box $\Lambda 
\subset \Z^{2}$. Particles come in two types: type $1$ and type $2$. Particles hop around 
subject to hard-core repulsion, and are conserved inside $\Lambda$. At the boundary of 
$\Lambda$ particles are created and annihilated as in a gas reservoir, where the two 
types of particles have different densities $e^{-\beta\Da}$ and $e^{-\beta\Db}$. We 
assume a binding energy $U$ between particles of different type, and no binding energy 
between particles of the same type. Because of the ``antiferromagnetic'' nature of the 
interaction, configurations with minimal energy have a ``checkerboard'' structure. The 
phase diagram of this simple model is already very rich. The model can be seen as a 
conservative analogue of the Blume-Capel model investigated by Cirillo and 
Olivieri~\cite{CO96}. 

Our model describes the condensation of a low-temperature and low-density supersaturated 
lattice gas. We are interested in studying the nucleation towards the liquid phase 
represented by the checkerboard configuration $\boxplus$, starting from the gas phase 
represented by the empty configuration $\Box$. It turns out that the \emph{geometry} of 
the energy landscape is much more complex than for the model of Kawasaki dynamics with 
one type of particle. Consequently, it is a somewhat delicate matter to capture the 
proper mechanisms behind the growing and shrinking of droplets. Our proofs use potential 
theory and rely on ideas developed in Bovier, den Hollander and Nardi~\cite{BdHN06} for 
Kawasaki dynamics with one type of particle.

\paragraph{Two previous papers.}
In $\cite{dHNT12}$ we identified the values of the parameters for which the model properly 
describes the condensation of a supersaturated gas and exhibits a metastable behavior. 
Under \emph{three hypotheses}, we determined the distribution and the expectation of the 
nucleation time, and identified the so-called critical configurations that satisfy a 
certain ``gate property''. The first hypothesis assumes that configuration $\boxplus$, 
corresponding to the liquid phase, is a minimizer of the Hamiltonian. The second hypothesis 
requires that the valleys of the energy landscape are not too deep. The third hypothesis 
requires that the critical configurations have an appropriate geometry. Subject to the 
three hypotheses, several theorems were derived, for which \emph{three model-dependent
quantities} needed to be identified as well: (1) the energy barrier $\Gamma\starred$ 
separating $\Box$ from $\boxplus$; (2) the set $\gate$ of critical configurations; 
(3) the cardinality $N\starred$ of the set of protocritical configurations, which can 
be thought of as the ``entrance'' set of $\gate$. Quantity (1) was identified in 
\cite{dHNT11}. In the present paper we identify quantities (2) and (3).

In \cite{dHNT11} the first two hypotheses were verified and the energy value $\Gamma\starred$ 
of the energy barrier separating $\Box$ and $\boxplus$ is identified. These results were 
sufficient to establish the exponential probability distribution of the nucleation time 
divided by its mean, and to determine the mean nucleation time up to a multiplicative 
factor $K$ of order $1+o(1)$ as the inverse temperature $\beta \to \infty$. 

\paragraph{Present paper.}
In the present paper we show that the model satisfies the third hypothesis, and we identify
the set of critical configurations. We give a geometric characterization of the configurations 
in $\gate$ and compute the value of $N\starred$. A prototype of the critical configuration 
was already identified in \cite{dHNT11}, and consists of a configuration of minimal energy 
with $\protonumber$ particles of type $\tb$ arranged in a cluster of minimal energy plus a 
particle of type $\tb$. The difficult task is to characterize the \emph{full} set of critical
configurations. This part of the analysis uses the specific dynamical features of the model, 
which are investigated in detail in a neighborhood of the saddle configurations. This task 
is carried over by observing that, in the regime $0 < \Da < U < \Db$ and in configurations of 
minimal energy, each particle of type $\tb$ is surrounded by particles of type $\ta$. 
This allows us to look at configurations of minimal energy not as clusters of single particles, 
but as clusters of ``tiles'': particles of type $\tb$ surrounded by particles of type $\ta$
The tiles making up the cluster can travel around the cluster faster than particles of type 
$\tb$ can appear at the boundary of $\Lambda$. This motion of tiles along the border gives 
the dynamics the opportunity to extend the set of critical configurations. Different mechanisms 
are identified that allow tiles to travel around a cluster. The energy barrier that must be 
overcome in order to activate these mechanisms is determined, and is compared with the energy 
barrier the dynamics has to overcome in order to let a particle enter $\Lambda$. How rich the 
set of critical configurations is depends on the relative magnitude of these barriers. 
Consequently, the geometry of the critical configurations is highly sensitive to the choice 
of parameters.

The problem of computing the value $N\starred$, i.e., the cardinality of the set of 
protocritical configurations, is reduced to counting the number of polyominoes of minimal 
perimeter belonging to certain classes of configurations that depend on the values of 
the parameters $\Da$ and $\Db$. This is a non-trivial problem that is interesting in 
its own right. With these results we are able to derive the sharp asymptotics for the 
nucleation time and to find the entrance distribution of the set of critical 
configurations.

Results in this paper are derived by using a foliation of the state space according to
the number of particles of type $\ta$ in $\Lambda$, plus the fact that configurations 
in $\gate$ must satisfy a ``gate property'', i.e., they must be visited by all optimal 
paths. These results allow us to compute sharp asymptotic values for the expected 
nucleation time.

\paragraph{Literature.}
Similar analyses have been carried out both for conservative and non-conservative 
dynamics. For Ising spins subject to \emph{Glauber dynamics} in finite volume, 
a rough asymptotics for the nucleation time was derived by Neves and Schonmann~\cite{NS91} 
(on $\Z^{2}$) and by Ben Arous and Cerf~\cite{BAC96} (on $\Z^{3}$). Their results 
were improved by Bovier and Manzo~\cite{BM02}, where the potential-theoretic approach 
to metastability developed by Bovier, Eckhoff, Gayrard and Klein~\cite{BEGK02} was used 
to compute a sharp asymptotics for the nucleation time.

For the model with three-state spins (Blume--Capel model), the transition time and 
the typical trajectories were characterized by Cirillo and Olivieri~\cite{CO96}. 
For conservative Kawasaki dynamics, metastable behavior was studied in den Hollander, 
Olivieri and Scoppola~\cite{dHOS00} (on $\Z^{2}$) and in den Hollander, Nardi, 
Olivieri and Scoppola~\cite{dHNOS03} (on $\Z^{3}$). The sharp asymptotics of the 
nucleation time was derived by Bovier, den Hollander and Nardi~\cite{BdHN06}. Models 
with an anisotropic interaction were considered in Kotecky and Olivieri~\cite{KO93} 
for Glauber dynamics and in Nardi, Olivieri and Scoppola~\cite{NOS05} for Kawasaki
dynamics. 

The model studied in the present paper falls in the class of variations on Ising 
spins subject to Glauber dynamics and lattice gas particles subject to Kawasaki 
dynamics. These variations include staggered magnetic field, next-nearest-neighbor 
interactions, and probabilistic cellular automata. In all these models the geometry 
of the energy landscape is highly complex, and needs to be controlled in order to 
arrive at a complete description of metastability. For an overview, see the monograph 
by Olivieri and Vares~\cite{OV04}, Chapter 7.

\paragraph{Outline.}
Section~\ref{sec model and dynamics description} defines the model,
Section~\ref{sec basic not} introduces basic notation and key definitions, 
Section~\ref{sec key theorems} states the main theorems, while
Section~\ref{sec discussion} discusses these theorems.


\subsection{Lattice gas subject to Kawasaki dynamics}
\label{sec model and dynamics description}

Let $\Lambda \subset \Z^2$ be a large box centered at the origin (later it 
will be convenient to choose $\Lambda$ rhombus-shaped). Let $|\cdot|$ denote 
the Euclidean norm, let
\begin{equation}
\label{intextLam}
\begin{aligned}
\partial^-\Lambda &= \{x\in\Lambda\colon\,\exists\,y\notin\Lambda\colon\,|y-x|=1\},\\
\partial^+\Lambda &= \{x\notin\Lambda\colon\,\exists\,y\in\Lambda\colon\,|y-x|=1\},\\ 
\end{aligned}
\end{equation}
be the internal, respectively, external boundary of $\Lambda$, and put $\Lambda^-
= \Lambda\backslash\partial^-\Lambda$ and $\Lambda^+=\Lambda\cup\partial^+\Lambda$. 
With each site $x\in\Lambda$ we associate a variable $\eta(x) \in \{0,1,2\}$ indicating 
the absence of a particle or the presence of a particle of type $\ta$ or type $\tb$,
respectively. A configuration $\eta=\{\eta(x)\colon\,x\in\Lambda\}$ is an element of 
$\cX=\{0,1,2\}^\Lambda$. To each configuration $\eta$ we associate an energy 
given by the Hamiltonian
\begin{equation}
\label{Ham1}
H = -U \sum_{(x,y)\in\Lambda^{*,-}} 
1_{\{\eta(x)\eta(y) = 2\}}\\
+ \Da \sum_{x\in\Lambda} 1_{\{\eta(x)=1\}} 
+ \Db \sum_{x\in\Lambda} 1_{\{\eta(x)=2\}},
\end{equation}
where $\Lambda^{*,-}=\{(x,y)\colon\,x,y\in\Lambda^-,\,|x-y|=1;\,|x-z|>2,\,|y-z|>2\,\, 
\forall\,z \in \partial^{-}\Lambda\}$ is the set of non-oriented bonds in $\Lambda$ 
at distance at least 3 from $\partial^-\Lambda$, $-U<0$ is the \emph{binding energy} 
between neighboring particles of \emph{different} types in $\Lambda^-$, and $\Da>0$ 
and $\Db>0$ are the \emph{activation energies} of particles of type $\ta$, respectively, 
$\tb$ in $\Lambda$. 
The width is taken to be $3$ for technical convenience only. This
change does not effect the theorems in \cite{dHNT12} and \cite{dHNT11}, for which
the boundary plays no role. See also Appendix \ref{appb3}.
Without loss of generality we will assume that 
\begin{equation}
\Da \leq \Db.
\end{equation} 
The Gibbs measure associated with $H$ is 
\begin{equation}
\label{Gibbsmeasure}
\mu_\beta(\eta) = \frac{1}{Z_\beta}\,e^{-\beta H(\eta)}, \qquad \eta\in\cX,
\end{equation}
where $\beta\in (0,\infty)$ is the inverse temperature and $Z_\beta$ is the normalizing 
partition sum.

Kawasaki dynamics is the continuous-time Markov process $(\eta_t)_{t \geq 0}$ with 
state space $\cX$ whose transition rates are
\begin{equation}
\label{rate}
c_\beta(\eta,\eta\prm) = \left\{\begin{array}{ll}
e^{-\beta [H(\eta\prm)-H(\eta)]_+},
&\eta,\eta\prm\in\cX,\,\eta\neq\eta\prm,\,\eta\leftrightarrow\eta\prm,\\
0,
&\mbox{otherwise},
\end{array}
\right. 
\end{equation}
where $\eta\leftrightarrow\eta\prm$ means that $\eta\prm$ can be obtained from $\eta$ 
by one of the following moves:
\begin{itemize}
\item[$\bullet$]
interchanging $0$ and $1$ or $0$ and $2$ between two neighboring sites in $\Lambda$\\
(``hopping of particles in $\Lambda$''),
\item[$\bullet$]
changing $0$ to $1$ or $0$ to $2$ in $\partial^-\Lambda$\\
(``creation of particles in $\partial^-\Lambda$''),
\item[$\bullet$]
changing $1$ to $0$ or $2$ to $0$ in $\partial^-\Lambda$\\
(``annihilation of particles in $\partial^-\Lambda$'').
\end{itemize}
Note that this dynamics preserves particles in $\Lambda^-$, but allows particles to be 
created and annihilated in $\partial^-\Lambda$. Think of the latter as describing particles
entering and exiting $\Lambda$ along non-oriented bonds between $\partial^+\Lambda$ and 
$\partial^-\Lambda$ (the rates of these moves are associated with the bonds rather than 
with the sites). The pairs $(\eta,\eta\prm)$ with $\eta\leftrightarrow\eta\prm$ are called 
\emph{communicating configurations}, the transitions between them are called \emph{allowed 
moves}. Note that particles in $\partial^-\Lambda$ do not interact: the interaction only 
works well inside $\Lambda^-$ (see \eqref{Ham1}). Also note that the Gibbs measure is the
reversible equilibrium of the Kawasaki dynamics:
\begin{equation}
\label{reveq}
\mu_\beta(\eta)c_\beta(\eta,\eta') = \mu_\beta(\eta')c_\beta(\eta',\eta) \qquad 
\forall\,\eta,\eta'\in\cX.
\end{equation} 

The dynamics defined by (\ref{Ham1}) and (\ref{rate}) models the behavior in $\Lambda$ of 
a lattice gas in $\Z^2$, consisting of two types of particles subject to random hopping, 
hard-core repulsion, and nearest-neigbor attraction between different types. We may think 
of $\Z^2\backslash\Lambda$ as an \emph{infinite reservoir} that keeps the particle densities 
fixed at $\rho_\ta=e^{-\beta\Da}$, respectively, $\rho_\tb=e^{-\beta\Db}$. In the above 
model this reservoir is replaced by an \emph{open boundary} $\partial^-\Lambda$, where 
particles are created and annihilated at a rate that matches these densities. Thus, the 
dynamics is a \emph{finite-state} Markov process, ergodic and reversible with respect to 
the Gibbs measure $\mu_\beta$ in \eqref{Gibbsmeasure}.

Note that there is \emph{no} binding energy between neighboring particles of the 
\emph{same} type (including such an interaction would make the model much more complicated). 
Consequently, our dynamics has an ``anti-ferromagnetic flavor'', and does \emph{not} 
reduce to Kawasaki dynamics with one type of particle when $\Da=\Db$. Also note that 
our dynamics does not allow swaps between particles, i.e., interchanging $1$ and $1$, 
or $2$ and $2$, or $1$ and $2$, between two neighboring sites in $\Lambda$. (The first 
two swaps would not effect the dynamics, but the third would; for Kawasaki dynamics with 
one type of particle swaps have no effect.)


\subsection{Basic notation and key definitions}
\label{sec basic not}

To state our main theorems in Section~\ref{sec key theorems}, we need some notation.

\begin{definition}
\label{def1}
{\rm (a)} $\Box$ is the configuration where $\Lambda$ is empty.\\
{\rm (b)} $\boxplus$ is the set consisting of the two configurations where $\Lambda$ 
is filled with the largest possible checkerboard droplet such that all particles
of type $\tb$ are surrounded by particles of type $\ta$
(see Section~\ref{sec coordinates}, item 3 and Section~\ref{sec def}, items 1--3).\\
{\rm (c)} $\omega \colon\,\eta\to\eta\prm$ is any (self-avoiding) path of allowed 
moves from $\eta \in \cX$ to $\eta\prm \in \cX$.\\
{\rm (d)} $\comlev(\eta,\eta\prm)$ is the communication height between $\eta,\eta\prm
\in\cX$ defined by
\begin{equation}
\comlev(\eta,\eta\prm) = \min_{\omega \colon\,\eta\to\eta\prm}
\max_{\xi\in\omega} H(\xi),
\end{equation}
and $\comlev(A,B)$ is its extension to non-empty sets $A,B\subset\cX$ defined by
\begin{equation}
\comlev(A,B) = \min_{\eta\in A,\eta\prm\in B} \comlev(\eta,\eta\prm).
\end{equation}
{\rm (e)} $\cS(\eta,\eta\prm)$ is the communication level set between $\eta$ and $\eta\prm$
defined by
\begin{equation}
\cS(\eta,\eta\prm) = \left\{\zeta\in\cX\colon\,\exists\,\omega\colon\,\eta\to\eta\prm,\,
\omega\ni\zeta\colon\,\max_{\xi\in\omega} H(\xi) = H(\zeta) = \Phi(\eta,\eta\prm)\right\}.
\end{equation}
A configuration $\zeta\in\cS(\eta,\eta\prm)$ is called a saddle for $(\eta,\eta\prm)$.\\
{\rm (f)} $\stablev{\eta}$ is the stability level of $\eta\in\cX$ defined by
\begin{equation}
\stablev{\eta} = \comlev({\eta},\lowset{\eta}) - H(\eta),
\end{equation}
where $\lowset{\eta}=\{\xi\in\cX\colon\,H(\xi)<H(\eta)\}$ is the set of
configurations with energy lower than $\eta$.\\
{\rm (g)} $\groundset = \{\eta\in\cX\colon\,H(\eta)=\min_{\xi\in\cX} H(\xi)\}$ is the 
set of stable configurations, i.e., the set of configurations with mininal energy.\\
{\rm (h)} $\metaset = \{\eta\in\cX\colon\,\stablev{\eta}=\max_{\xi\in\cX\backslash\groundset} 
\stablev{\xi}\}$ is the set of metastable configurations, i.e., the set of non-stable 
configurations with maximal stability level.\\
{\rm (i)} $\Gamma=\stablev{\eta}$ for $\eta\in\metaset$ (note that $\eta\mapsto\stablev{\eta}$ 
is constant on $\metaset$), $\Gamma\starred=\comlev(\Box,\boxplus) - H(\Box)$
(note that $H(\Box) = 0$). 
\end{definition}

\begin{definition}
\label{def3}
{\rm (a)} $(\eta\to\eta\prm)_\mathrm{opt}$ is the set of paths realizing the minimax in
$\Phi(\eta,\eta\prm)$.\\
{\rm (b)} A set $\cW\subset\cX$ is called a gate for $\eta\to\eta\prm$ if $\cW\subset
\cS(\eta,\eta\prm)$ and $\omega\cap\cW\neq\emptyset$ for all $\omega\in
(\eta\to\eta\prm)_\mathrm{opt}$.\\
{\rm (c)} A set $\cW\subset\cX$ is called a minimal gate for $\eta\to\eta\prm$ if
it is a gate for $\eta\to\eta\prm$ and for any $\cW\prm\subsetneq\cW$ there 
exists an $\omega\prm \in (\eta\to\eta\prm)_\mathrm{opt}$ such that $\omega\prm
\cap\cW\prm=\emptyset$.\\
{\rm (d)} A priori there may be several (not necessarily disjoint) minimal gates.
Their union is denoted by $\cG(\eta,\eta\prm)$ and is called the essential gate
for $(\eta\to\eta\prm)_\mathrm{opt}$. The configurations in $\cS(\eta,\eta\prm)
\backslash\cG(\eta,\eta\prm)$ are called dead-ends.\\
{\rm (e)} Let $S(\omega)=\{\arg\max_{\xi\in\omega}H(\xi)\}$. A saddle $\zeta\in
\cS(\eta,\eta\prm)$ is called unessential if, for all $\omega \in (\eta\to\eta
\prm)_{\mathrm{opt}}$ such that $\omega\ni\zeta$ the following holds: $S(\omega)
\backslash\{\zeta\}\neq\emptyset$ and there exists an $\omega\prm \in (\eta\to
\eta\prm)_{\mathrm{opt}}$ such that $S(\omega\prm)\subseteq S(\omega)\backslash
\{\zeta\}$.\\
{\rm (f)} A saddle $\zeta \in \cS(\eta,\eta\prm)$ is called essential if it is not 
unessential, i.e., if either of the following occurs:
\begin{mytextlikelist}
\item[{\rm (f1)}]
There exists an $\omega \in (\eta\to\eta\prm)_{\mathrm{opt}}$ such that $S(\omega)
= \{\zeta\}$.		
\item[{\rm (f2)}] 
There exists an $\omega \in (\eta\to\eta\prm)_{\mathrm{opt}}$ such that $S(\omega)
\supseteq \{\zeta\}$ and $S(\omega\prm) \nsubseteq S(\omega) \backslash \{\zeta\}$ 
for all $\omega\prm \in (\eta\to\eta\prm)_{\mathrm{opt}}$.
\end{mytextlikelist}
\end{definition}

\begin{lemma}
\label{lemma-mnos-essential-saddles}
{\rm [Manzo, Nardi, Olivieri and Scoppola~\cite{MNOS04}, Theorem~5.1]}\\
A saddle $\zeta \in \cS(\eta,\eta\prm)$ is essential if and only if $\zeta \in 
\cG(\eta,\eta\prm)$.	
\end{lemma}

In \cite{dHNT12} we are interested in the transition of the Kawasaki dynamics from 
$\Box$ to $\boxplus$ in the limit as $\beta\to\infty$. This transition, which is
viewed as a crossover from a ``gas phase'' to a ``liquid phase'', is triggered 
by the appearance of a \emph{critical droplet} somewhere in $\Lambda$. The critical 
droplets form a subset $\gate$ of the essential gate $\cG(\Box,\boxplus)$, and all 
have energy $\Gamma\starred$ (because $H(\Box)=0$). 

In \cite{dHNT12} we showed that the first entrance distribution on the set of critical 
droplets is uniform, computed the expected transition time up to and including a 
multiplicative factor of order one, and proved that the nucleation time divided by 
its expectation is exponentially distributed, all in the limit as $\beta\to\infty$. 
These results, which are typical for metastable behavior, were proved under 
\emph{three hypotheses}: 
\begin{itemize}
\item[(H1)] 
$\groundset=\boxplus$.
\item[(H2)] 
There exists a $V\starred<\Gamma\starred$ such that $V_\eta\leq V\starred$ for 
all $\eta\in\cX\backslash\{\Box,\boxplus\}$.
\item[(H3)]
See (H3-a,b,c) and Fig.~\ref{fig-H3illus} below.
\end{itemize}
The third hypothesis consists of three parts characterizing the entrance set of 
$\cG(\Box,\boxplus)$ and the exit set of $\cG(\Box,\boxplus)$. To formulate these 
parts some further definitions are needed.

\begin{definition}
\label{defdroplets-a}
{\rm (a)} $\entgate$ is the minimal set of configurations in $\cG(\Box,\boxplus)$ such that all 
paths in $(\Box\to\boxplus)_\mathrm{opt}$ enter $\cG(\Box,\boxplus)$ through $\entgate$.\\
{\rm (b)} $\proto$ is the set of configurations visited by these paths just prior to their first
entrance of $\cG(\Box,\boxplus)$.
\end{definition}

\begin{itemize}
\item[(H3-a)]
Every $\hat{\eta}\in\proto$ consists of a \emph{single droplet} somewhere in $\Lambda^-$. 
This single droplet fits inside an $L\starred \times L\starred$ square somewhere in 
$\Lambda^-$ for some $L\starred \in \N$ large enough that is independent of $\hat{\eta}$ 
and $\Lambda$. Every $\eta\in\entgate$ consists of a single droplet $\hat{\eta}\in\proto$ 
and one \emph{additional free particle} of type $\tb$ somewhere in $\partial^-\Lambda$.
\end{itemize}

\begin{definition}
\label{defdroplets-b}
{\rm (a)} $\gateatt$ is the set of configurations obtained from $\entgate$ by moving the 
free particle of type $\tb$ along a path of empty sites in $\Lambda$ and attaching it to
the single droplet (i.e., creating at least one additional active bond). This set 
decomposes as $\gateatt = \cup_{\hat{\eta}\in\proto} \gateatt(\hat{\eta})$.\\
{\rm (b)} $\gate$ is the set of configurations obtained from $\entgate$ by moving the
free particle of type $\tb$ along a path of empty sites in $\Lambda$ without ever
attaching it to the droplet. This set 
decomposes as $\gate=\cup_{\hat{\eta}\in\proto}\gate(\hat{\eta})$. 
\end{definition}

\noindent
Note that $\Gamma\starred=H(\gate)=H(\proto)+\Db$, and that $\gate$ consists of precisely 
those configurations ``interpolating'' between $\proto$ and $\gateatt$: a free 
particle of type $\tb$ enters $\partial^-\Lambda$ and moves to the single droplet where 
it attaches itself via an active bond, i.e., a bond between particles of type $\ta$ and 
$\tb$. Think of $\proto$ as the set of configurations where the dynamics is ``almost over 
the hill'', of $\gate$ as the set of configurations where the dynamics is ``on top of the 
hill'', and of the free particle as ``achieving the crossover'' when it attaches itself 
``properly'' to the single droplet (the meaning of the word ``properly'' will become clear 
in Section~\ref{Proofs}; see also \checkref{}\cite{dHNT12}, Section~2.4). The sets $\proto$ 
and $\gate$ are referred to as the \emph{protocritical droplets}, respectively, the 
\emph{critical droplets}.

\begin{itemize}
\item[(H3-b)] 
All transitions from $\gate$ that either add a particle in $\Lambda$ or increase the 
number of droplets (by breaking an active bond) lead to energy $>\Gamma\starred$.
\item[(H3-c)] 
All $\omega \in (\entgate \to \boxplus)_{\mathrm{opt}}$ pass through $\gate_{\mathrm{att}}$.
For every $\hat\eta \in \proto$ there exists a $\zeta \in \gate_{\mathrm{att}}(\hat\eta)$ 
such that $\Phi(\zeta,\boxplus)<\Gamma\starred$.
\end{itemize} 

\begin{figure}[htbp]
\begin{centering}
{\includegraphics[width=0.4\textwidth]{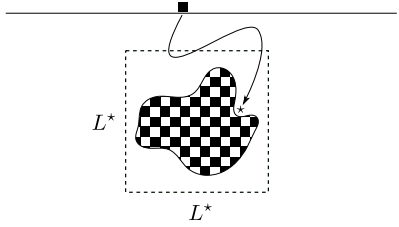}}
\par\end{centering}
\caption{A qualitative representation of a configuration in $\entgate$. If the free particle 
of type $\tb$ reaches the site marked as $\star$, then the dynamics has entered the ``basin 
of attraction'' of $\boxplus$.	}
\label{fig-H3illus}
\end{figure}

\medskip\noindent
{\bf Remark:}
Hypothesis (H3-a) and Definition \ref{defdroplets-b} are slightly different from how they
appear in \cite{dHNT12} and \cite{dHNT11}. This is done to make their analogues in 
\cite{dHNT12} and \cite{dHNT11} more precise, and to allow for a more precise proof 
of Lemma~1.18 and Lemma~2.2 in \cite{dHNT12}, which we repeat in Appendix~\ref{appB}.

\medskip
As shown in \cite{dHNT12}, (H1--H3) are needed to derive the metastability theorems in 
\cite{dHNT12} with the help of the \emph{potential-theoretic approach} to metastability 
outlined in Bovier~\cite{B09}. In \cite{dHNT11} we proved (H1--H2) and identified the 
energy $\Gamma\starred$ of critical droplets. In the present paper we prove (H3), identify 
the set $\gate$ of critical droplets, and compute the cardinality $N\starred$ of the set 
$\proto$ of protocritical droplets modulo shifts, thereby completing our analysis.

Hypotheses (H1--H2) imply that $(\metaset,\groundset)=(\Box,\boxplus)$, and that 
the highest energy barrier between a configuration and the set of configurations 
with lower energy is the one separating $\Box$ and $\boxplus$, i.e., $(\Box,\boxplus)$ 
is the unique \emph{metastable pair}. Hypothesis (H3) is needed to find the asymptotics 
of the prefactor of the expected transition time in the limit as $\Lambda\to\Z^2$ and 
will be proved in Theorem~\ref{th-H3} below. The main theorems in \cite{dHNT12} involve 
\emph{three model-dependent quantities}: the energy, the shape and the number of 
critical droplets. The first ($\Gamma\starred$) was identified in \cite{dHNT11}, 
the second ($\gate$) and the third ($N\starred$) will be identified in 
Theorems~\ref{th-RA}--\ref{th-RC} below.


\subsection{Main theorems}
\label{sec key theorems}

In \cite{dHNT12} it was shown that $0<\Da + \Db < 4U$ is the \emph{metastable region}, 
i.e., the region of parameters for which $\Box$ is a local minimum but not a global 
minimum of $H$. Moreover, it was argued that within this region the subregion where 
$\Da,\Db<U$ is of little interest because the critical droplet consists of two free 
particles, one of type $\ta$ and one of type $\tb$. Therefore the \emph{proper 
metastable region} is 
\begin{equation}
\label{propmetreg}
0<\Da \leq \Db, \quad \Da + \Db < 4U, \quad \Db \geq U,
\end{equation}
as indicated in Fig.~\ref{fig-propmetreg}.

\begin{figure}[htbp]
\begin{centering}
{\includegraphics[width=0.28\textwidth]{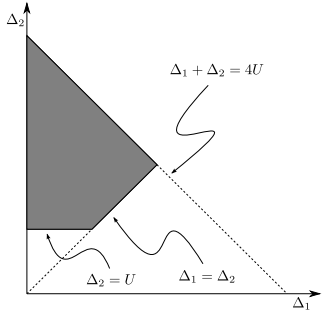}}
\par\end{centering}
\caption{Proper metastable region.}
\label{fig-propmetreg}
\end{figure}

In this present paper, as in \cite{dHNT11}, the analysis will be carried out for the 
subregion of the proper metastable region defined by
\begin{equation}
\label{subpropmetreg}
\Da < U, \quad  \Db - \Da > 2U, \quad \Da + \Db < 4U,
\end{equation}
as indicated in Fig.~\ref{fig-subpropmetreg}. (\emph{Note}: The second and third 
restriction imply the first restriction. Nevertheless, we write all three because 
each plays an important role in the sequel.)

\begin{figure}[htbp]
\begin{centering}
{\includegraphics[width=0.28\textwidth]{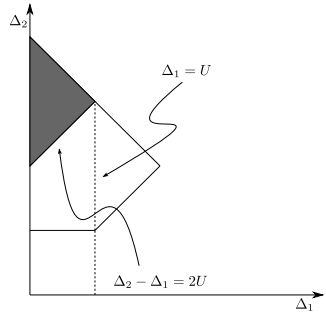}}
\par\end{centering}
\caption{Subregion of the proper metastable region given by (\ref{subpropmetreg}).}
\label{fig-subpropmetreg}
\end{figure}

Hypothesis (H3) involves additional characterizations of the sets $\proto$ and $\gate$. 
It turns out that these sets vary over the region defined in (\ref{subpropmetreg}).
The subregion where $\Db \le 4U - 2\Da$ is trivial: the configurations in $\proto$ 
consist of a single droplet, with one particle of type $\tb$ surrounded by four 
particles of type $\ta$, located anywhere in $\Lambda^-$. For this case, $\Gamma\starred
=4\Da+2\Db-4U$ and $N\starred=1$. We will split the subregion where $\Db > 4U - 2\Da$ 
into four further subregions (see Fig.~\ref{fig-subregions}). 

\begin{figure}[htbp]
\begin{centering}
\begin{tikzpicture}[line cap=round,line join=round,>=triangle 45,x=1.0cm,y=1.0cm]
\draw[->,color=black] (0,1.5) -- (0,5);
\draw[->,color=black] (0,1.5) -- (2.2,1.5);
\draw[shift={(0,2)},color=black]   (-2pt,0pt) node[left] {\scriptsize $2U$};
\draw[shift={(0,3)},color=black]   (-2pt,0pt) node[left] {\scriptsize $3U$};
\draw[shift={(0,4)},color=black]   (-2pt,0pt) node[left] {\scriptsize $4U$};
\draw[shift={(0,4.8)},color=black]   (-2pt,0pt) node[left] {\scriptsize $\Delta_{2}$};
\draw[shift={(2,1.5)},color=black]  (0pt,-2pt) node[below] {\scriptsize $\Delta_{1}$};
\draw[shift={(1,1.5)},color=black]  (0pt,-2pt) node[below] {\scriptsize $U$};
\clip(0,1.5) rectangle (1.49,4.5);
\fill[line width=0pt,fill=black,fill opacity=1.0] (0,4) -- (0,2) -- (0.86,2.86) -- cycle;
\fill[line width=0pt,color=cccccc,fill=cccccc,fill opacity=1.0] (0.75,3) -- (0.86,2.86) -- (1,3) -- cycle;
\fill[line width=0pt,color=wwwwww,fill=wwwwww,fill opacity=1.0] (0.75,3) -- (1,3) -- (0.67,3.33) -- (0.6,3.2) -- cycle;
\fill[line width=0pt,fill=black,pattern=north east lines] (0,4) -- (0.43,3.43) -- (0.5,3.5) -- cycle;
\draw (0,0)-- (2,2);
\draw (0,4)-- (2,2);
\draw (0,3)-- (1,3);
\draw (0,3)-- (0.5,3.5);
\draw (0,2)-- (0.67,3.33);
\draw (0,2)-- (1,3);
\draw (0,4)-- (0.86,2.86);
\draw [dotted] (1,1)-- (1,3);
\end{tikzpicture}
\par
\end{centering}

\caption{Subregions of the parameter space. In the black region: $\ell\starred \le 3$.
The regions $\RA$, $\RB$ and $\RC$ are, respectively, light gray, dark grey and dashed.}
\label{fig-subregions}
\end{figure}
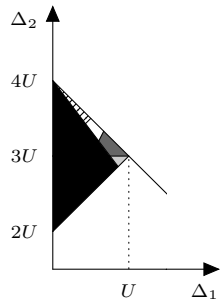

\medskip
For three of these subregions we will indentify $\proto$, $\gate$ and $\entgate$, prove 
(H3), and compute $N\starred$, namely,
\begin{itemize}
\item[$\RA$:] $\Da < 3U$;
\item[$\RB$:] $3U < \Db < 2U + 2\Da$;
\item[$\RC$:] $\Db > 3U + \Da$.
\end{itemize}
The fourth subregion is more subtle and is not analyzed in detail (see 
Section~\ref{sec discussion} for comments). All subregions are open sets. This is done 
to avoid parity problems. We also require that
\begin{equation}
\Da/\epsi \notin \N \quad \mbox{ with } \quad \epsi = 4U-\Da-\Db
\end{equation} 
and put
\begin{equation}
\label{lstardef}
\ell\starred = \left\lceil \frac{\Da}{\epsi}\right\rceil \in \N \setminus \{1\}.
\end{equation}

To state our main theorem we need the following definitions. A $\btile$ is a particle 
of type $\tb$ surrounded by four particles of type $\ta$. Dual coordinates map the 
support of a $\btile$ to a unit square. A monotone polyomino is a polyomino whose 
perimeter has the same length as that of its circumscribing rectangle. Given a set
of configurations $\mathcal{D}$, we write $\mathcal{D}\boub$ to denote the configurations 
obtained from $\mathcal{D}$ by adding a particle of type $\tb$ to a site in $\partial^{-}
\Lambda$. (For precise definitions see Sections~\ref{sec coordinates}--\ref{sec def}.)

\begin{definition}
\label{Dsetdefs}
(a) \CA is the set of $\btiled$ configurations with $\protonumber$ particles of type 
$\tb$ whose dual tile support is a rectangle of side lengths $\ell\starred,\ell\starred
-1$ plus a protuberance on one of the four side of the rectangle (see 
Fig.~{\rm \ref{fig:paradigm_class_A}}).\\
(b) $\CB$ is the set of $\btiled$ configurations with $\protonumber$ particles of type 
$\tb$ whose dual tile support is a monotone polyomino and whose circumscribing rectangle 
has side lengths either $\ell\starred, \ell\starred$ or $\ell\starred + 1, \ell\starred - 1$
(see Fig.~{\rm \ref{fig:paradigm_class_B}}).\\
(c) $\CC$ is the set of $\btiled$ configurations with $\protonumber$ particles of 
type $\tb$ whose dual tile support is a monotone polyomino and whose circumscribing 
rectangle has perimeter $4\ell\starred$ (see Fig.~{\rm \ref{fig:paradigm_class_F}}).\\
Note that $\CA \subseteq \CB \subseteq \CC$.
\end{definition}

\begin{theorem}
\label{th-H3}
Hypothesis {\rm (H3)} is satisfied in each of the subregions $\RA-\RC$.
\end{theorem}

\begin{theorem}
\label{th-RA}
In subregion $\RA$, $\proto = \CA$, $\entgate = \CA\boub$, and $N\starred 
= 8\ell\starred-4$.
\end{theorem}

\begin{theorem}
\label{th-RB}
In subregion $\RB$, $\proto = \CB$, $\entgate = \CB\boub$, and $N\starred 
= 8[q_{\ell\starred - 1} + r_{\ell\starred - 1 - 1}]$.
\end{theorem}

\begin{theorem}
\label{th-RC}
In subregion $\RC$, $\proto = \CC$, $\entgate = \CC\boub$, and $N\starred 
= 8[q_{\ell\starred - 1} + \sum_{c=1}^{\left\lfloor\sqrt{\ell\starred - 1}\right\rfloor} 
r_{\ell\starred - c^{2} - 1}]$.
\end{theorem}

\noindent
Here, $(r_k)$ and $(q_k)$ are the coefficients of two generating functions defined in 
Appendix~\ref{appA}, which count polyominoes with fixed volume and minimal perimeter.
The claims in Theorems~\ref{th-RA}--\ref{th-RC} are valid for $\ell\starred \geq 4$
only. For $\ell\starred=2,3$, see Section~\ref{sec-small-lc}.


\subsection{Discussion}
\label{sec discussion}

{\bf 1.} 
In \eqref{Ham1} we take an annulus of width 3 without interaction instead of an annulus 
of width 1 as in \cite{dHNT12} and \cite{dHNT11}. This allows us to prove that the model 
satisfies (H3), without having to deal with complications that arise when droplets are
too close to the boundary of $\Lambda$. 
In this case we would have to deal with the problem that particles cannot always travel 
around the clutser without leaving
$\Lambda$.The theorems in \cite{dHNT12} and \cite{dHNT11} 
remain valid.

\medskip\noindent
{\bf 2.} 
Theorems~\ref{th-H3} and \ref{th-RA}--\ref{th-RC}, together with the theorems 
presented in \cite{dHNT12} and \cite{dHNT11}, complete our analysis for part of 
the subregion given by \eqref{subpropmetreg}. Our results do not carry over to 
other values of the parameters, for a variety of reasons explained in \cite{dHNT12}, 
Section~1.5. In particular, for $\Da>U$ the critical droplets are square-shaped 
rather than rhombus-shaped. Moreover, (H2) is expected to be much harder to prove 
for $\Db-\Da<2U$.

\medskip\noindent
{\bf 3.} 
Theorems~\ref{th-RA}--\ref{th-RC} show that, even within the subregion given by 
\eqref{subpropmetreg}, the model-dependent quantities $\gate$ and $N\starred$, 
which play a central role in the metastability theorems in \cite{dHNT12}, are 
highly sensitive to the choice of parameters. This is typical for metastable 
behavior in multi-type particle systems, as explained in \cite{dHNT12}, Section~1.5. 

\medskip\noindent
{\bf 4.} 
The arguments used in the proof of Theorems~\ref{th-H3} and \ref{th-RA}--\ref{th-RC}  
are geometric. Along any optimal path from $\Box$ to $\boxplus$, as the energy gets 
closer to $\Gamma\starred$ the motion of the particles becomes more resticted. By 
analyzing this restriction in detail we are able to identify the shape of the critical 
droplets.   

\medskip\noindent
{\bf 5.}
The fourth subregion is more subtle. The protocritical set $\proto$ is somewhere between 
$\CB$ and $\CC$, and we expect $\proto=\CC$ for small $\Da$ and $\CB\subsetneq\proto
\subsetneq\CC$ for large $\Da$. The proof of Theorems~\ref{th-RA}--\ref{th-RC} in
Section~\ref{Proofs} will make it clear where the difficulties come from. 

\medskip\noindent
{\bf Outline:} In the remainder of this paper we provide further notation and 
definitions (Section~\ref{sec defnot}), state and prove a number of preparatory 
lemmas (Section~\ref{Preplem}), describe the motion of ``tiles'' along the boundary 
of a droplet (Section~\ref{tilemotion}), and give the proof of Theorem~\ref{th-H3} 
and \ref{th-RA}--\ref{th-RC} (Section~\ref{Proofs}). In Appendix~\ref{appA} we 
recall some standard facts about polyominoes with minimal perimeter.


\section{Coordinates and definitions}
\label{sec defnot}

Section~\ref{sec coordinates} introduces two coordinate systems that are used 
to describe the particle configurations: standard and dual. Section~\ref{sec def} 
lists the main geometric definitions that are needed in the rest of the paper.


\subsection{Coordinates}
\label{sec coordinates}

\newcounter{notcounter}
\begin{list}{\textbf{\arabic{notcounter}.~}}
{\usecounter{notcounter}
\labelsep=0em \labelwidth=0em \leftmargin=0em \itemindent=0em}

\item
A site $i\in\Lambda$ is identified by its \emph{standard coordinates} 
$x(i)=(x_{1}(i),x_{2}(i))$, and is called odd when $x_{1}(i)+x_{2}(i)$ is 
odd and even when $x_{1}(i)+x_{2}(i)$ is even. 
Given a configuration $\eta \in \cX$, a site $x \in \Lambda$ such that
$\eta(x)$ is $1$ or $2$ is referred to as a particle $p$ at site $x$.
The standard coordinates 
of a particle $p$ in a configuration $\eta$ 
are denoted by $x(p) = (x_{1}(p),x_{2}(p))$.
The \emph{parity} of a particle $p$ in a configuration $\eta$ is defined as 
$x_{1}(p)+x_{2}(p)+\eta(x(p))$ modulo 2, and $p$ is said to be odd when the 
parity is $1$ and even when the parity is $0$.

\item
A site $i\in\Lambda$ is also identified by its \emph{dual coordinates}
\begin{equation}
u_1(i) = \frac{x_1(i) - x_2(i)}{2}, \qquad u_2(i) = \frac{x_1(i) + x_2(i)}{2}.
\end{equation}
Two sites $i$ and $j$ are said to be \emph{adjacent}, written $i \sim j$, 
when $|x_1(i)-x_1(j)|+ |x_2(i)-x_2(j)|=1$ or, equivalently, $|u_1(i)-u_1(j)| 
= |u_2(i)-u_2(j)| = \tfrac12$ (see Fig.~\ref{fig-coord}).

\item
For convenience, we take $\Lambda$ to be the $(L+\tfrac32) \times (L+\tfrac32)$ 
dual square with bottom-left corner at site with dual coordinates $(-\frac{L+1}{2}, 
-\frac{L+1}{2})$ for some $L\in\N$ with $L>2\ell\starred + 2$ (to allow for $H(\boxplus)
< H(\Box)$). Particles interact only in a 
$(L-\tfrac{3}{2})\times (L-\tfrac{3}{2})$ dual square centered as $\Lambda$. 
This dual square, a \emph{rhombus} in standard 
coordinates, is convenient because the local minima of $H$ are rhombus-shaped as well 
(for more details see \cite{dHNT11}). 

\end{list}

\begin{figure}[htbp]
\centering
\subfigure[]
{\includegraphics[height=0.24\textwidth]{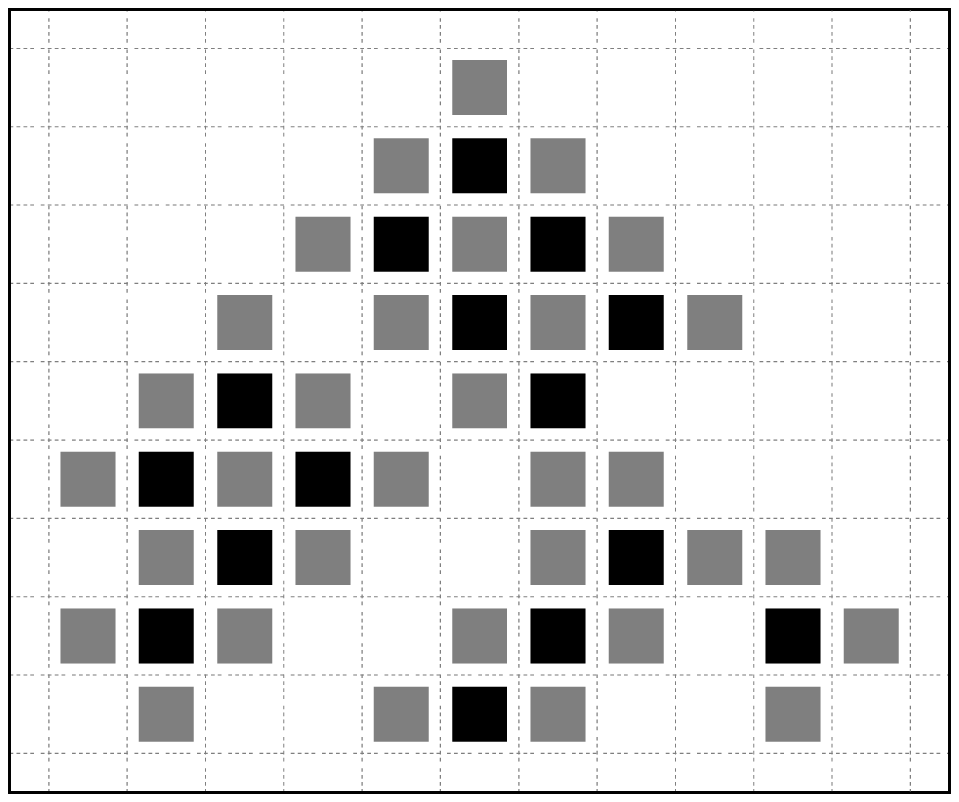}}
\qquad
\subfigure[]
{\includegraphics[height=0.24\textwidth]{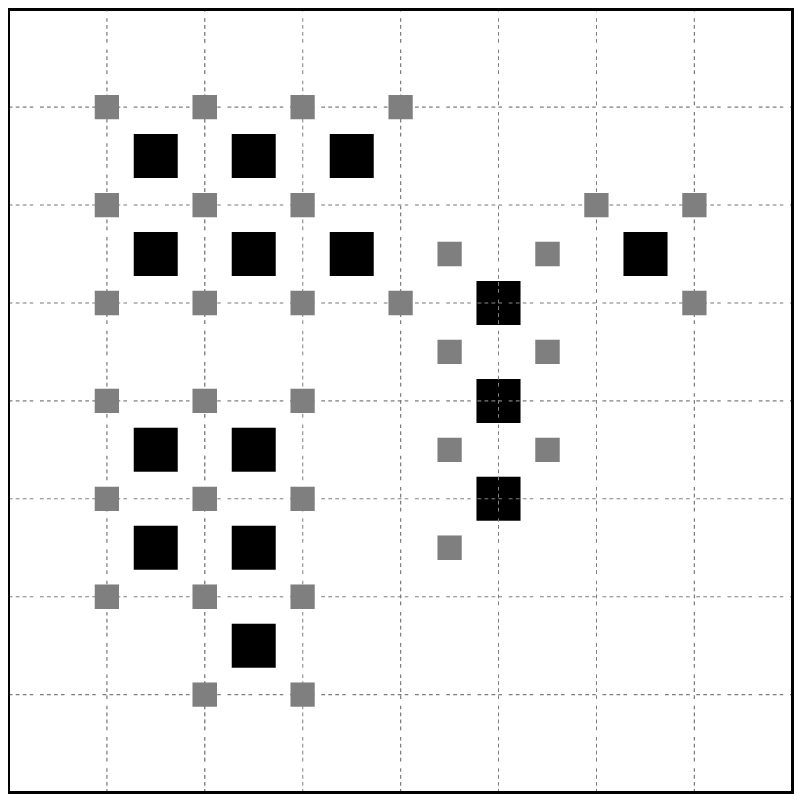}}
\caption{A configuration represented in: (a) standard coordinates; 
(b) dual coordinates. Light-shaded squares are particles of type $\ta$,
dark-shaded squares are particles of type $\tb$. In dual coordinates, 
particles of type $\tb$ are represented by larger squares than particles 
of type $\ta$ to exhibit the ``tiled structure'' of the configuration.}
\label{fig-coord}
\end{figure}


\subsection{Definitions}
\label{sec def}

\newcounter{defcounter}
\begin{list}{\textbf{\arabic{defcounter}.~}}
{\usecounter{defcounter}
\labelsep=0em \labelwidth=0em \leftmargin=0em \itemindent=0em}

\item
A site $i\in\Lambda$ is said to be \emph{lattice-connecting} in the configuration $\eta$ 
if there exists a lattice path $\lambda$ from $i$ to $\partial^-\Lambda$ such that 
$\eta(j) = 0$ for all $j \in \lambda$ with $j \neq i$. We say that a particle $p$ is 
lattice-connecting if $x(p)$ is a lattice-connecting site.
\label{def group lattice}

\item	
Two particles in $\eta$ at sites $i$ and $j$ are called \emph{connected} if $i \sim j$ and 
$\eta(i)\eta(j) = 2$. If two particles $p_{1}$ and $p_{2}$ are connected, then we say that 
there is an \emph{active bond} $b$ between them. The bond $b$ is said to be \emph{incident} 
to $p_{1}$ and $p_{2}$. A particle $p$ is said to be \emph{saturated} if it is connected to 
four other particles, i.e., there are four active bonds incident to $p$. The support of 
the configuration $\eta$, i.e., the union of the unit squares centered at the occupied 
sites of $\eta$, is denoted by $\supp{(\eta)}$. For a configuration $\eta$, $n_\ta(\eta)$ 
and $n_\tb(\eta)$ denote the number of particles of type $\ta$ and $\tb$ in $\eta$, and 
$B(\eta)$ denotes the number of active bonds. The energy of $\eta$ equals $H(\eta)
=\Da n_\ta(\eta)+\Db n_\tb(\eta)-UB(\eta)$.
\label{def group connected particles and bonds}

\item
Let $G(\eta)$ be the \emph{graph} associated with $\eta$, i.e., $G(\eta)=(V(\eta),E(\eta))$, 
where $V(\eta)$ is the set of sites $i\in\Lambda$ such that $\eta(i)\ne 0$, and $E(\eta)$ 
is the set of pairs $\{i,j\}$, $i,j\in V(\eta)$, such that the particles at sites $i$ 
and $j$ are connected. A configuration $\eta'$ is called a \emph{subconfiguration} of 
$\eta$, written $\eta' \subconf \eta$, if $\eta'(i)=\eta(i)$ for all $i\in\Lambda$ such 
that $\eta'(i)>0$. A subconfiguration $c\subconf\eta$ is called a \emph{cluster} if the 
graph $G(c)$ is a maximal connected component of $G(\eta)$. The set of non-saturated 
particles in $c$ is called the \emph{boundary} of $c$, and is denoted by $\boundary c$. 
Clearly, all particles in the same cluster have the same parity. Therefore the concept of 
parity extends from particles to clusters. 
\label{def group clusters and parity}

\item
For a site $i\in\Lambda$, the \emph{tile} centered at $i$, denoted by $\tile(i)$, is the 
set of five sites consisting of $i$ and the four sites adjacent to $i$. If $i$ is an even 
site, then the tile is said to be even, otherwise the tile is said to be odd. The five 
sites of a tile are labeled $a$, $b$, $c$, $d$, $e$ as in Fig.~\ref{fig-2tile}. The sites 
labeled $a$, $b$, $c$, $d$ are called \emph{junction sites}. If a particle $p$ sits at site 
$i$, then $\tile(i)$ is alternatively denoted by $\tile(p)$ and is called the tile associated 
with $p$. In standard coordinates, a tile is a square of size $\sqrt{2}$. In dual coordinates, 
it is a unit square. 
\label{def group tiles}

\item
A tile whose central site is occupied by a particle of type $\tb$ and whose junction 
sites are occupied by particles of type $\ta$ is called a \emph{$\btile$} (see
Fig.~\ref{fig-2tile}). Two $\btiles$ are said to be adjacent if their particles
of type $\tb$ have dual distance 1. A horizontal (vertical) \emph{$\abbar$} is a 
maximal sequence of adjacent $\btiles$ all having the same horizontal (vertical) 
coordinate. If the sequence has length $1$, then the $\abbar$ is called a \emph{$\btiled$ 
protuberance}. A cluster containing at least one particle of type $\tb$ such that 
all particles of type $\tb$ are saturated is said to be $\btiled$. A $\btiled$ 
configuration is a configuration consisting of $\btiled$ clusters only.
\label{def group btiles}
A \emph{hanging protuberance} (or hanging $\btile$) is a $\btile$ where three particles 
of type $\ta$ are adjacent to the particle of type $\tb$ of the $\btile$ only (see 
Fig.~\ref{fig:possible-reattachment}(b)).

\begin{remark}
\label{remark-supercritical-square}
A configuration consisting of a dual $\btiled$ square of side length $\ell\starred$
belongs to $\cX_{\boxplus}$.	
\end{remark}

\begin{figure}[htbp]
\centering
\subfigure[]
{\includegraphics[height=0.12\textwidth]{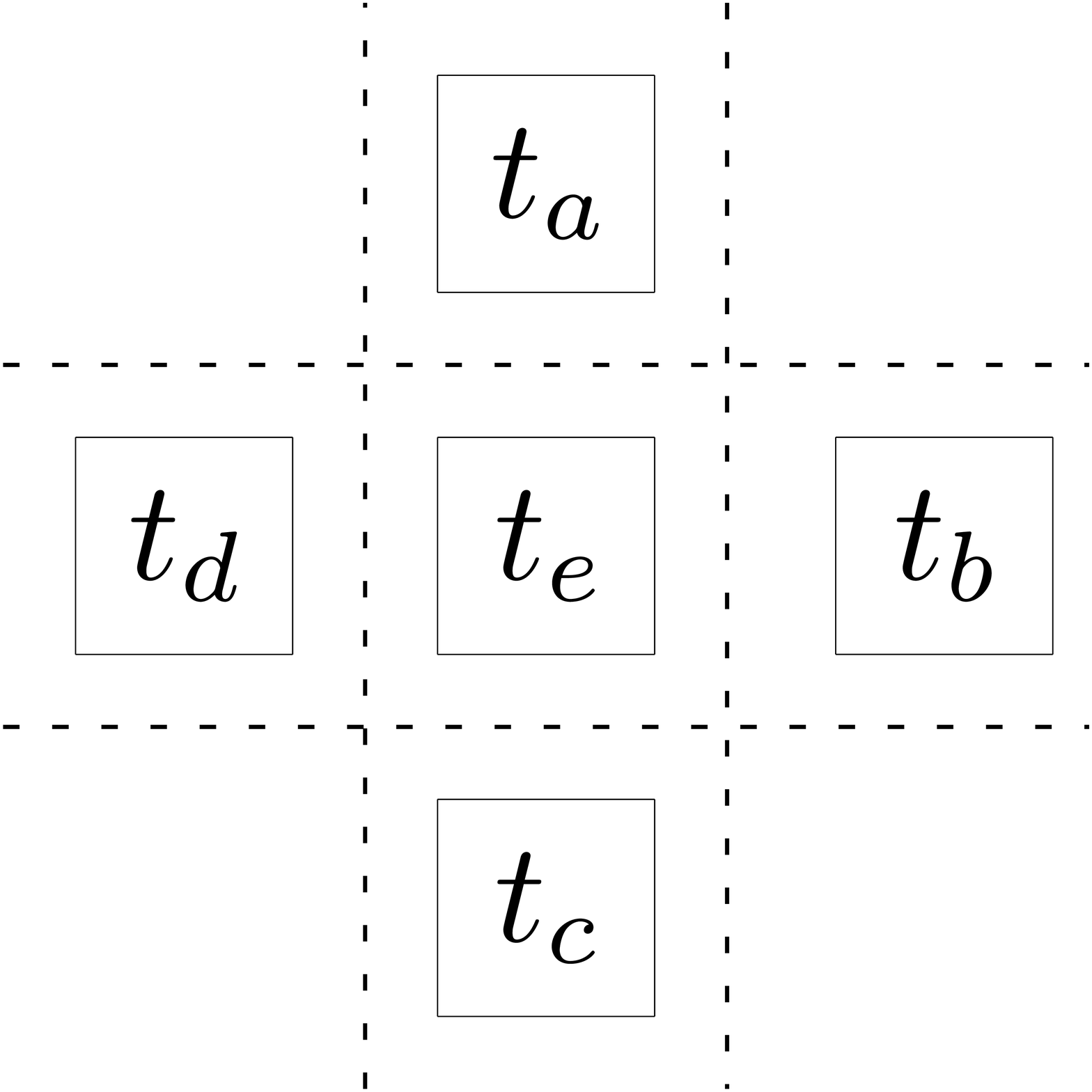}}
\quad
\subfigure[]
{\includegraphics[height=0.12\textwidth]{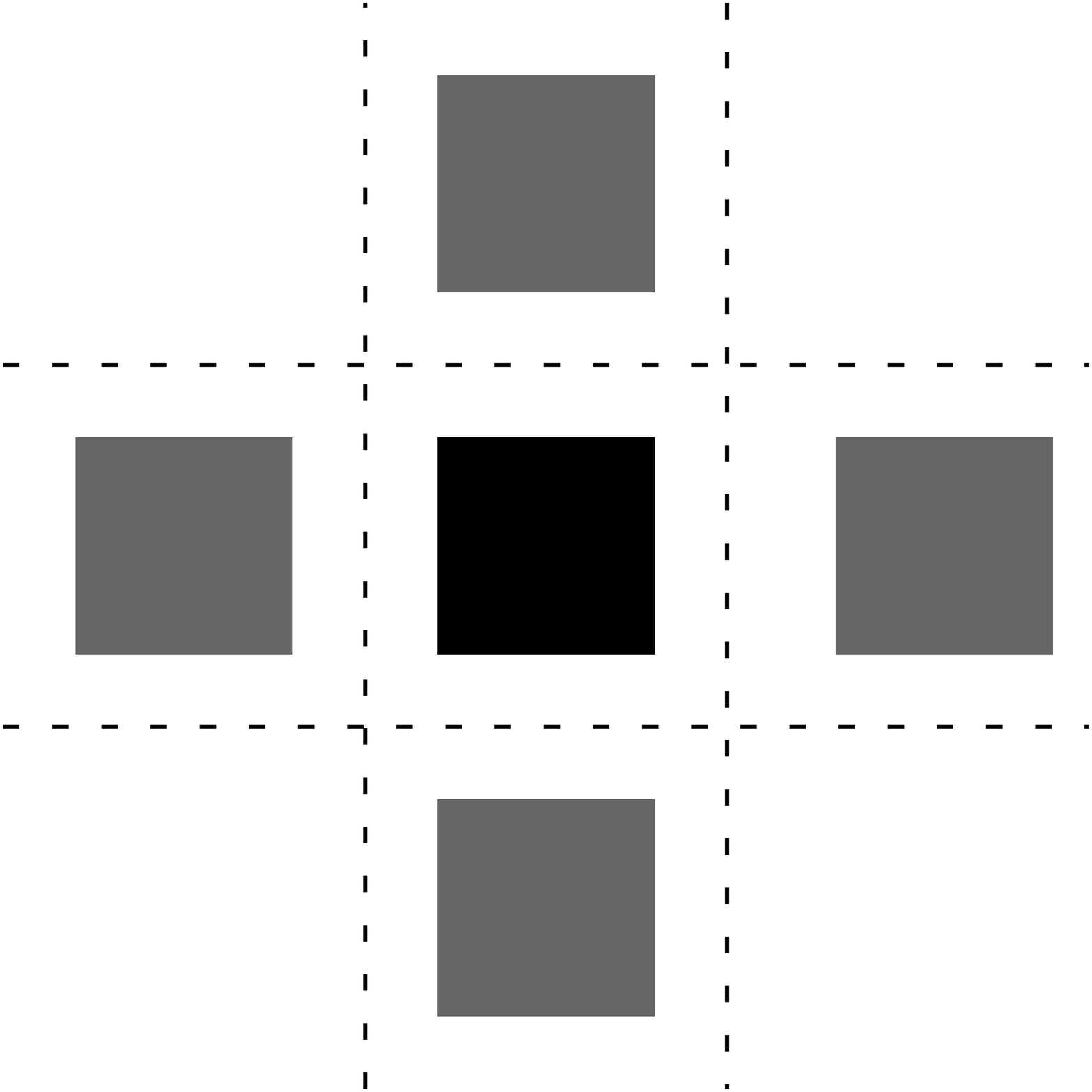}}
\quad
\subfigure[]
{\includegraphics[height=0.12\textwidth]{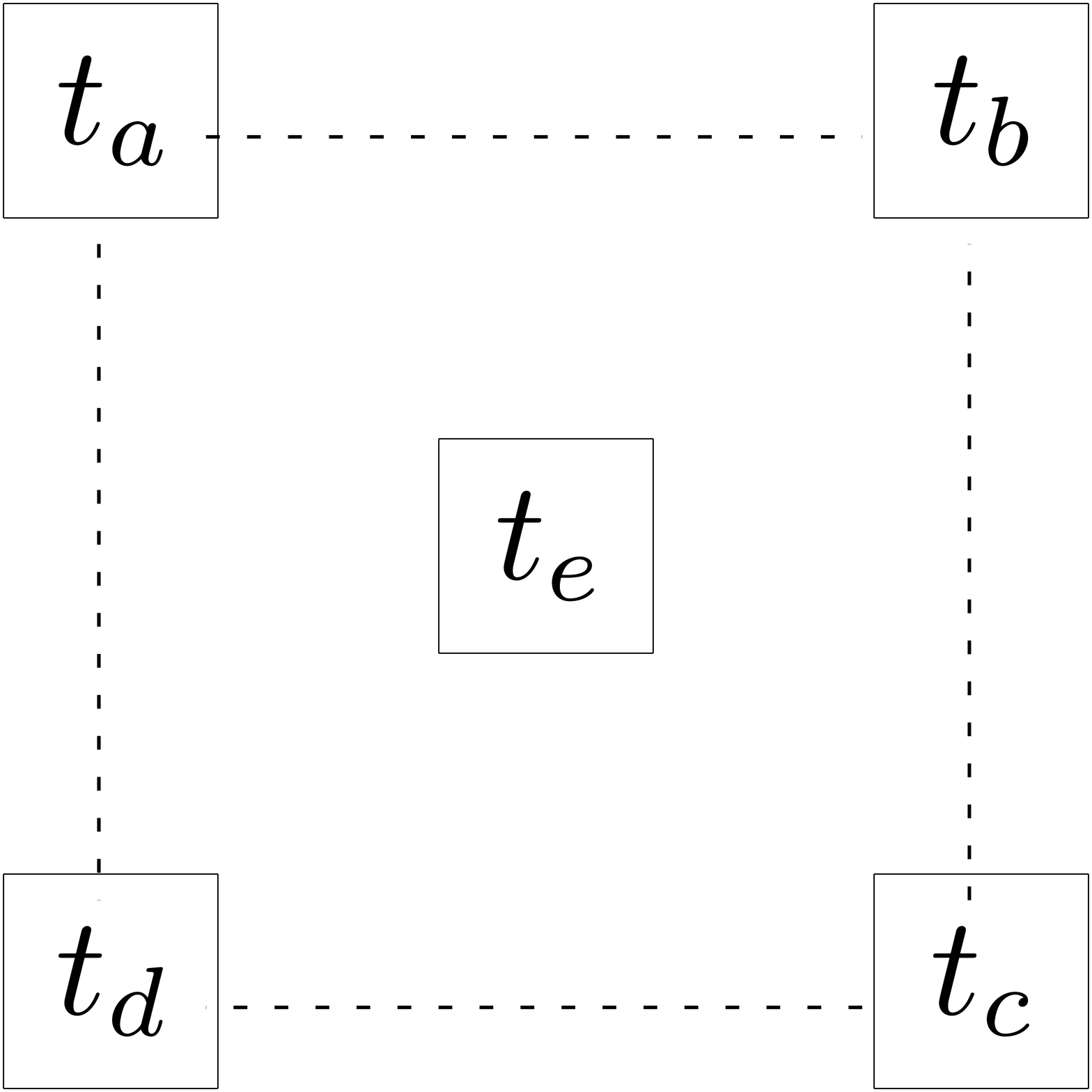}}
\quad
\subfigure[]
{\includegraphics[height=0.12\textwidth]{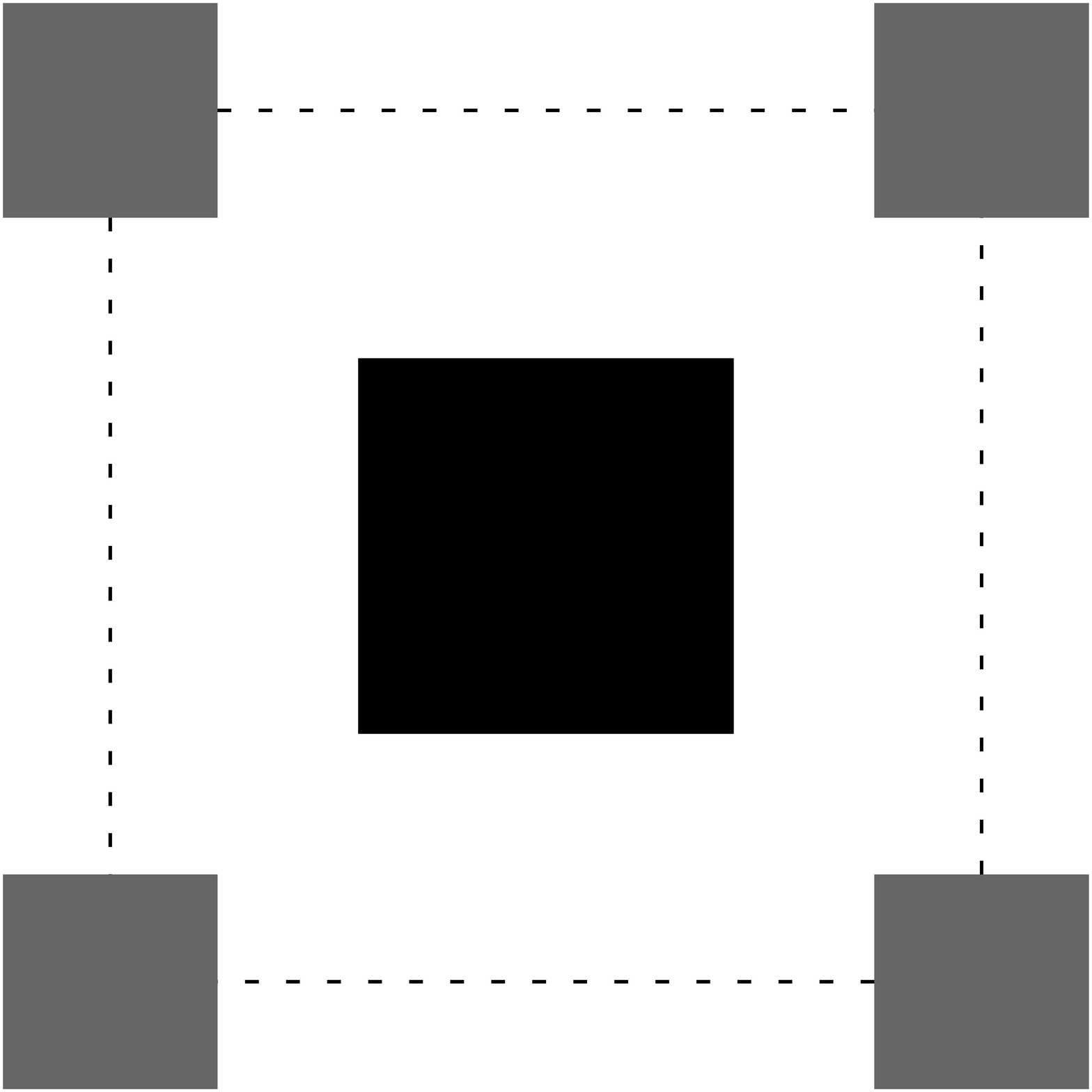}}
\caption{Tiles: (a) standard representation of the labels of a tile; (b) standard 
representation of a $\btile$; (c) dual representation of the labels of a tile; 
(d) dual representation of a $\btile$.}
\label{fig-2tile}
\end{figure}
\label{def group btile}

\item
The \emph{tile support} of a configuration $\eta$ is defined as
\begin{equation}
[\eta] = \bigcup_{p \in \setb(\eta) } \tile(p),
\end{equation}
where $\setb(\eta)$ is the set of particles of type $\tb$ in $\eta$. Obviously,
$[\eta]$ is the union of the tile supports of the clusters making up $\eta$.
\label{def group tile support}
For a standard cluster $c$ the \emph{dual perimeter}, denoted by $P(c)$, is the length of 
the Euclidean boundary of its tile support $[c]$ (which includes an inner boundary
when $c$ contains holes). The dual perimeter $P(\eta)$ of a $\btiled$ configuration 
$\eta$ is the sum of the dual perimeters of the clusters making up $\eta$.
\label{def group dual perimeter}

\item
Denote by $\nbset$ the set of configurations such that in $(\Lambda^{-})^{-}$ the number 
of particles of type $\tb$ is $\nb$. Denote by $\molset$ the subset of $\nbset$ where the 
number of active bonds is $4\nb$ and there are no non-interacting particles of type $\ta$, 
i.e., the set of $\btiled$ configurations with $\nb$ particles of type $\tb$. 
\label{def group set}
A configuration $\eta$ is called \emph{standard} if $\eta \in \molset$ and its tile 
support is a standard polyomino in dual coordinates (see Definition~\ref{def standard
polyominoes} below).
\label{def group standard configuration}
A configuration $\eta$ with $\nb(\eta)$ particles of type $\tb$ is called 
\emph{quasi-standard} if it can be obtained from a standard configuration with 
$\nb(\eta)$ particles of type $\tb$ by removing some (possibly none) of the particles 
of type $\ta$ with only one active bond, i.e., corner particles of type $\ta$.
Denote by $\minnbis{n}$ the set of configurations of \emph{minimal energy} in $\nbis{n}$. 

\item
The state space $\cX$ can be partitioned into manifolds: 
\begin{equation}
\cX = \bigcup_{n_2 = 0}^{|\Lambda|} \nbis{n_2}. 
\end{equation}
Two manifolds $\nbis{n}$ and $\nbis{n'}$ are called adjacent if $|n - n'| = 1$. Note that 
transitions between two manifolds are possible only when they are adjacent and are obtained 
either by adding a particle of type $\tb$ to $\partial^{-}\Lambda$ ($\nbis{n} \to \nbis{n+1}$) 
or removing a particle of type $\tb$ from $\partial^{-}\Lambda$ ($\nbis{n} \to \nbis{n-1}$).
Note further that $\Box \in \nbis{0}$ 
and $\boxplus \in \nbis{(L-2)^{2}}$. 
Therefore, to realize 
the transition $\Box \to \boxplus$, the dynamics must visit 
at least
all manifolds $\nbis{n}$ with 
$n = 1,\ldots,(L-2)^2$. 
Abbreviate $\nbis{\le m} = \bigcup_{n=0}^{m} \nbis{n}$ and 
$\nbis{\ge m}= \bigcup_{n=m}^{|\Lambda|}\nbis{n}$. 

\item
For $\cY \subset \cX$, $\hat{\eta} \in \cY$ and $x \in \Lambda\backslash\supp[\hat{\eta}]$, 
we write $\eta=(\hat{\eta},x)$ to denote the configuration that is obtained from $\hat{\eta}$ 
by adding a particle of type $\tb$ at site $x$. We write $\cY\boub$ to denote the set of configurations obtained from a configuration in $\cY$ by adding a particle of type $\tb$ 
in $\partial^{-}\Lambda$, i.e., $\cY\boub = \bigcup_{\hat{\eta} \in \cY} \bigcup_{x \in 
\partial^{-}\Lambda}(\hat{\eta}, x)$. For $\omega\colon\,\Box\to\boxplus$, let $\sigma_\cY
(\omega)$ be the configuration in $\cY$ that is first visited by $\omega$. Define
\begin{equation}
\begin{aligned}
\entrance{\cY} &= \bigcup_{\omega\colon\,\Box\to\boxplus} \sigma_\cY(\omega),\\
\optentrance{\cY} &= \bigcup_{ {\omega\colon\,\Box\to\boxplus} \atop {\text{optimal}} } 
\sigma_\cY(\omega),
\end{aligned}
\end{equation}
called the \emph{entrance}, respectively, the \emph{optimal extrance} of $\cY$. With this
notation we have $\entgate = \optentrance\cG(\Box,\boxplus)$.

\item
For $\cA, \cB \subset \cX$, define
\begin{equation}
\begin{aligned}
g(\cA,\cB) &= \{\eta\in\cB\colon\,\exists\,\zeta\in\cA \text{ and } 
\omega\colon\,\zeta \to \eta \colon\,
\nb(\eta) \le \nb(\xi) \le \nb(\zeta),\,H(\xi) < \Gamma\starred\,\forall\,\xi\in\omega\},\\
\bar{g}(\cA,\cB) &= \{\eta\in\cB\colon\,\exists\,\zeta\in\cA \text{ and } 
\omega\colon\,\zeta \to \eta \colon\,
\nb(\eta) \le \nb(\xi) \le \nb(\zeta),\,H(\xi) \le \Gamma\starred\,\forall\,\xi\in\omega\}.
\end{aligned}
\end{equation} 

\begin{figure}[htbp]
\centering
\subfigure[]
{\includegraphics[height=0.08\textwidth]{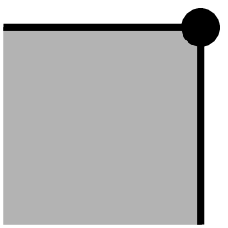}}
\qquad
\subfigure[]
{\includegraphics[height=0.08\textwidth]{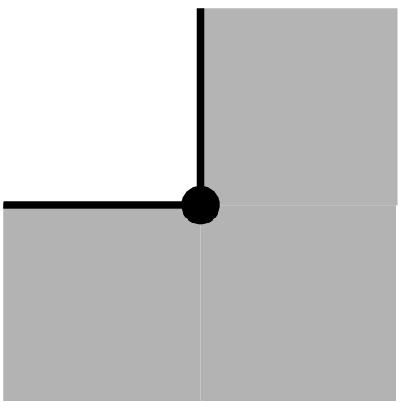}}
\qquad
\subfigure[]
{\includegraphics[height=0.10\textwidth]{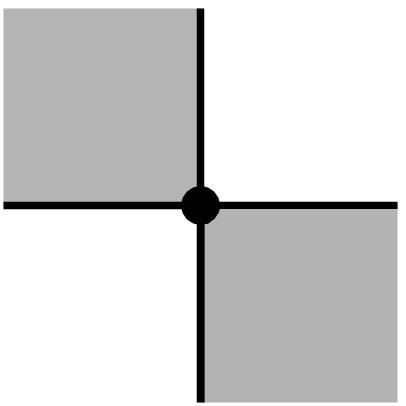}}
\caption{Corners of polyominoes: (a) one convex corner; (b) one concave corner; 
(c) two concave corners. Shaded mean occupied by a unit square.}
\label{fig:convex and concave corners}
\end{figure}

\item
A \emph{unit hole} is an empty site such that all four of its neighbors are 
occupied by particles of the same type (either all of type $\ta$ or all of type $\tb$).
\label{def group unit hoyle}
An empty site with three neighboring sites occupied by a particle of type $\ta$ 
is called a \emph{good dual corner}. In the dual representation a good dual corner 
is a concave corner (see Fig.~\ref{fig:convex and concave corners}).
\label{def group good corner}
The \emph{surface} of $\eta \in \cX$ is defined as
\begin{equation}
F(\eta) = \{x \in \Lambda\colon\,\exists\,y \sim x\colon\, \eta(y) = 1\}.
\end{equation}
For $\eta\in\cX$, let
\begin{equation}
\label{Tdef}
\cT(\eta) = 2P(\eta) + [\psi(\eta) - \phi(\eta)] = 2P(\eta) + 4[C(\eta)-Q(\eta)],
\end{equation}
where $C(\eta)$ is the number of clusters in $\eta$, $P(\eta)$ the total length of the 
perimeter of these clusters, $Q(\eta)$ the number of holes, $\psi(\eta)$ the number of 
convex corners, and $\phi(\eta)$ is the number of concave corners. Note that $\cT(\eta) 
= \sum_{c\in\eta} \cT(c)$, where the sum runs over the clusters in $\eta$.

\end{list}

We also need the following definition:

\begin{definition}
\label{def standard polyominoes}
{\rm [Alonso and Cerf~\cite{AC96}.]}
A polyomino (= a union of unit squares) is called monotone if its perimeter is equal 
to the perimeter of its circumscribing rectangle. A polyomino is called standard if 
its support is a quasi-square (i.e., a rectangle whose side lengths differ by at most 
one), with possibly a bar attached to one of its longest sides.
\end{definition}


\section{Preparatory lemmas}
\label{Preplem}

In this section we collect a number of preparatory lemmas that are valid throughout the 
subregion given by (\ref{subpropmetreg}). These lemmas will be needed in Section~\ref{Proofs} 
to prove Theorems~\ref{th-H3} and \ref{th-RA}--\ref{th-RC}. In Section~\ref{Preplem1} we 
characterize $\optentrance\nbis{\protonumber}$ 
(Lemmas~\ref{lemma-quasi-square-is-optimal}--\ref{lemma-entrance-of-set-protonumber} below), 
in Section~\ref{Preplem2} we characterize $g(\{\Box\}, \minnbis{\protonumber})$ and $\bar{g}
(\{\Box\},\minnbis{\protonumber})$ (Lemma~\ref{lemma-entrance-of-set-protonumber} below), 
and in Section~\ref{Preplem3} we characterize $\cG(\Box,\boxplus)$ 
(Lemmas~\ref{lemma-reachable-from-standard}--\ref{lemma-unessential-saddles} below). 

An elementary observation is the following:

\begin{lemma}
\label{lemma-minimal-energy-coincides-with-minimal-perimeter}
If $\eta \in \minnbis{n_{2}}$ with $\protonumber \le n_{2} \le (\ell\starred)^{2}$, 
then $\eta$ is $\btiled$ and its dual perimeter is equal to $4\ell\starred$.
\end{lemma}

\begin{proof}
Immediate from Lemmas~2.2--2.3 and 4.1 in \checkref{}\cite{dHNT11}, and also 
from Corollary 2.5 in \checkref{}\cite{AC96}.
\end{proof}


\subsection{Characterization of $\optentrance\minnbis{\protonumber}$}
\label{Preplem1}

\begin{lemma}
\label{lemma-quasi-square-is-optimal}
Let $\rho$ be a $\btiled$ configuration with $\quasiqnumber$ particles of type $\tb$ and 
with dual tile support equal to a rectangle of side lengths $\ell\starred,(\ell\starred-1)$ 
(i.e., $\rho$ is a standard configuration).\\ 
{\rm (1)} If $\eta \neq \rho$ is a $\btiled$ configuration with $\quasiqnumber$ particles of 
type $\tb$, then $H(\eta) > \Gamma\starred - \Db$.\\
{\rm (2)} If $\eta \neq \rho$ is a configuration with $\quasiqnumber$ particles of type $\tb$ 
such that $\tsupp{\eta} \neq \tsupp{\rho}$, then $H(\eta) > \Gamma\starred - \Db$.\\
{\rm (3)} If $\eta \neq \rho$ is a configuration with $\quasiqnumber$ particles of type $\tb$ 
obtained from $\rho$ by removing at least one of the ``non-corner'' particles of type 
$\ta$ in $\rho$ (note that $\eta$ and $\rho$ have the same dual tile support), then 
$H(\eta) > \Gamma\starred - \Db$.
\end{lemma}

\begin{proof}
Note that the standard
configurations with $\quasiqnumber$ particles of type
$\tb$ are unique modulo translations and rotations.

(1) Since $\eta$ is a $\btiled$ configuration, it follows from Lemma~2.3 in \cite{dHNT11} 
that $H(\eta)-H(\rho) = \tfrac14 [\cT(\eta)-\cT(\rho)] \Da$, because the energy difference 
between the two configurations only depends on the difference in the number of particles 
of type $\ta$. From Lemma~2.2 in \cite{dHNT11} it follows that $\cT(\eta) > \cT(\rho)$.
From the definition of $\cT$ in \eqref{Tdef} and Eq.\ (2.4) in \cite{dHNT11} we have that, 
for any $\btiled$ $\eta$, $\cT(\eta) = 4k$ for some $k\in\N$. Hence $\cT(\eta) - \cT(\rho) 
\ge 4$, and so $H(\eta) - H(\rho) \ge \Da$. The claim now follows by observing that
$H(\rho) = \Gamma\starred + \epsi - \Da - \Db$ and $\epsi>0$.

\medskip\noindent
(2) First consider the case where $\eta$ consists of a single cluster. Then there exists a 
configuration $\eta\prm$, obtained from $\eta$ by saturating all particles of type $\tb$, 
such that $H(\eta\prm) \le H(\eta)$ with equality if and only if $\eta\prm = \eta$. Clearly, 
$\tsupp{\eta} = \tsupp{\eta\prm}$. By part (1), 
we have
$H(\eta) \ge H(\eta\prm) > \Gamma\starred - \Db$. If $\eta$ consists of clusters $c_{1},
\ldots,c_{m}$ with $m \in \N\backslash\{1\}$, then observe that $H(\eta) = \sum_{i=1}^{m} 
c_{i}$. Let $\eta^{\nb(c_{i})}$ denote any standard configuration with $\nb(c_{i})$ particles 
of type $\tb$. By Lemmas~3.1 and 4.2 in \cite{dHNT11}, we have $H(\eta) = \sum_{i=1}^{k} 
H(c_{i}) \geq \sum_{i = 1}^{k} H(\eta^{\nb(c_{i})})$. Since $\rho$ is a standard configuration, 
it follows from Lemma~2.2 in \cite{dHNT11} that $\sum_{i=1}^{k} \cT(\eta^{\nb(c_{i})})>
\cT(\rho)$, and so, as in the proof of part (1), $\sum_{i=1}^{k} \cT(\eta^{\nb(c_{i})})
-\cT(\rho)>4$. Using (3.6) in \cite{dHNT11} for the energy of a standard configuration, we 
obtain that $\sum_{i = 1}^{k} H(\eta^{\nb(c_{i})}) - H(\rho) = \tfrac14[\sum_{i=1}^{k}
\cT(\eta^{\nb(c_{i})}) - \cT(\rho)] \Da$, from which we get the claim.

\medskip\noindent
(3) Let $m \in \N$ denote the number of non-corner particles of type $\ta$ removed from $\rho$ 
to obtain $\eta$. Then $H(\eta) \ge H(\rho) +m(2U - \Da) \ge H(\rho) + 2U - \Da$ (because
each of the non-corner particles of type $\ta$ in $\rho$ has at least 2 active bonds).
Substituting the value of $H(\rho)$ into the latter expression, we obtain $H(\eta) \ge 
\Gamma\starred - \Da - \Db + \epsi + 2U - \Da$. The claim follows by observing that 
$\Da < U$.
\end{proof}

\begin{lemma}
\label{lemma-entrance-of-set-protonumber}
{\rm (1)} All paths in $\optpaths$ enter the set $\nbis{\protonumber}$ via a configuration 
$(\hat{\eta},x)$ with $\hat{\eta} \in \nbis{\quasiqnumber}$ a quasi-standard 
configuration and $x \in \partial^{-}\Lambda$.\\
{\rm (2)} All paths in $\optpaths$ enter the set $\nbis{\critinumber}$ via a configuration 
$(\hat{\eta}, x)$ with $\hat{\eta} \in \minnbis{\protonumber}$ such that $\comlev
(\Box,\hat{\eta}) \le \Gamma\starred$, i.e., $\hat{\eta} \in \bar{g}(\{\Box\}, 
\minnbis{\protonumber})$, and $x \in \partial^{-}\Lambda$. Consequently, $\bar{g}(\{\Box\},
\minnbis{\protonumber})\boub$ is a gate for the transition $\Box \to \boxplus$.
\end{lemma}

\begin{proof}
(1) This is immediate from Lemma~\ref{lemma-quasi-square-is-optimal}.

\medskip\noindent
(2) By Theorem~1.5 in \cite{dHNT11} (which identifies $\Gamma\starred$) and Lemmas~3.1--3.2 
in \cite{dHNT11} (which determine the energy of configurations in $\minnbis{n}$ for all 
$n$), if $\eta \in \minnbis{\protonumber}$, then $H(\eta) = \Gamma\starred - \Db$. We argue
by contradiction. Suppose that $\omega \in \optpaths$ enters $\nbis{\critinumber}$ via a 
configuration $\zeta = (\hat\zeta, x)$ with $\zeta \in \nbis{\protonumber} \backslash 
\minnbis{\protonumber}$ and $x \in \partial^{-}\Lambda$. Then $H(\zeta) = H(\hat{\zeta}) 
+ \Db > \Gamma\starred$, because $H(\zeta) > \Gamma\starred - \Db$. Hence $\omega$ is not
optimal.
\end{proof}


\subsection{Characterization of $g(\{\Box\}, \minnbis{\protonumber})$ and 
$\bar{g}(\{\Box\}, \minnbis{\protonumber})$}
\label{Preplem2}

\begin{definition}
\label{def-set-standard-configuration}
{\rm (a)} For $n\in\N$, let $\standard{n}$ be the set of standard configurations with $n$ 
particles of type $\tb$.\\
{\rm (b)} Let $\omega\colon\Box\to\boxplus = (\Box, \ldots, \xi, \eta, \zeta, \ldots, 
\boxplus$). Write $P_{\omega}(\eta)$ to denote the part of $\omega$ from $\Box$ to 
$\xi$ and $S_{\omega}(\eta)$ to denote the part of $\omega$ from $\zeta$ to $\boxplus$. 
Any configuration in $P_{\omega}(\eta)$ is called a predecessor of $\eta$ in $\omega$, 
while any configuration in $S_{\omega}(\eta)$ is called a successor of $\eta$ in $\omega$. 
The configurations $\xi$ and $\zeta$ are called the immediate predecessor, respectively, 
the immediate successor of $\eta$ in $\omega$.
\end{definition}

\begin{lemma}
\label{lemma-reachable-from-standard}
{\rm (1)} For every $\zeta \in \bar{g}(\{\Box\},\minnbis{\protonumber})$ there is a 
standard configuration $\bar{\eta} \in \standard{\protonumber}$ such that $\zeta \in
\bar{g}(\{\bar{\eta}\},\minnbis{\protonumber})$.\\
{\rm (2)} For every $\zeta \in g(\{\Box\}, \minnbis{\protonumber})$ there is a 
standard configuration $\bar{\eta} \in \standard{\protonumber}$ such that $\zeta \in 
g(\{\bar{\eta}\},\minnbis{\protonumber})$. Consequently,
\begin{equation}
g(\{\Box\}, \minnbis{\protonumber})
= \bigcup_{\bar{\eta} \in \standard{\protonumber}} 
g(\{\bar{\eta}\},\minnbis{\protonumber}).
\end{equation}
\end{lemma}

\begin{proof}
(1) Pick $\zeta \in \bar{g}(\{\Box\},\minnbis{\protonumber})$. Let $\omega\colon\,\Box
\to\zeta$ be such that $\max_{\xi \in \omega} H(\xi) \le \Gamma\starred$. Let 
$\eta$ be the configuration visited by $\omega$ when it enters the set $\nbis{\protonumber}$ 
for the last time before visiting $\zeta$. Write $\omega$ as $\omega_{1}+\omega_{2}$, 
where $\omega_{1}$ is the part of $\omega$ from $\Box$ to $\eta$ and $\omega_{2}$ is 
the part of $\omega$ from $\eta$ to $\zeta$. By Lemma~\ref{lemma-quasi-square-is-optimal}, 
we have $\eta = (\hat{\eta}, x)$, where $\hat{\eta}$ is a quasi-standard configuration in 
$\nbis{\quasiqnumber}$ and $x \in \partial^{-}\Lambda$, otherwise $H(\eta) > \Gamma\starred$.
We will show that there is a standard configuration $\bar{\eta} \in \minnbis{\protonumber}$ 
and a path $\omega_{3}\colon\,\eta \to \bar{\eta}$ such that $H(\xi) \le H(\eta)$ and 
$\nb(\xi) = \protonumber$ for all $\xi \in \omega_{3}$. 
	
Let $\tilde{\eta}$ be the standard configuration in $\nbis{\quasiqnumber}$ with the same 
tile support as $\hat{\eta}$. This configuration exists because every quasi-standard 
configuration whose support lies in $\Lambdaminus$ has no particle of type $\tb$ in
$\partial^{-}\Lambdaminus$. (The latter is due to the fact that, in a quasi-standard 
configuration with $\quasiqnumber$ particles of type $\tb$, each site that is occupied 
by a particle of type $\tb$ has at least three neighboring sites occupied by a particle 
of type $\ta$, and all sites in $\partial^{-}\Lambdaminus$ have at most two adjacent sites 
in $\Lambdaminus$.) Let $\bar{\eta}$ the standard configuration in $\nbis{\protonumber}$ 
obtained from $\tilde{\eta}$ by adding a protuberance, with the particle of type $\tb$ in 
this protuberance located at a site $y\starred$ on one of the longest sides of the rectangular 
cluster of $\tilde{\eta}$. This is always possible because at least one of the longest 
sides of $[\tilde{\eta}]$ is far away from $\partial^-\Lambda$.   
	
Consider the path $S_{\omega_{2}}(\eta)$. Since $\eta$ is the configuration visited by 
$\omega$ when the set $\nbis{\protonumber}$ is entered for the last time before visiting 
$\zeta$, all configurations in $S_{\omega_{2}}(\eta)$ have at least $\protonumber$ particles 
of type $\tb$. In particular, the particle of type $\tb$ in $x$ cannot leave $\Lambda$. We 
refer to this particle as the ``floating particle''. Observe that $H(\eta) \ge \Gamma\starred 
+ \epsi - \Da$, with equality if and only if $\hat{\eta}$ is standard. This implies that 
only moves of the floating particle are allowed until it enters $\Lambdaminus$ (particles 
in $\partial^{-}\Lambda$ cannot have active bonds). Furthermore, since $L>2\ell
\starred$ (and hence the sides of $\hat{\eta}$ are smaller than the sides of $\Lambda$), it 
follows that all sites $y \in \Lambdaminus$ such that $y \notin \supp(\hat{\eta})$ are 
lattice-connecting. In particular, there exists a lattice path $\lambda = x_{0}, x_{1},
\ldots,x_{m}$ in $\Lambda$ for some $m \in \N$ with $x_{0}=x$ and $x_{m} = y\starred$. 
	
Let $\omega_{3}$ be the path from $\eta$ to $\bar{\eta}$ obtained by first letting the 
floating particle move along the lattice path $\lambda$ until it reaches site $y\starred$ 
and then saturating all the particles of type $\tb$ in $(\hat{\eta}, y\starred)$. Note 
that $H(\hat{\eta}, y\starred) \le H(\eta) - U$ and that all configurations in $\gamma_{3}$ 
have at least as many active bonds as $\eta$. Therefore $H(\xi) \le H(\eta)$ for all 
$\xi \in \omega_{3}$.	 
	
Let $\hat\omega_{3}$ denote the path from $\bar{\eta}$ to $\eta$ obtained by inverting 
$\omega_{3}$. Then, by construction, $\tilde\omega = \hat\omega_{3} + \omega_{2}$ is a 
path from $\bar{\eta}$ to $\zeta$ such that $H(\xi) \le \Gamma\starred$ and $\nb(\xi)
\ge \protonumber$ for all $\xi \in \tilde\omega$.

\medskip\noindent
(2) Same as part (1).
\end{proof}


\subsection{Characterization of $\cG(\Box,\boxplus)$}
\label{Preplem3}

In this section we want to characterize the essential gate $\cG(\Box,\boxplus)$ for the transition
$\Box \to \boxplus$.
By Lemma~\ref{lemma-mnos-essential-saddles} the set $\cG(\Box,\boxplus)$ coincides with
the set of essential saddles.
Remind here that a saddle is characterized only by 
its energy and not by its number of particles (see Def.~\ref{def3}(e)-(f)).
In the following lemma we show that all saddles with strictly less than $\critinumber$
particles of type $\tb$ can not be essential.

\begin{lemma}
\label{lemma-unessential-saddles}
{\rm (1)} All saddles in $\nbis{\le\quasiqnumber}$ are unessential.\\
{\rm (2)} Let $\zeta\in \nbis{\le \protonumber}$ be such that $H(\zeta) = \Gamma\starred$. 
Let $\cO(\zeta)= \{\omega \in \optpaths\colon\,\omega \ni \zeta\}$ be the set of optimal 
paths visiting $\zeta$. If all paths in $\cO(\zeta)$ visit $g(\{\Box\},\minnbis{\protonumber})$
after visiting $\zeta$, then $\zeta$ is unessential.
\end{lemma}

\begin{proof}
(1) Let $\zeta \in \nbis{\le\quasiqnumber}$ be a configuration such that $H(\zeta) = 
\Gamma\starred$. By Lemma~\ref{lemma-entrance-of-set-protonumber}(2), we have (recall 
Definition~\ref{def3}(e)) $S(\omega)\backslash \{\zeta\} \neq \emptyset$ for all 
$\omega \in \cO(\zeta)$, i.e., all paths in $\cO(\zeta)$ visit at least one other saddle 
configuration. Pick $\omega \in \optpaths$. By Lemma~\ref{lemma-entrance-of-set-protonumber}(1), 
$\omega$ (last) enters the set $\nbis{\protonumber}$ via a configuration $(\rho_{i},x)$ with 
$\rho_{i}$ a quasi-standard configuration obtained from the standard configuration $\rho_{0}
\in \nbis{\quasiqnumber}$ by removing $i$ corner particles of type $\ta$, and $x\in
\partial^-\Lambda$. It is clear that the configuration visited by $\omega$ just before 
$(\rho_{i}, x)$ is $\rho_{i}$. Write $\omega = \omega_{1} + \omega_{2}$, where $\omega_{1}$ 
is a path from $\Box$ to $\rho_{i}$ and $\omega_{2}$ is a path from $\rho_{i}$ to $\boxplus$. Obviously, $\zeta \in \omega_{1}$. Moreover, $\omega \in \optpaths$ implies that $H(\rho_{i},x) 
\le \Gamma\starred$ and, consequently, $H(\rho_{i}) \le \Gamma\starred - \Db$. Furthermore, 
$H(\rho_{j})< H(\rho_{i})$ for $j < i$. Let $\omega_{3}\colon\,\Box \to \rho_{0}$ be a path 
from $\Box$ to $\rho_{0}$ such that $H(\xi) < \Gamma\starred$ for all $\xi \in
\omega_{3}$ (e.g.\ follow the construction of the ``reference path'' in \cite{dHNT11}), 
and let $\omega_{4}\colon\,\rho_{0}\to\rho_{i}$ be the path obtained by, iteratively, 
detaching and moving out of $\Lambda$ one corner particle of type $\ta$ until configuration 
$\rho_{i}$ is reached. It is easy to see that $\max_{\xi \in \omega_{4}} H(\xi)<
\Gamma\starred$. Consider the path $\hat\omega = \omega_{3} + \omega_{4} + \omega_{2}$. By
construction, $\omega \in \optpaths$ and $S(\hat\omega) \subset S(\omega)\backslash\{\zeta\}$. 
Finally, observe that the same argument holds for any $\omega \in \cO(\zeta)$.

\medskip\noindent
(2) Since $g(\{\bar{\eta}\}, \minnbis{\protonumber}) \subset \nbis{\protonumber}$, all optimal 
paths from $g(\{\bar{\eta}\}, \minnbis{\protonumber})$ to $\boxplus$ must visit the set
$\optentrance{\nbis{\critinumber}}$. Therefore $S(\omega)\backslash \{\zeta\} \neq \emptyset$ 
for all $\omega \in \cO(\zeta)$, since all configurations in $\optentrance{\nbis{\critinumber}}$ 
have energy $\Gamma\starred$. Pick $\omega \in \cO(\zeta)$, and let $\eta$ be the first 
configuration in $g(\{\bar{\eta}\},\minnbis{\protonumber})$ visited by $\omega$ after 
visiting $\zeta$. Let $\omega_{1}$ be the part of $\omega$ from $\Box$ to $\eta$ and 
$\omega_{2}$ the part of $\omega$ from $\eta$ to $\boxplus$. Since $\eta \in g(\{\bar{\eta}\},
\minnbis{\protonumber})$, there is a path $\omega_{3}$ from $\Box$ to $\eta$ such that 
$H(\sigma) < \Gamma\starred$ for all $\sigma\in \omega_{3}$. Let $\hat\omega = \omega_{3} 
+ \omega_{2}$. By construction, $\hat\omega \in \optpaths$ and $S(\hat\omega) \subseteq 
S(\omega) \backslash \{\zeta\}$. Finally, observe that the same argument holds for all 
$\omega \in \cO(\zeta)$.
\end{proof}

\begin{lemma}
\label{lemma-entgate-subset-ent-critinumber}
If $\gbar = \g$, then $\entgate \subset \g\boub$ and $\proto \subset \g$.	
\end{lemma}

\begin{proof}
By Lemma~\ref{lemma-entrance-of-set-protonumber}(2), if $\gbar = \g$, then $\g\boub$ is a 
gate for the transition $\Box \to \boxplus$. Since $\g\boub$ is a gate, there exists a 
$\cW \subset \g\boub$ that is a minimal gate for the transition $\Box \to \boxplus$. 
Let $\omega \in \optpaths$. Since $\cW \subset \cG(\Box, \boxplus)$ and $\omega \cap 
\cW \neq \emptyset$, it follows that $\omega \cap \cG(\Box, \boxplus) \cap \g\boub 
\neq \emptyset$. Combining Lemmas~\ref{lemma-entrance-of-set-protonumber} and 
\ref{lemma-unessential-saddles}, it follows that if $\gbar = \g$, then all saddles in 
$\nbis{\le \protonumber}$ are unessential and, by Lemma~\ref{lemma-mnos-essential-saddles},
do not belong to $\cG(\Box,\boxplus)$. Therefore the first configuration in $\cG(\Box,
\boxplus)$ visited by $\omega$ is an element of $\g\boub$. Since the choice of $\omega
\in \optpaths$ is arbitrary, we conclude that $\entgate \subset \g\boub$.
	
It remains to show that $\proto \subset \g$. The proof is by contradiction. Pick $\hat{\eta}
\in \proto \backslash \g$. Since $\entgate \subset \g\boub$, there is a configuration $\eta
\in \g\boub$ obtained in a single step from $\hat{\eta}$. Clearly, $\hat{\eta} \in 
\nbis{\ge \protonumber}$, since the number of particle of type $\tb$ in $\Lambda$ changes 
at most by one at each step. Since, by Lemma~\ref{lemma-entrance-of-set-protonumber} and 
the hypothesis $\gbar = \g$, all paths in $\optpaths$ enter $\nbis{\critinumber}$ by adding 
a particle of type $\tb$ in $\partial^{-}\Lambda$ to a configuration in $\g$, it follows that
$\hat{\eta} \in \nbis{\ge \critinumber}$ by assumption.
\end{proof}


\section{Motion of $\btiles$}
\label{tilemotion}

In Section~\ref{sec-moving-dimers} we study the motion of $\btiles$. In 
Section~\ref{sec-heavysteps} we derive some restrictions on the transitions 
between configurations with different tile support. In Section~\ref{sec-small-lc} 
we identify the critical droplets for small values of $\ell\starred$, namely, 
$\ell\starred=2,3$.   


\subsection{Motion of dimers of $\btiles$}
\label{sec-moving-dimers}

\begin{definition}
\label{def-various}
{\rm (a)}
Two configurations $\eta$ and $\eta\prm$ (with the same tile support) are called equivalent 
if there is a path $\omega\colon\,\eta \to \eta\prm$ (possibly of length zero) such that all configuration in $\omega$ have the same tile support and $\comlev(\eta,\eta\prm)<\Gamma
\starred$. In words, two configurations are equivalent if it is possible to go from one to 
the other via a sequence of moves of particles of type $\ta$ without reaching energy level 
$\Gamma \starred$.\\ 
{\rm (b)}
A heavy-step is a sequence of moves realizing the transition between two configurations 
$\eta$ and $\eta\prm$ with different tile support. Note that a heavy-step is completed 
by moving, removing or adding a particle of type $\tb$ in $\Lambda$.
\end{definition}

Let $\eta \in \nbis{n_{2}}$ and $\bar{\eta} \in \minnbis{n_{2}}$. By \cite{dHNT11}, 
Lemma~4.1 and the proof of Lemma~3.1, both $B(\bar{\eta})$ and $\na(\bar{\eta})$ are 
constant in $\minnbis{n_{2}}$.

\begin{definition}
{\rm (a)}
A configuration $\eta \in \nbis{n_{2}}$ is said to have $m$ broken bonds if $B(\eta)
=B(\bar{\eta})-m$ for all $\bar{\eta} \in \minnbis{n_{2}}$. The number of broken bonds 
in configuration $\eta$ is denoted by $\broken{\eta}$.\\
{\rm (b)}
A configuration $\eta \in \nbis{n_{2}}$ is said to have $n$ extra particles of type $\ta$ 
if $\na(\eta)=\na(\bar{\eta})+n$ for all $\bar{\eta} \in \minnbis{n_{2}}$. The number of 
extra particles of type $\ta$ in configuration $\eta$ is denoted by $\extrapart{\eta}$.\\
{\rm (c)}
$B(p,\eta)$ denotes the number of active bonds adjacent to particle $p$ in configuration 
$\eta$.\\
{\rm (d)}
A dimer consists of two adjacent particles of different type such that the particle of 
type $\ta$ is lattice-connecting and has only one active bond (i.e., is a corner particle 
of type $\ta$). The particle of type $\tb$ belonging to a dimer is called a corner particle 
of type $\tb$.\\
{\rm (e)}
A particle of type $\tb$ in a $\btiled$ configuration $\eta$ is called external if it 
can be moved without moving any other particle of type $\tb$ in $\eta$ (see 
Fig.~{\rm \ref{fig:paradigm_class_A}}).
\end{definition}
	
In this section we will exhibit two methods to move a dimer in a configuration $\eta$ to 
a good dual corner of the cluster it belongs to (see Fig.~\ref{fig:dimer-moved}). The 
configuration $\eta\prm$ that is obtained in this way satisfies $H(\eta\prm) \le H(\eta)$. 
In particular, we will exhibit two different choices for a path $\omega$ from $\eta$ as 
in Fig.~\ref{fig:dimer-moved}(a) to $\eta\prm$ as in Fig.~\ref{fig:dimer-moved}(b), and 
we will determine $\max_{\xi \in \omega} H(\xi)$. In what follows we write $\D H(\omega) 
= H(\tilde{\eta}) - H(\eta)$, where $\tilde{\eta}$ is the configuration that is reached 
after the last step in $\omega$.

\begin{figure}[htbp]
\centering
\subfigure[]
{\includegraphics[height=5\tilesize]{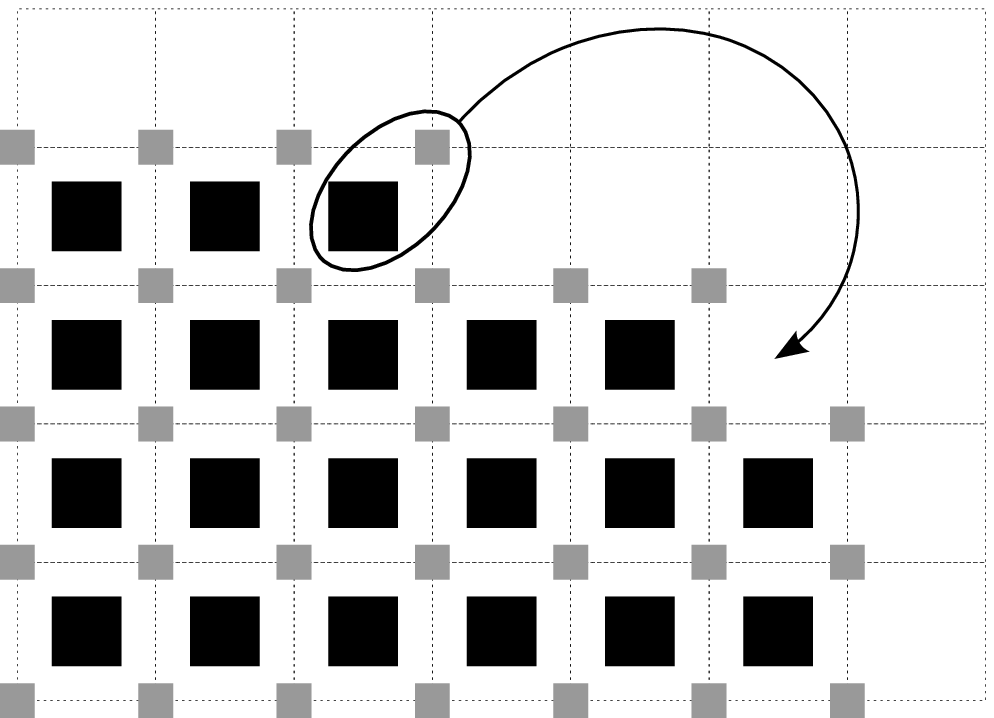}}
\qquad
\subfigure[]
{\includegraphics[height=5\tilesize]{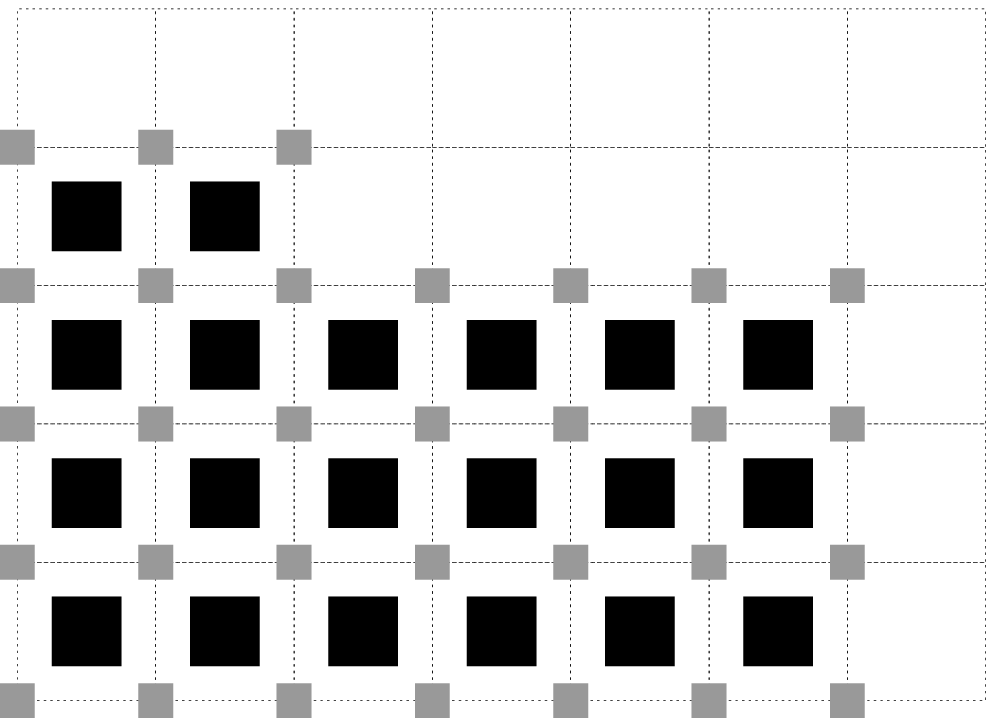}}
\caption{Motion of a $\btile$.} 
\label{fig:dimer-moved}
\end{figure}

\begin{lemma}
A $\btile$ can be moved within energy barrier $3U$ and $U+4\Da$.
\end{lemma}

\begin{proof}
We will give examples that are paradigmatic for the general case. Let $p$, $q$ denote, 
respectively, the particle of type $\tb$ and of type $\ta$ of the dimer that we want to 
move.

\medskip\noindent
{\bf 1.}
The first method is achieved within energy barrier $3U$ energy (i.e., $\max_{\xi\in\omega}
H(\xi)=3U$) and goes as follows (see Fig.~\ref{fig:dimer-3U}). First, particle $q$ is moved 
one step North-East ($\D H(\omega) = U$). Next, also particle $p$ is moved one step North-East 
($\D H(\omega) = 3U$; see Fig.~\ref{fig:dimer-3U}(a)). After that, particle $p$ is moved one 
step South-East ($\D H(\omega) = 2U$), and particle $q$ is moved in two steps to the site at 
dual distance $2\sqrt{2}$ from particle $p$ in the North-East direction ($\D H(\omega) = 2U$; 
see Fig.~\ref{fig:dimer-3U}(b)). It is possible to continue following a pattern of this type 
until particle $p$ is adjacent only to the (original) corner particle of type $\ta$ at the 
end of the bar ``just below'' $p$ ($\D H(\omega) = 3U$; see Fig.~\ref{fig:dimer-3U}(c)). Call 
$\eta_{1}$ the configuration reached after this last step. Particle $q$ can now be moved to 
the site at dual distance $2\sqrt{2}$ from particle $p$ in the South-East direction ($\D H
(\omega) = 3U$; see Fig.~\ref{fig:dimer-3U}(c)). Call this configuration $\eta_{2}$. Move 
particle $p$ first one step South-East ($\D H(\omega) = 3U$) and then one step South-West 
($\D H(\omega) = U$). Finally, move particle $q$ to the free site adjacent to $p$ ($\D H
(\omega) = 0$).

\begin{remark}
Note that from $\eta_{2}$ to $\eta\prm$ particle $p$ moves in the South direction via the 
same mechanism that was used to move in the East direction from $\eta$ to $\eta_{1}$. This 
symmetry in the motion of the dimer around a corner of the cluster will be used also in 
the sequel.
\end{remark}
 
\begin{figure}[htbp]
\centering
\subfigure[]
{\includegraphics[height=5\tilesize]{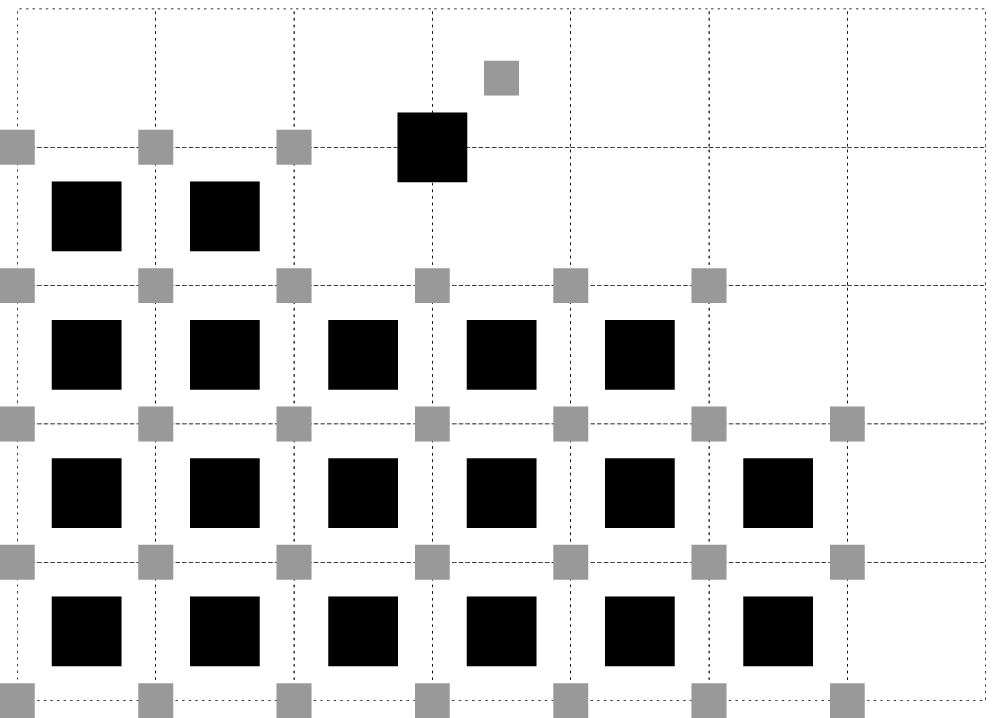}}
\qquad
\subfigure[]
{\includegraphics[height=5\tilesize]{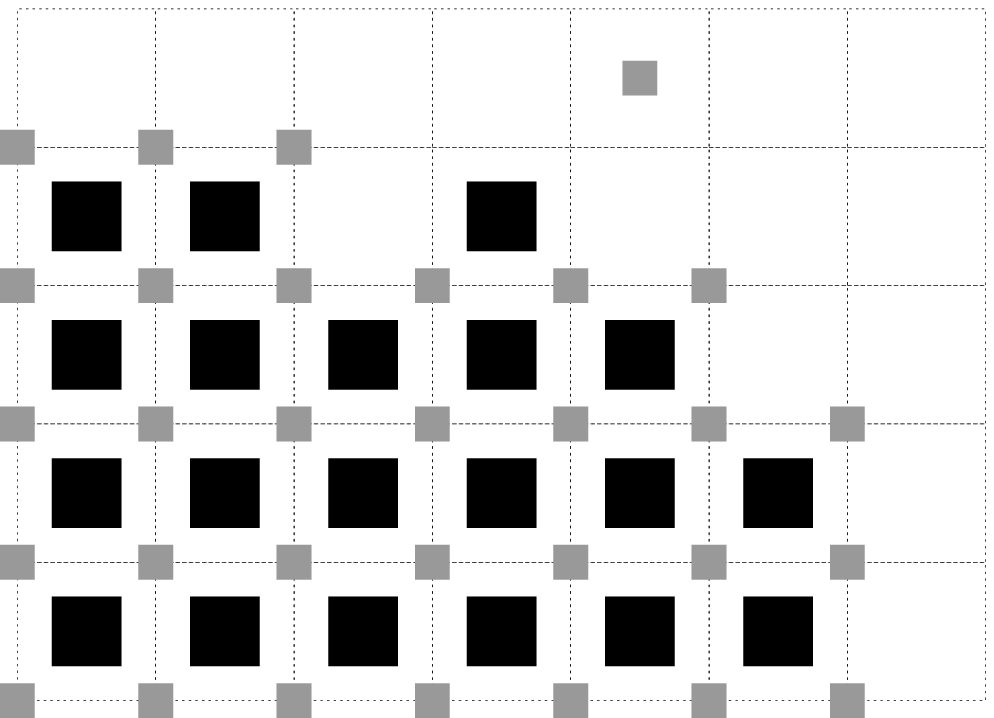}}
\qquad
\subfigure[]
{\includegraphics[height=5\tilesize]{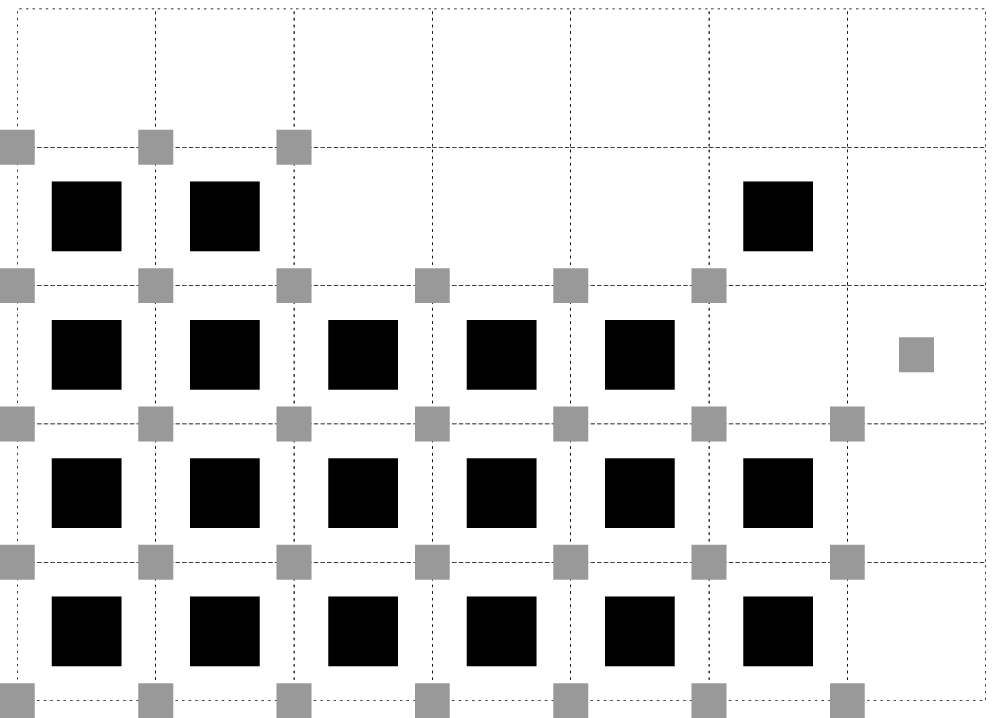}}
\caption{A dimer is moved to a corner within energy barrier $3U$.} 
\label{fig:dimer-3U}
\end{figure}

\medskip\noindent
{\bf 2.} The second method is achieved within energy barrier $U + 4\Da$ (i.e., 
$\max_{\xi\in\omega} H(\xi)=U + 4\Da$) and goes as follows (see Fig.~\ref{fig:dimer-U4Da}). 
First, move particle $q$ one step in the North-East direction ($\D H(\omega) = U$), and 
let two extra particles of type $\ta$ enter $\Lambda$ and reach the two sites at dual 
distance $1$ from $q$ in the West and the South direction ($\D H(\omega) = U + 2\Da$). 
Next, move particle $p$ one step in the North-East direction ($\D H(\omega) = U + 2\Da$, 
see Fig.~\ref{fig:dimer-U4Da}(a)). After that, move the particle of type $\ta$ at dual 
distance $1$ in the West direction from $p$ to the site adjacent to $p$ in the South-West 
direction ($\D H(\omega) = U+2\Da$), let one extra particle of type $\ta$ enter $\Lambda$ 
($\D H(\omega) = 2+3\Da$), and move this particle to the site at dual distance $1$ in 
the West direction from $p$ ($\D H(\omega) = 3\Da$, see Fig.~\ref{fig:dimer-U4Da}(b)).

\begin{figure}[htbp]
\centering
\subfigure[]
{\includegraphics[height=5\tilesize]{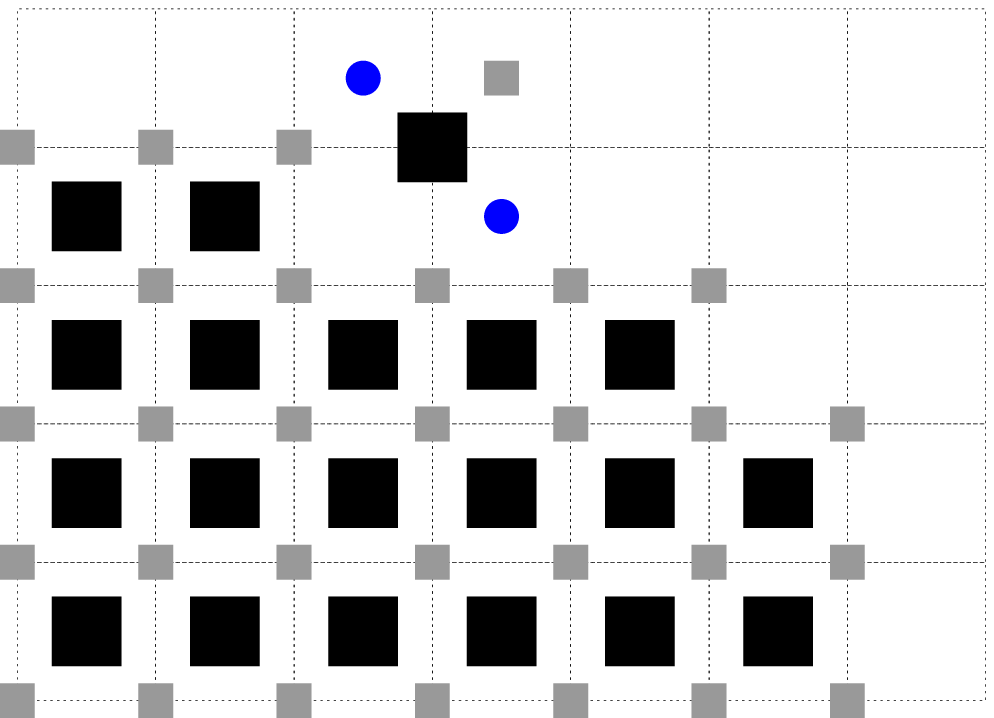}}
\qquad
\subfigure[]
{\includegraphics[height=5\tilesize]{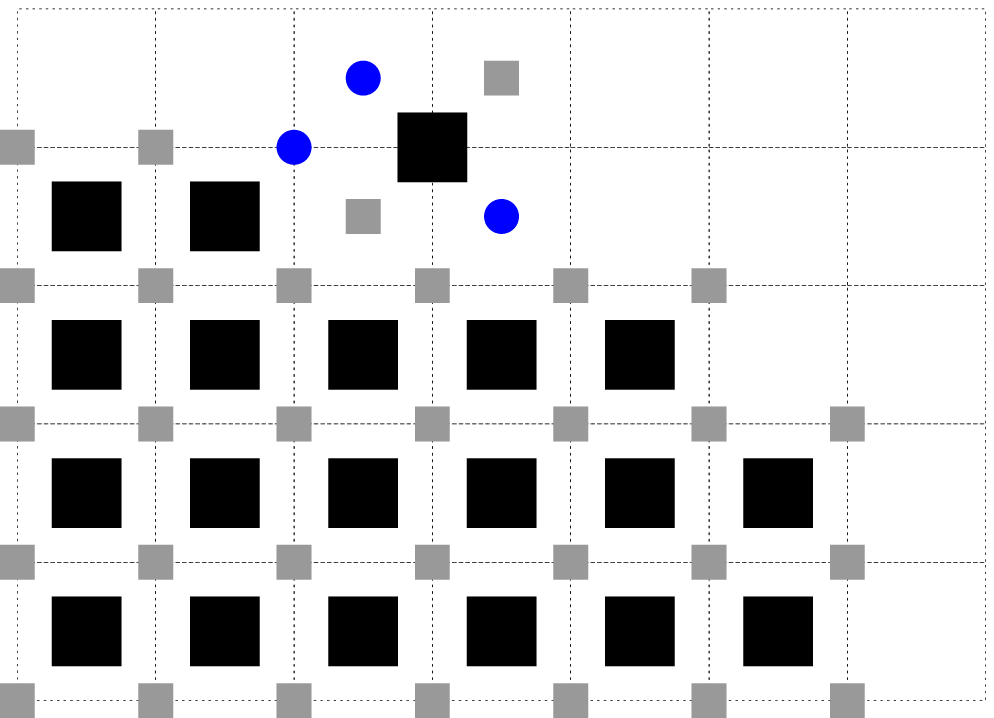}}
\qquad
\subfigure[]
{\includegraphics[height=5\tilesize]{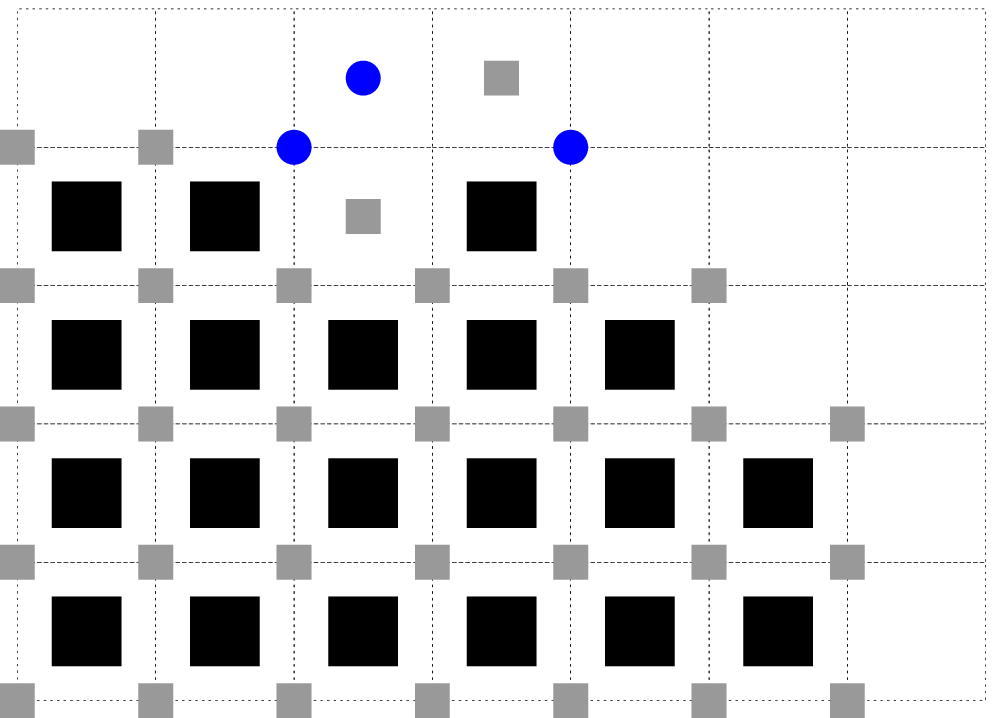}}
\\
\subfigure[]
{\includegraphics[height=5\tilesize]{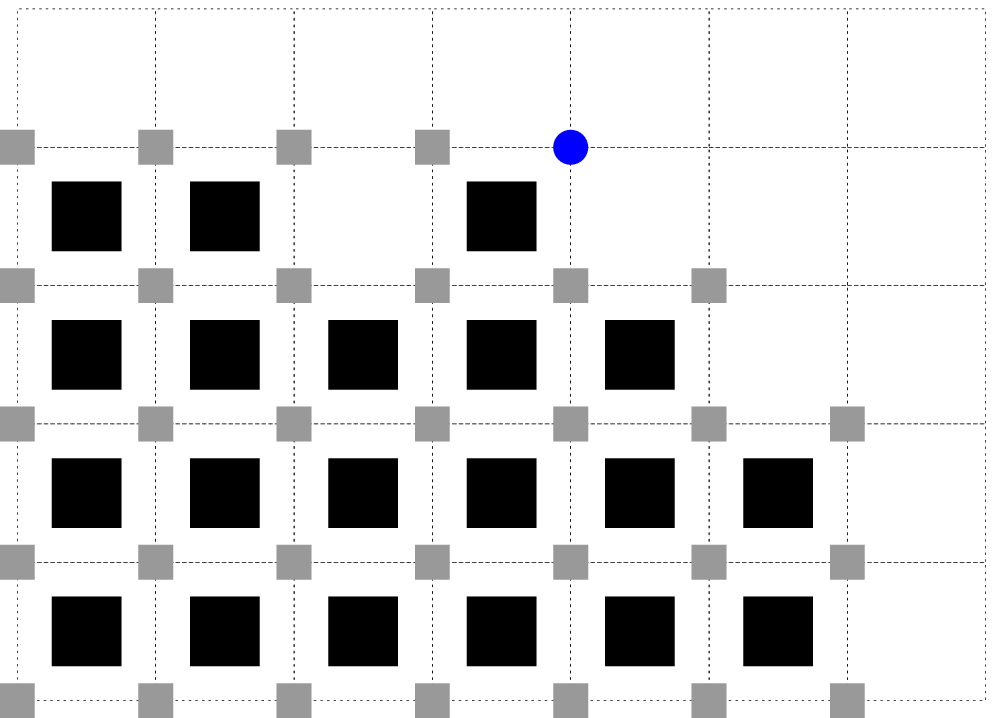}}
\qquad
\subfigure[]
{\includegraphics[height=5\tilesize]{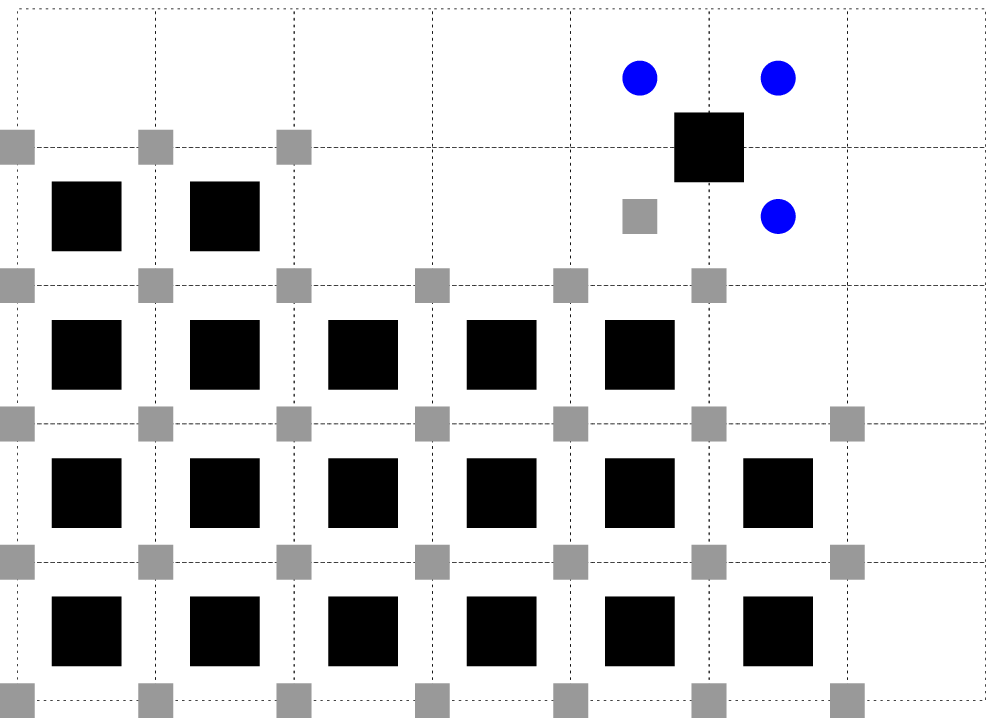}}
\qquad
\subfigure[]
{\includegraphics[height=5\tilesize]{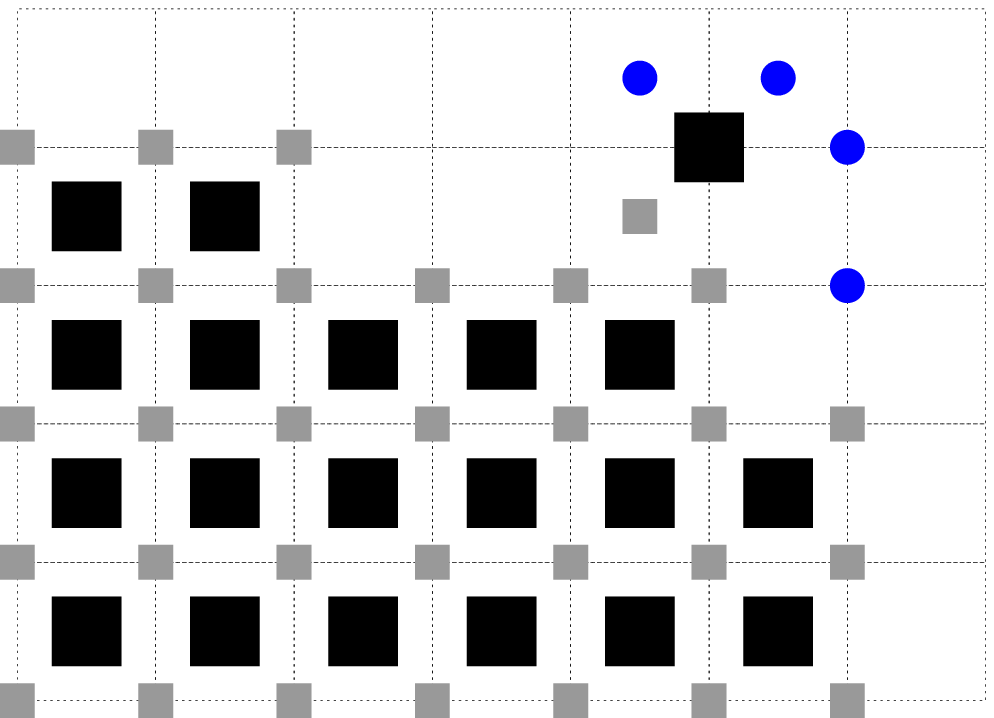}}
\caption{A dimer is moved to a corner within energy barrier $U + 4\Da$.
The small circles represent the extra particles of type $\ta$ with 
respect to those in the starting configuration.} 
\label{fig:dimer-U4Da}
\end{figure}

Move the particle adjacent to $p$ in the South-East direction one step in the North-East 
direction ($\D H(\omega) = U + 3\Da$). Move $p$ one step in the South-East direction 
($\D H(\omega) = U + 2\Da$, see Fig.~\ref{fig:dimer-U4Da}(c)). Afterwards, use one of 
the free particles of type $\ta$ to saturate $p$ ($\D H(\omega) = + 3\Da$), and remove 
the other free particles of type $\ta$ from $\Lambda$ ($\D H(\omega) = 2\Da$, see 
Fig.~\ref{fig:dimer-U4Da}(d)). The same procedure can be repeated until the configuration 
in Fig.~\ref{fig:dimer-U4Da}(e) is reached ($\D H(\omega) = 3\Da$). Next, let a particle 
enter $\Lambda$ and reach the site at dual distance $1$ in the East direction from $p$, 
and move one step in the South-East direction the particle adjacent to $p$ in the South-East 
direction ($\D H(\omega) = U + 4\Da$, see Fig.~\ref{fig:dimer-U4Da}(f)). Next, move $p$ in 
the Sout-East direction ($\D H(\omega) = U + 4\Da$), saturate it with one of the free 
particles of type $\ta$, and remove the other particles of type $\ta$ from $\Lambda$ 
($\D H(\omega) = 2\Da$). Now particle $p$ can be moved in the South direction in the 
same way it was moved in the East direction within energy barrier $U+3\Da$.
\end{proof}

By Lemma~\ref{lemma-entrance-of-set-protonumber}, we know that any path in $\optpaths$ 
enters the set $\nbis{\critinumber}$ when a particle of type $\tb$ is added in 
$\partial^{-}\Lambda$ to a configuration in $\bar{g}(\Box, \minnbis{\protonumber})$. 
By Lemma~\ref{lemma-reachable-from-standard}, we know that $\bar{g}(\Box,\minnbis{
\protonumber})$ can be determined by looking at all the configurations that can be 
reached starting from the standard configurations in $\nbis{\protonumber}$ without 
changing the number of particles of type $\tb$ in $\Lambda$ and taking into account 
all the moves that are possible within energy barrier $\Db$. Note that different 
configurations in $\minnbis{\protonumber}$ necessarily have different tile support, 
and so to move between these classes of configurations it is necessary to perform a 
sequence of heavy-steps.

\begin{remark}
\label{rem-da-halfU-transition}
Note that, from the point of view of the maximal energy barrier that needs to be overcome 
to go from $\eta$ to $\eta\prm$, for $\Da > \tfrac12 U$ the first method is more efficient
while fOr $\Da < \tfrac12 U$ the second method is more efficient.
\end{remark}

\begin{remark}
\label{rem-visit-all-monotone}
Starting from a $\btiled$ configuration with a monotone dual support inscribed in a 
rectangle of side lengths $l_{1},l_{2}$ (``far enough'' from the boundary of $\Lambda$), 
it is possible to reach, via one of the two mechanisms described above, all the $\btiled$ configurations with a monotone dual support and the same circumscribing rectangle.  
\end{remark}


\subsection{Restriction on heavy-steps}
\label{sec-heavysteps}

\begin{lemma}
\label{lemma-first-2step-from-corner}
Let $\Db < 3U + \Da$ and $\eta \in \minnbis{\protonumber}$. If the first heavy-step 
starting from $\eta$ (that does not change the number of particles of type $\tb$)
does not result in the motion of a corner particles of type $\tb$ along the edge 
where in $\eta$ it shares a bond with a corner particle of type $\ta$, then it is 
not possible to reach a new configuration $\eta\prm \in \minnbis{\protonumber}$ 
within energy barrier $\Db$.
\end{lemma}

\begin{figure}[htbp]
\centering
{\includegraphics[height=32mm]{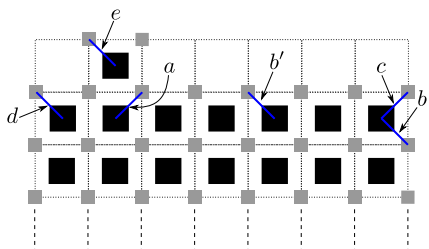}}
\caption{Types of edges for the first possible heavy-step starting from a standard 
configuration.}
\label{fig:region_A_classes_of_2-steps}
\end{figure}

\begin{proof}
It is clear that the first heavy-step can only involve one of the external particles 
of type $\tb$. For the proof we refer to Fig.~\ref{fig:region_A_classes_of_2-steps}, 
where a prototype configuration $\bar{\eta} \in \minnbis{\protonumber}$ is represented. 
	
We will show that if a heavy-step is performed along one of the edges $a$, $b$ or $b\prm$, 
then it is not possible to reach a new configuration $\eta\prm \in \minnbis{\protonumber}$
without exceeding energy barrier $\Db$. These edges are representatives of the possible
types of edges along which the first heavy-step is possible without involving the motion
of a corner particle of type $\tb$ along the edge where it shares a bond with a corner 
particle of type $\ta$.

\medskip\noindent
{\bf 1.}	
Assume that the first heavy-step is along edge $a$. Let $(u,v)$ denote the dual coordinates 
of the particle $p_{1}$ of type $\tb$ we want to move. Let $\eta_{1} \notin 
\minnbis{\protonumber}$ be the configuration that is reached when $p_{1}$ is moved, 
along edge $a$, to the site with dual coordinates $(u+\tfrac{1}{2}, v+\tfrac{1}{2})$,
and let $\eta_{0}$ be the configuration visited just before the heavy-step is performed.
When $\eta_{1}$ is reached, either the site with dual coordinates $(u+1,v+1)$ is empty 
or it is occupied by a particle of type $\ta$. Without loss of generality, we may assume 
that $\eta_{1}$ does not contain free particles of type $\ta$, since these particles 
could be iteratively removed from $\Lambda$ while decreasing the energy of the configuration.
Similarly, we may assume that all particles of type $\ta$ that do not interfere with the 
heavy-step that is performed are still in $\Lambda$, since in $\bar{\eta}$ they had at 
least one active bond and thier removal would increase the energy of the configuration
(since $\Da < U$). In the former case, $B(\eta_{1}) = B(\bar{\eta}) - 5$, and so 
$H(\eta_{1}) - H(\bar{\eta}) = 5U - \Da > \Db$. In the latter case, $B(\eta_{1}) 
= B(\bar{\eta}) - 4$, and so $H(\eta_{1}) - H(\bar{\eta}) = 4U > \Db$.

\medskip\noindent
{\bf 2.}
Assume that the first heavy-step is along edge $b$ (see Fig.~\ref{fig:region_A_2-step_b1}(a)).
Again, let $w = (u,v)$ denote the dual coordinates of the particle $p_{1}$ of type $\tb$
we want to move. Denote by $q_{1}$ and $q_{2}$ the particles of type $\ta$ sitting in 
$\bar{\eta}$, respectively, at the sites with dual coordinates $(u+\tfrac{1}{2},v-\tfrac{1}{2})$ 
and $(u+\tfrac{1}{2}, v-\tfrac{1}{2})$. Let $\eta_{1} \notin \minnbis{\protonumber}$ be the
configuration obtained by moving particle $p_{1}$ to the site $x$ with dual coordinates
$(u+\tfrac{1}{2}, v-\tfrac{1}{2})$, and let $\eta_{0}$ be the configuration visited just 
before the heavy-step is performed. When particles $p_{1}$ reaches site $x$, it has at most 
two active bonds, depending on whether the dual sites $y_{1}=(u+1,v-1)$ and $y-{2}=(u+1,v)$ 
are empty or occupied by a particle of type $\ta$. If both $y_{1}$ and $y_{2}$ are empty, 
then $H(\eta_{1}) - H(\bar{\eta}) = 5U  - 2\Da > \Db$ (again we assume that $\Lambda$ does 
not contain free particles of type $\ta$ and all other particles of type $\tb$ are saturated).
If only one dual site between $y_{1}$ and $y_{2}$ is occupied, then, arguing as before, we
get $H(\eta_{1}) - H(\bar{\eta}) = 4U > \Db$ if particle $q_{2}$ is still inside $\Lambda$ 
and $H(\eta_{1}) - H(\bar{\eta}) = 4U -\Da > \Db$ if particle $q_{1}$ has been removed from 
the $\Lambda$. If both $y_{1}$ and $y_{2}$ are occupied, then $H(\eta_{1}) - H(\bar{\eta}) 
= 3U$ (again we assume that there are no free particles of type $\ta$ in $\Lambda$). Note 
that, since $\Db < 3U + \Da$, no particle of type $\ta$ is allowed to enter $\Lambda$ nor
is it allowed to break any active bond. Therefore the only moves that are possible starting 
from $\eta_{1}$ are those that do not increase the energy. Only two moves are possible. Either 
the particle of type $\tb$ is moved back from site $x$ to site $w$, or the particle $q_{3}$ 
of type $\ta$ sitting at the dual site $y_{3}=(u-\tfrac{1}{2}, v+\tfrac{1}{2})$ is moved 
to site $x$. In the former case, it is clear that the configuration that is reached is
$\eta_{0}$ and the move produces a non-self-avoiding path. In the latter case, we reach 
a configuration $\eta_{2} \notin \minnbis{\protonumber}$ with the same energy as $\eta_{1}$. 
But now the only move that does not increase the energy is the motion of $q_{3}$ back to 
$y_{3}$, which again produces a non-self-avoiding path.

\begin{figure}[htbp]
\centering
\subfigure[]{\includegraphics[height=3\tilesize]{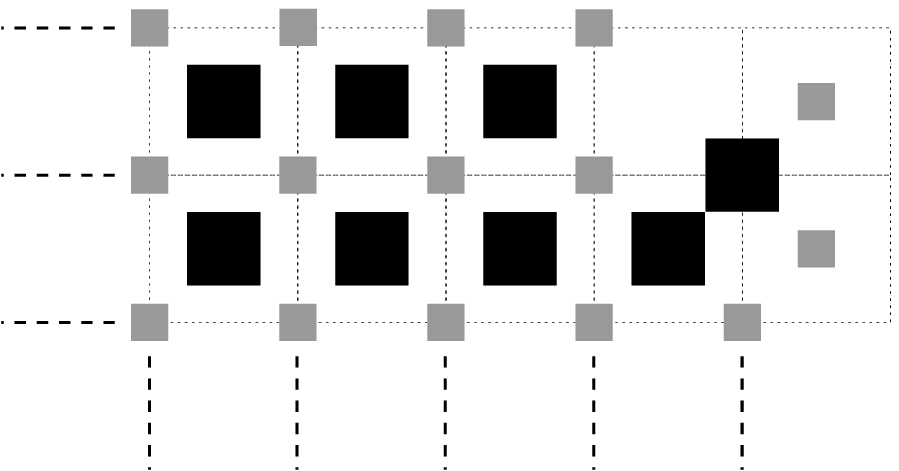}}
\qquad
\subfigure[] {\includegraphics[height=4\tilesize]{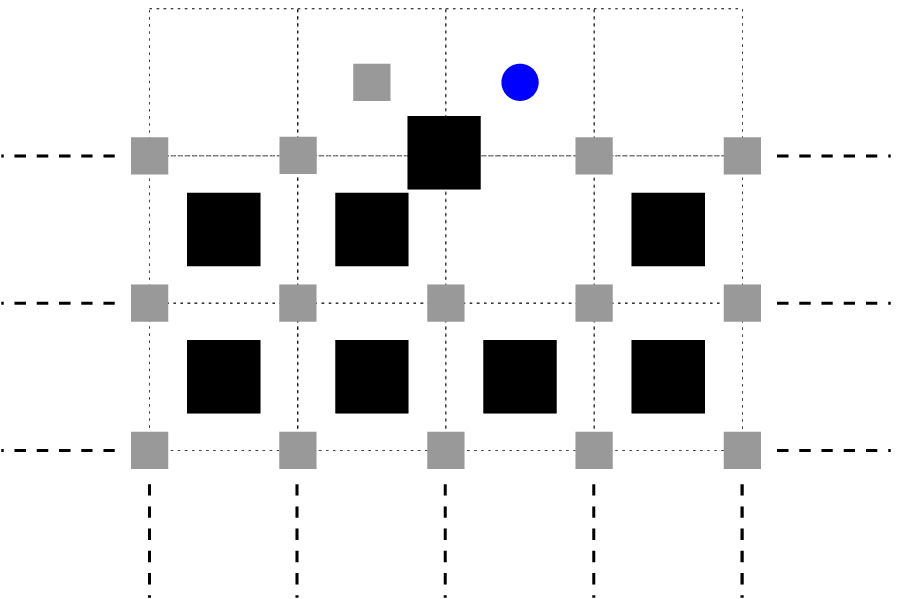}}
\caption{Two choices for the first heavy-step that are too costly.}
\label{fig:region_A_2-step_b1}
\end{figure}	

\medskip\noindent	
{\bf 3.} 
The case where the first heavy--step is performed along edge $b\prm$ is similar to 
the previous case (see Fig.~\ref{fig:region_A_2-step_b1}(b)).
\end{proof}

We see from Lemma~\ref{lemma-unessential-saddles} that $\gbar\boub$ is a good candidate 
for $\entgate$.


\subsection{Small values of $\ell\starred$}
\label{sec-small-lc}

In Section~\ref{Proofs} we will identify the geometry of the protocritical and critical
configurations for $\ell\starred \ge 4$, i.e., for the subregion 
$\Db > 4U
-\tfrac{4}{3}\Da$. The analysis will show that the set $\g = \gbar$ consists of all 
the configurations in $\minnbis{\protonumber}$. 
The region $\Db \le 4U -\tfrac{4}{3}\Da$ is denoted by $\RT$.

For $\ell\starred = 2$ this set consists of those $\btiled$ configurations whose dual 
tile support is either a $2\times2$ square from which a corner has been removed or a 
$3 \times 1$ rectangle. For $\ell\starred = 3$ it consists of those $\btiled$ configurations 
whose dual tile support is either a $3\times2$ rectangle plus a ``protuberance'' on 
one of the four sides or a $3 \times 3$ square from which two corners have been removed.
This can be easily verified by noting that it is possible to move a tile protuberance 
within energy barrier $4U - 2\Da < \Db$ and that it is possible to ``slide'' an external 
$\abbar$ of length $2$ within energy barrier $2U + \Da < \Db$, as described next (see 
Fig.~\ref{fig:sliding-bar-2}).

\begin{figure}[htbp]
\centering
\subfigure[]{\includegraphics[height=4\tilesize]{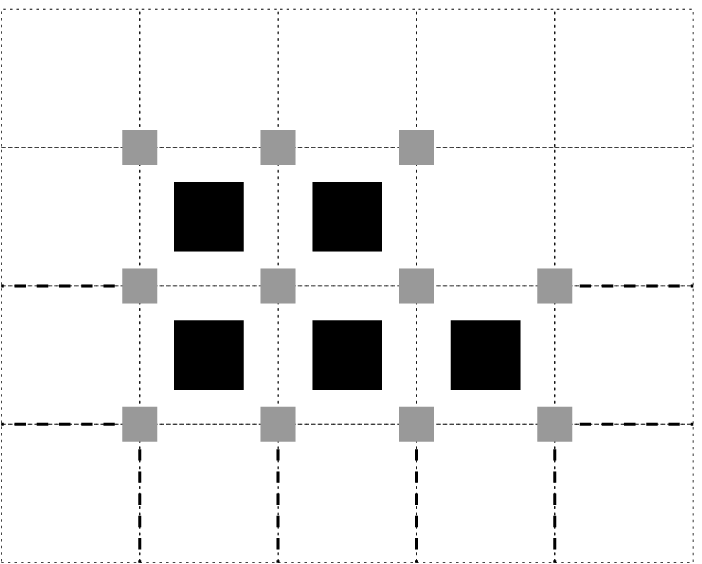}}
\qquad
\subfigure[]{\includegraphics[height=4\tilesize]{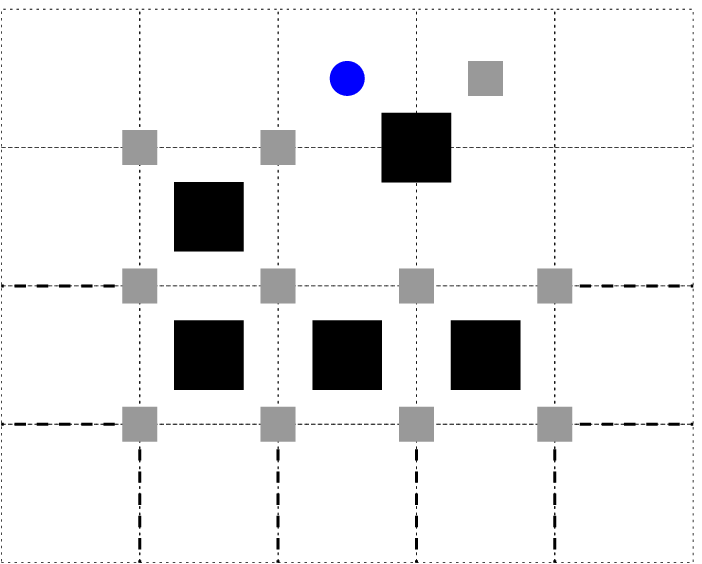}}
\qquad
\subfigure[]{\includegraphics[height=4\tilesize]{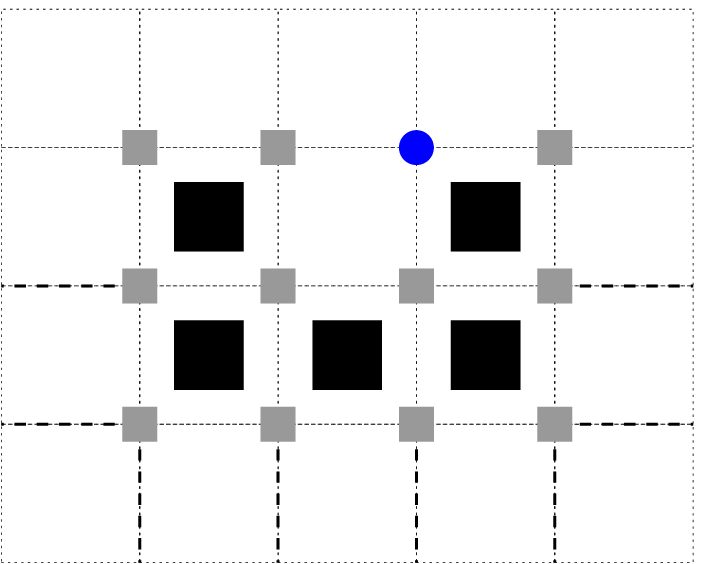}}
\qquad
\subfigure[]{\includegraphics[height=4\tilesize]{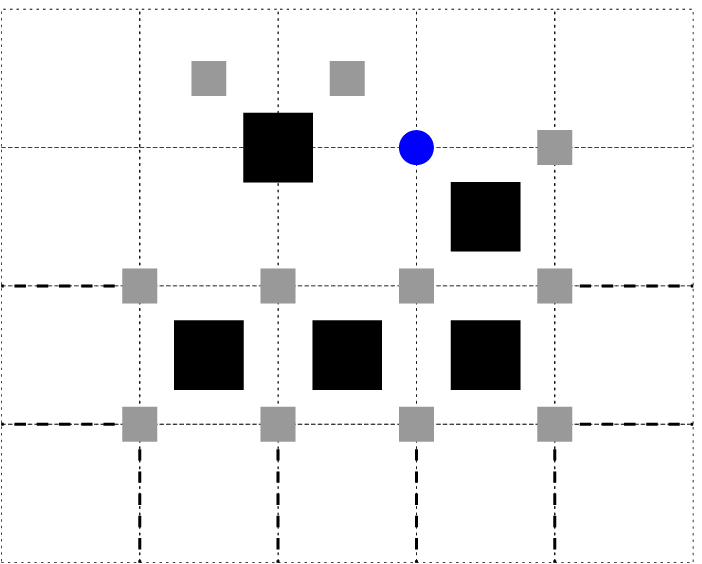}}
\caption{Sliding a $\abbar$ of length $2$.}
\label{fig:sliding-bar-2}	
\end{figure}

We construct a path $\omega$ from configuration $\eta$ of Fig.~\ref{fig:sliding-bar-2}(a) 
to configuration $\eta\prm$ obtained by shifting the $\abbar$ of length $2$ in the East 
direction by one dual unit. Let $p_{1}$ denote the Eastern-most particle of type $\tb$ of 
the $\abbar$ and $p_{2}$ the other particle of type $\tb$. Let $q_{1}$ be the particle of 
type $\ta$ adjacent to $p_{1}$ in the North-East direction, $q_{2}$ the particle of type 
$\ta$ adjacent to $p_{1}$ in the North-West direction, and $q_{3}$ the particle of type 
$\ta$ adjacent to $p_{2}$ in the North-West direction. Move $q_{1}$ one step North-East 
($\D H(\omega) = U$), let an extra particle $q_{4}$ of type $\ta$ enter $\Lambda$ 
($\D H(\omega) = U + \Da$), and let this particle reach the site at dual distance $1$ 
in the West direction from $q_{1}$, and move $p_{1}$ one step North-East ($\D H(\omega) 
= 2U + \Da$; see Fig.\ref{fig:sliding-bar-2}(b)). Then, without increasing the energy of 
the configuration, move $p_{2}$, $q_{1}$ and $q_{4}$ subsequently one step South-East 
($\D H (\omega) = \Da$; see Fig.~\ref{fig:sliding-bar-2}(c)). Afterwards, move first 
$q_{2}$ and $q_{3}$ one step North-East ($\D H (\omega) = 2U + \Da$) and $p_{2}$ one 
step North-East ($\D H (\omega) = 2U + \Da$; see Fig.~\ref{fig:sliding-bar-2}(d)). Finally, 
move $p_{2}$ one step South-East, use $q_{2}$ to saturate $p_{2}$, and remove $q_{3}$ 
from $\Lambda$ ($\D H (\omega) = 0$).

It turns out that in each of these cases $\proto = \minnbis{\protonumber}$ and
$\entgate = \minnbis{\protonumber}\boub$. The proofs are essentially analogous to those 
that will be given in Section~\ref{sec-RA} below.


\section{Proof of Theorems \ref{th-H3} and \ref{th-RA}--\ref{th-RC}}
\label{Proofs}

In Sections~\ref{sec-RA}--\ref{sec-RC} we will identify $\proto$ and $\entgate$ for the 
subregions $\RA$, $\RB$ and $\RC$, respectively. Once the structure of the configurations 
in $\proto$ and $\entgate$ are identified, (H3-a) and (H3-b) will follow immediately. To 
prove (H3-c), we will show the existence of a $\zeta \in \gateatt(\hat{\eta})$ such that 
$\comlev(\zeta, \boxplus) < \Gamma\starred$ (i.e., the existence of a ``good site'' in 
$\hat{\eta}$). 
Due to the fact that $\Lambda$ has a border of width $3$ where there is no 
interaction between particles, it will be 
immediate that such a good site can always be reached by a particle of type $\tb$ in 
$\partial^{-}\Lambda$ without touching the cluster. Finally, to see that all $\omega 
\in \optpaths$ pass through $\gateatt$, we observe that, as long as the free particle of 
type $\tb$ does not attach itself to the cluster, the energy cannot drop below $\Gamma
\starred$ and hence no other particle is allowed to enter $\Lambda$. In particular, this 
implies that the set $\nbis{\ge \critinumber} \ni \boxplus$ cannot be reached.

In each of the following sections we look at a specific standard configuration,
labelled by the position of the lower-left particle, acting as the representative 
of the set of all standards configurations that are obtained by translation and 
rotation. Each section is split into three parts: (1) identification of $g(\{\bar{\eta}\},
\minnbis{\protonumber})$ and $\bar{g}(\{\bar{\eta}\},\minnbis{\protonumber})$ via
the motion of $\btiles$ between dual corners; (2) existence of a good site via the
construction of a path to $\boxplus$ below energy level $\Gamma\starred$; 
(3) identification of $\proto$ and $\entgate$.
 

\subsection{Region $\RA$: proof of Theorem \ref{th-RA}}
\label{sec-RA}

Let \CA be the set of $\btiled$ configurations with $\protonumber$ particles of type 
$\tb$ whose dual tile support is a rectangle of side lengths $\ell\starred,\ell\starred
-1$ plus a protuberance on one of the four side of the rectangle (see 
Fig.~\ref{fig:paradigm_class_A}).

\begin{figure}[htbp]
\centering
\subfigure[]
{\includegraphics[height=0.21\textwidth]{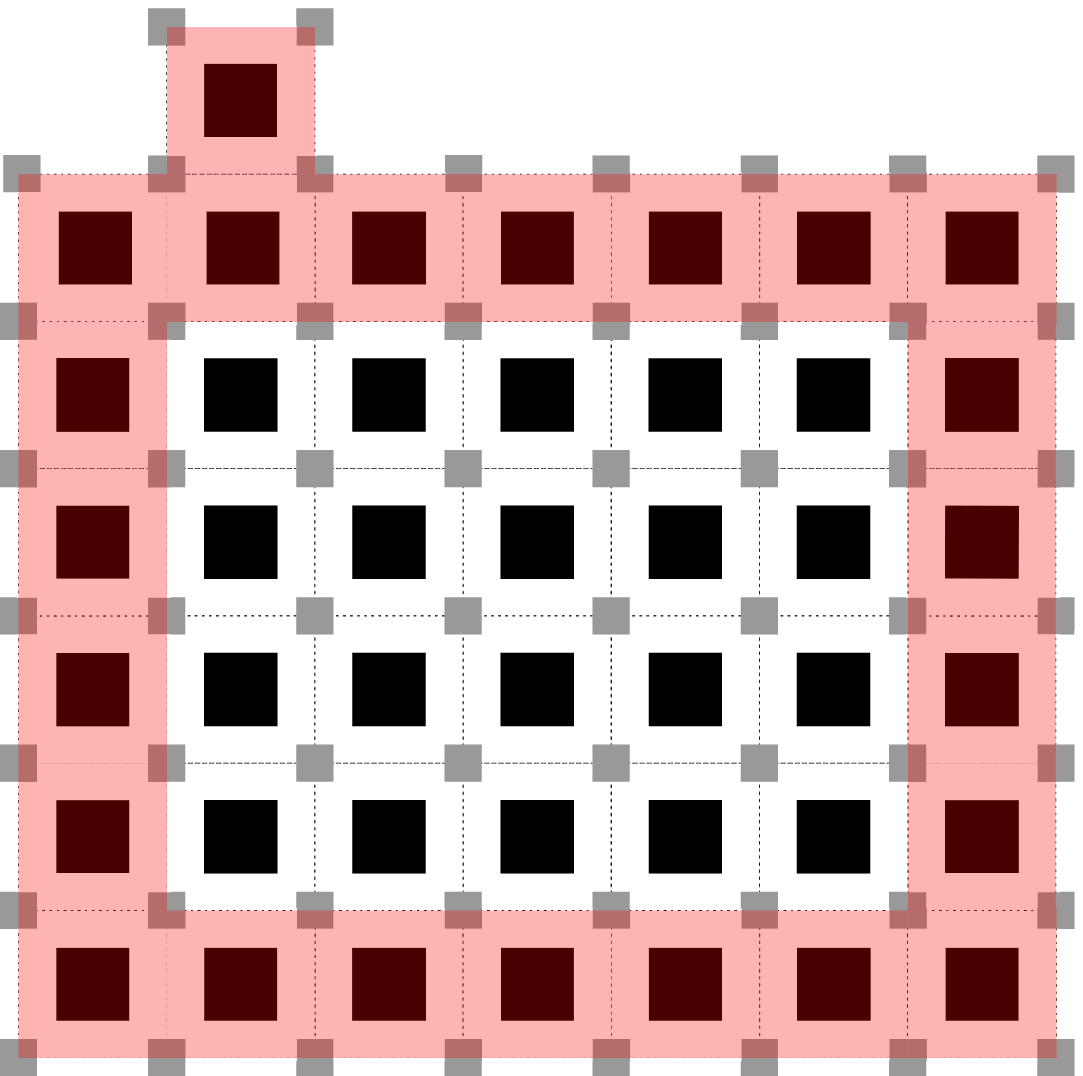}}
\qquad
\subfigure[]
{\includegraphics[height=0.18\textwidth]{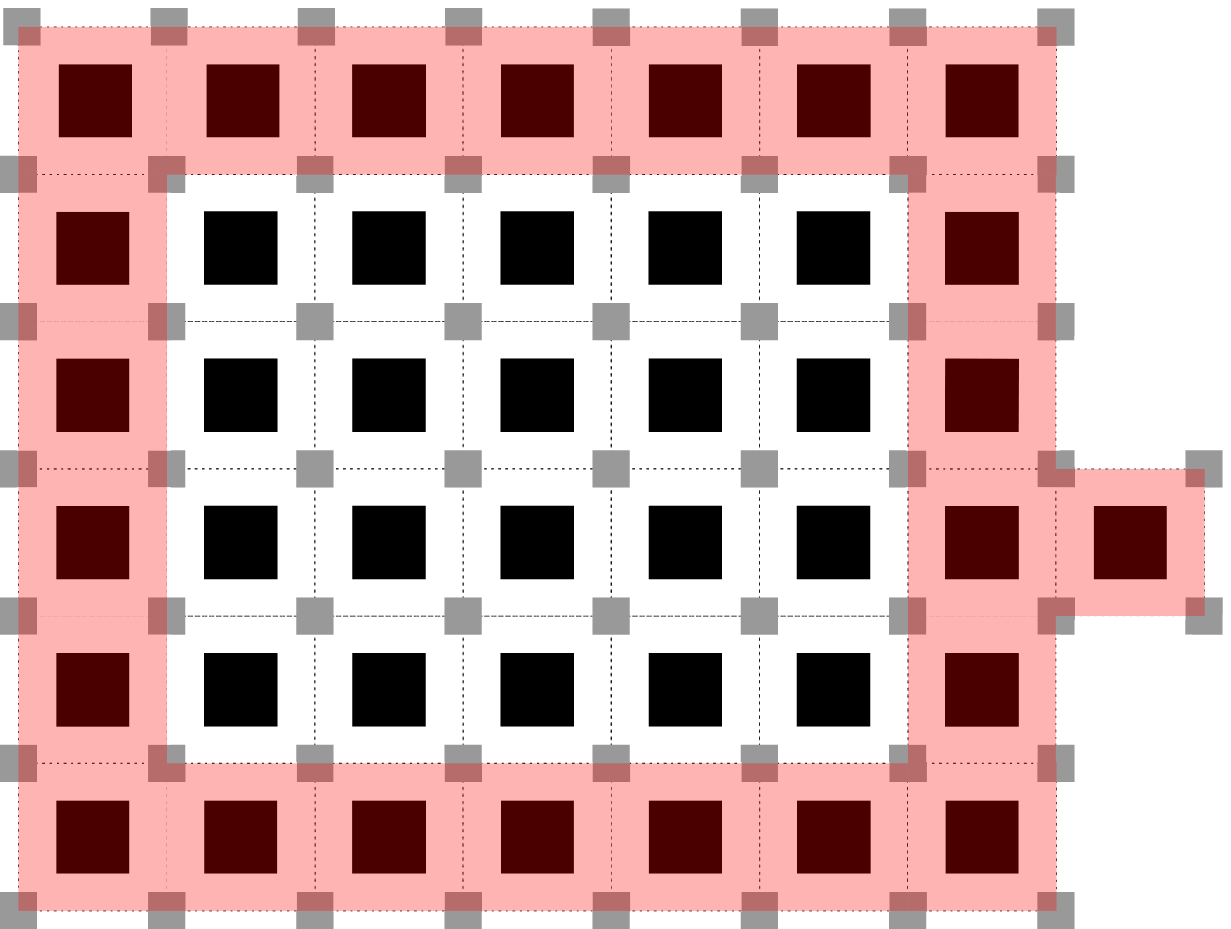}}
\caption{Two examples of configurations in \CA for $\ell\starred = 7$. The external 
particles of type $\tb$ are those lying in the shaded area.} 
\label{fig:paradigm_class_A}
\end{figure}

\begin{definition}
\label{def-modifying-path}
A path $\omega$ from $\eta \in \minnbis{\protonumber}$ to $\eta\prm \in \minnbis{\protonumber}$ 
such that $\nb(\xi) = \nb(\eta)$ for all $\xi \in \omega$ is called a modifying path 
from $\eta$ to $\eta\prm$.
\end{definition}

\noindent
With this definition, a configuration $\eta\prm \in \minnbis{\protonumber}$ belongs to 
$\bar{g}(\{\bar{\eta}\},\minnbis{\protonumber})$ (respectively, $g(\{\bar{\eta}\},\minnbis{\protonumber})$) with $\bar{\eta}$ a standard configuration in $\nbis{\protonumber}$ if 
and only if there is a modifying path from $\bar{\eta}$ to $\eta\prm$ that does not exceed (respectively, stays below) energy level $\Gamma\starred$.

\begin{remark}
\label{rem-heavy-self-avoiding}
Note that if there is a path $\omega\colon\,\eta \to \eta\prm$ that does not exceed (stay 
below) $\Gamma\starred$, then there is also a path $\omega\prm\colon\,\eta \to \eta\prm$ 
that does the same without ever completing a heavy-step and reaching a configuration that 
is equivalent to a configuration that has already been visited.
\end{remark}


\subsubsection{Identification of $g(\{\bar{\eta}\},\minnbis{\protonumber})$ and 
$\bar{g}(\{\bar{\eta}\},\minnbis{\protonumber})$}

\begin{lemma}
\label{lemma-proto-below-gamma-RA}
$g(\{\bar{\eta}\},\minnbis{\protonumber})=\bar{g}(\{\bar{\eta}\},\minnbis{\protonumber})$ 
= \CA.	
\end{lemma}

\begin{proof}
Let $\rho$ be a configuration consisting of $\btiled$ dual rectangle with horizontal side 
length $\ell\starred$ and vertical side length $\ell\starred-1$ whose top-rightmost $\btile$ 
is centered at site $(a_{\rho}, b_{\rho})$. Let $\bar{\eta}$ be the $\btiled$ standard 
configuration obtained from $\rho$ by adding a $\btiled$ protuberance centered at (dual) 
site $x = (a_{\rho} + 1,b_{\rho} + 1)$ (see Fig.~\ref{fig:region_A_classes_of_2-steps}).
Observe that $H(\bar{\eta}) = \Gamma\starred - \Db$. Let $y$ be one of the two sites of the 
tile centered at $x$ that is occupied by a particle of type $\ta$ with only one active bond. 
It is easy to check that $\CA \subseteq g(\{\bar{\eta}\},\minnbis{\protonumber})$. Indeed, it 
is enough to consider the configuration $(\rho,y)$ obtained from $\bar{\eta}$ by first 
detaching and removing the two corner particles of type $\ta$ adjacent to $x$ and then 
moving the particle of type $\tb$ from $x$ to $y$, which gives $\comlev((\rho,y),\eta) 
< \Gamma\starred$ for all $\eta \in\CA$. Note that this is true irrespective of the 
distance of $\eta$ to the boundary of $\Lambda$. 
	
To conclude the proof we will show that all modifying paths starting from $\bar{\eta}$
either lead to a configuration in $\CA$ or exceed energy level $\Gamma\starred$. Note 
that we are interested only in those modifying paths consisting of at least one heavy-step, 
otherwise the only configuration in $\minnbis{\protonumber}$ that can be reached from 
$\bar{\eta}$ is $\bar{\eta}$ itself. In other words, since all configurations in $\CA$ 
consist of a $\btiled$ dual rectangle plus a $\btiled$ protuberance, we have to show that, 
without exceeding energy level $\Gamma\starred$, i.e., within energy barrier $\Db$, the 
only configurations in $\minnbis{\protonumber}$ that can be reached via a path consisting 
of configurations with $\protonumber$ particles of type $\tb$ are obtained by ``moving'' 
the $\btiled$ protuberance.
	
Let $\eta \in \bar{g}(\{\bar{\eta}\},\minnbis{\protonumber})$, and let $\omega$ be a 
modifying path from $\bar{\eta}$ to $\eta$ such that $H(\xi)\le\Gamma\starred = 
H(\bar{\eta}) + \Db$ for all $\xi \in \omega$. Note that, since $\Db < 3U$, we have 
$\broken{\xi} \le 2$ for all $\xi \in \omega$. Furthermore, in $\CA \backslash 
(\CA\cap\RT)$ we have $\Da > \tfrac12 U$ and $\Db > U + 3\Da$. This implies that, for 
$\xi \in \omega$,
\begin{mytextlikelist}
\item if $\broken{\xi} = 1$, then $\extrapart{\xi} \le 2$;
\item if $\broken{\xi} = 2$, then $\extrapart{\xi} \le 1$.
\end{mytextlikelist}
\noindent
Fig.~\ref{fig:region_A_classes_of_2-steps} shows the different classes of edges along 
which the first heavy--step of a modifying path is possible. We will group modifying 
paths according to the edge along which the first heavy-step is made. Note that $\Db 
< 3U + \Da$ in region $\RA$ and so Lemma~\ref{lemma-first-2step-from-corner} applies. 
Therefore, as a first possible heavy-step, we need to consider only those consisting 
of the motion of a corner particle of type $\tb$ along the edge where it shares a bond 
with a corner particle of type $\ta$. In order to identify particles and edges, we 
refer to Fig.~\ref{fig:region_A_classes_of_2-steps}.

\begin{figure}[htbp]
\centering
\subfigure[]{\includegraphics[height=4\tilesize]{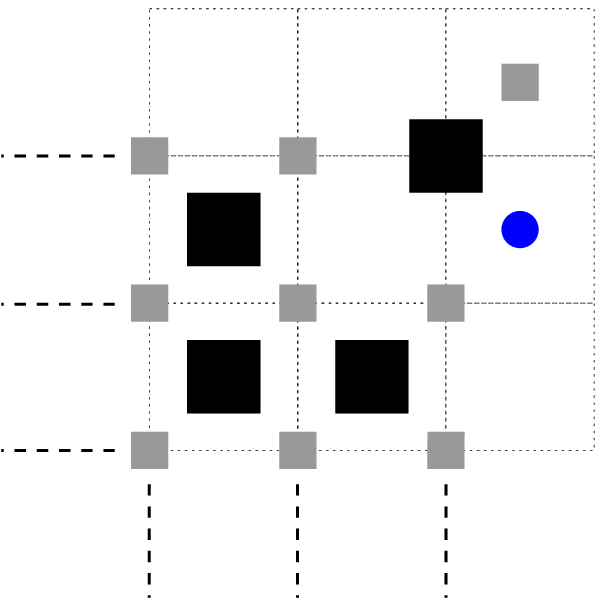}}
\qquad
\subfigure[]{\includegraphics[height=4\tilesize]{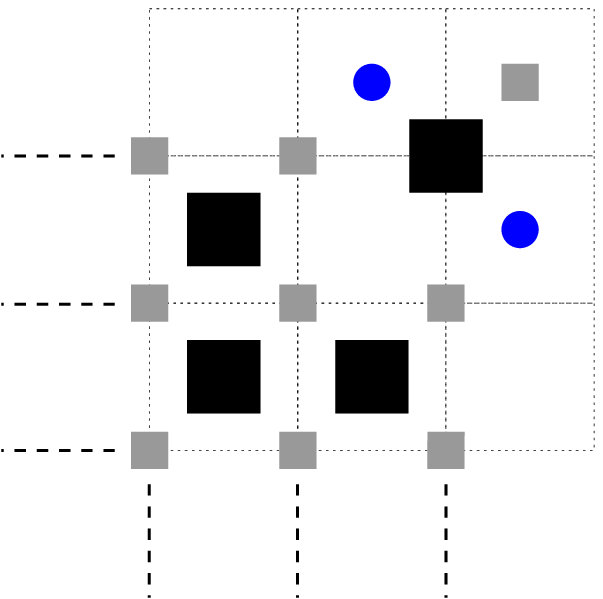}}
\qquad
\subfigure[]{\includegraphics[height=4\tilesize]{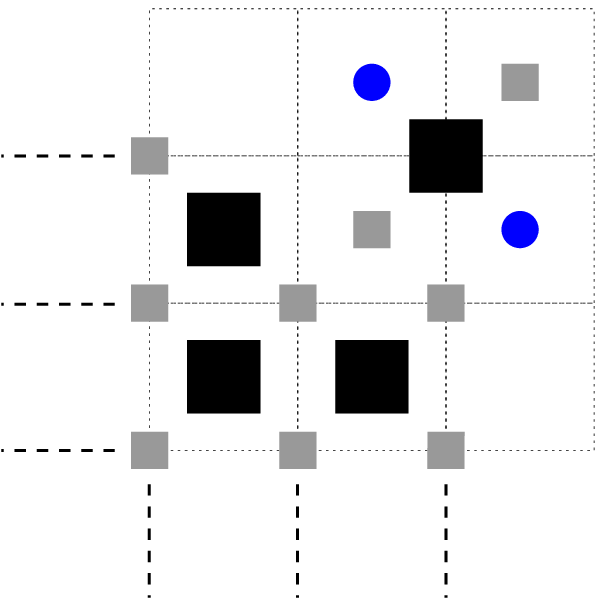}}
\caption{Region $\RA$, first heavy--step along edge $c$.} 
\label{fig:regionAmoveC}
\end{figure}
	

\step{1}
 
\begin{claim}
All modifying paths starting with a heavy-step involving a particle of type $\tb$ other 
than that in the protuberance exceed energy level $\Gamma\starred$.
\end{claim}
	
\begin{proof}
{\bf 1.} Asumme that he first heavy-step is along edge $c$. Let $x$ be the site where 
the particle $p_{1}$ that we want to move sits in configuration $\bar{\eta}$. Let 
$\eta_{1-}$ be the configuration visited just before the first heavy-step is performed, 
and let $\eta_{1}$ be the configuration that is reached when $p_{1}$ is moved one step 
North-East to site $y$. There are four possible cases.
\begin{itemize}
\item
$B(p_{1}, \eta_{1}) = 0$. Let us assume that all free particles of type $\ta$ are removed 
from $\Lambda$. Then $H(\eta_{1}) - H(\bar{\eta}) = 4U - \Da > \Db$ and hence the path 
exceeds energy level $\Gamma\starred$.
\item
$B(p_{1}, \eta_{1}) = 1$: $H(\eta_{1}) - H(\bar{\eta}) = 3U > \Db$.
\item
$B(p_{1}, \eta_{1}) = 2$. Let us assume that the two sites occupied by a particle of type 
$\ta$ are $y_{1}$ (South-East of $y$) and $y_{2}$ (North-East of $y$). Since $\Da < U$, 
the least expensive way in terms of energy cost to have $B(p_{1}, \eta_{1}) = 2$ is achieved 
by bringing one extra particle of type $\ta$ in $\Lambda$ (see Fig.~\ref{fig:regionAmoveC}(a)): 
$H(\eta_{1}) - H(\bar{\eta)} = 2U + \Da$. Therefore, starting from $\eta_{1}$, only moves 
that do not increase the energy are possible. Since $\eta_{1}$ is not equivalent to any 
configuration in $\minnbis{\protonumber}$, at least one extra heavy-step is necessary from 
$\eta_{1}$. By Remark~\ref{rem-heavy-self-avoiding}, the first heavy-step from $\eta_{1}$ 
cannot be completed by moving $p_{1}$ back to $x$. Before the next heavy-step is performed, 
the only moves from $\eta_{1}$ that do not increase the energy further are motions of 
particles of type $\ta$ with one active bond to or from a site adjacent to $p_{1}$. In 
particular, it is not possible to bring inside $\Lambda$ any other particle or type $\ta$. 
Any possible sequence of such moves cannot change the energy of the configuration. When 
the next heavy-step is completed, at least one extra bond is added and energy level 
$\Gamma\starred$ is exceeded.
\item
$B(p_{1}, \eta_{1}) = 3$ (see Fig.~\ref{fig:regionAmoveC}(b)). Similarly to the previous case, 
the least expensive way to have $B(p_{1},\eta_{1}) = 3$ is achieved by bringing two extra 
particles of type $\ta$ inside $\Lambda$. From $\eta_{1}$ at least one other heavy-step is 
necessary, and moves that increase the energy are not allowed. As in the previous case, only 
motions of particles of type $\ta$ that do not decrease the number of bonds are allowed (see, 
for instance, Fig.~\ref{fig:regionAmoveC}(c)), and the completion of the next heavy-step 
exceeds energy level $\Gamma\starred$.
\end{itemize}
	
\begin{figure}[htbp]
\centering
\subfigure[]{\includegraphics[height=5\tilesize]{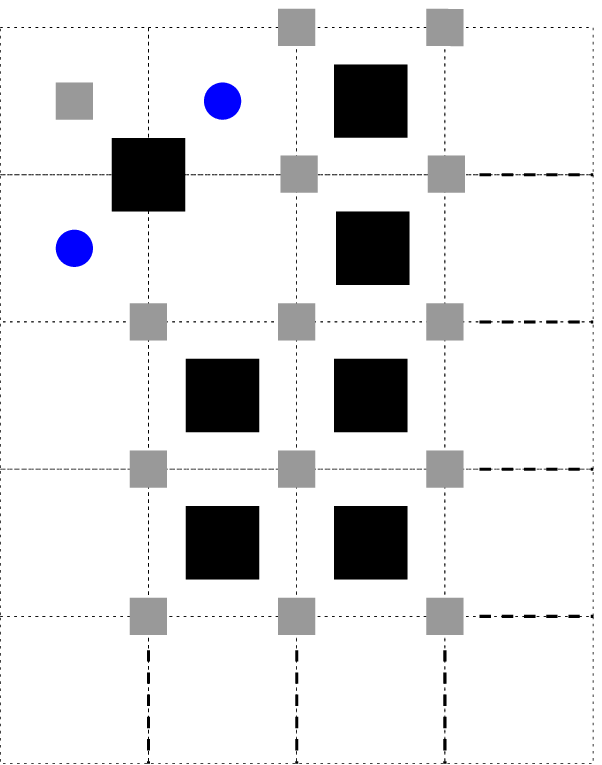}}
\qquad
\subfigure[]{\includegraphics[height=5\tilesize]{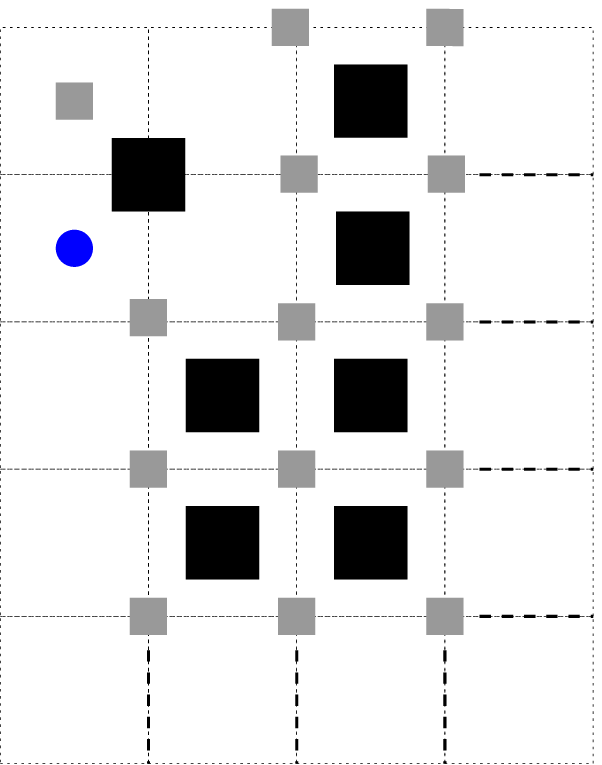}}
\qquad
\subfigure[]{\includegraphics[height=5\tilesize]{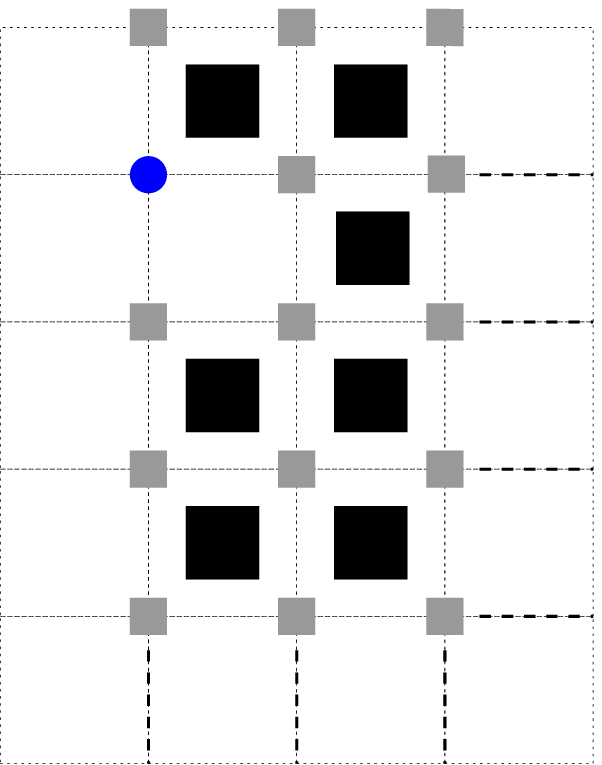}}
\caption{Region $\RA$, first heavy--step along edge $d$.} 
\label{fig:regionAmoveD}
\end{figure}

\medskip\noindent
{\bf 2.} 
Assume that the first heavy--step is along edge $d$. Let $x$ be the site where the 
particle $p_{1}$ that we want to move sits in configuration $\bar{\eta}$. Let $\eta_{1-}$ 
be the configuration visited just before the first heavy-step is performed, and let 
$\eta_{1}$ the configuration that is reached from $\eta$ after moving $p_{1}$ one step 
North-West along the edge $d$ to site $y$. Also, let $p_{2}$ denote the particle of type 
$\tb$ in the tile protuberance, and $p_{3}$ the particle of type $\tb$ at dual distance 
one from $p_{1}$ in the South direction in configuration $\bar{\eta}$. As in the previous
case, after the heavy-step is performed we must have $B(p_{1}, \eta_{1}) \ge 2$.

\begin{itemize}
\item
$B(p_{1}, \eta_{1}) = 3$. The best choice is when $\broken{\eta_{1}} = 1$ and 
$\extrapart{\eta_{1}} = 2 $ (see Fig.~\ref{fig:regionAmoveD}(a)). From $\eta_{1}$ 
only moves that do not increase the energy are possible. There is only possible one
non-backtracking move: the particle of type $\ta$ South-West of $x$ is moved one 
step North-East. After this move, it is only possible to move the particle of type 
$\ta$ South-West of $p_{1}$ one step South-West. The configuration that is reached 
does not belong to $\minnbis{\protonumber}$ and no other move is allowed (this is the
analogue of $B(p_{1}, \eta_{1}) = 3$ in the previous case).
\item
$B(p_{1}, \eta_{1}) = 2$. As in the previous case, the choice that minimizes the 
energy of $\eta_{1}$ is such that $\broken{\eta_{1}}=2$ and $\extrapart{\eta_{1}} = 1$, 
and the particles of type $\ta$ adjacent to $p_{1}$ sit at sites $y_{1}$ and $y_{2}$, 
respectively, South-West and North-West of $y$ (see Fig.~\ref{fig:regionAmoveD}(b)). 
It is easy to see that this is the choice that allows for ``more freedom'' of the 
path, in the sense that it is the only choice from which it is possible to complete 
a further heavy-step. Therefore, only moves that do not increase the energy are 
allowed starting from $\eta_{1}$, and again note that no other particle of type $\ta$ 
is allowed to enter $\Lambda$. From $\eta_{1}$ at least one other heavy-step is necessary. 
This means that from $\eta_{1}$ it is only possible to move $p_{1}$ one step North-East 
to site $z$ (this case will be examined afterwards) or to start a sequence of moves 
of particles of type $\ta$ to or from a site adjacent to $p_{1}$. Any possible sequence 
of such moves cannot decrease the energy of the configuration and only a heavy-step 
can be completed by moving $p_{1}$ one step North-East to site $z$ to reach configuration
$\eta_{2}$. Clearly, the configurations the path can reach (strictly) below $\Gamma\starred$ 
from $\eta_{2}$ are the same as those that can be reached from $\eta_{3}$ by saturating 
$p_{1}$ with the two free particle of type $\ta$ (see Fig.~\ref{fig:regionAmoveD}(c)). 
Arguing as in the previous case, we see that the only heavy-step possible without 
exceeding energy level $\Gamma\starred$ from $\eta_{3}$ is the one that is completed 
by moving $p_{1}$ back to site $y$ (configuration $\eta_{4}$). The transition from $\eta_{3}$ 
to $\eta_{4}$ can be treated in the same way as that from $\bar{\eta}$ to $\eta_{1}$,
and so we conclude that from $\eta_{4}$ it is only possible to reach a configuration 
equivalent to $\bar{\eta}$.	
\end{itemize}
\end{proof}
	

\step{2}
 
\begin{claim}
Let the first step of a modifying path starting from a configuration in $\CA$ involve 
the particle $p$ of type $\tb$ in the protuberance. If $p$ is not re-attached to the 
main cluster before the next heavy-step is completed, then the path exceeds energy 
level $\Gamma\starred$ before reaching a configuration in $\minnbis{\protonumber}$.
\end{claim}

\begin{proof}	
The proof of this claim will be deferred to the proof of Lemma~\ref{lemma-proto-below-gamma-RB},
Step 1, in Section~\ref{sec-RB}. There it will be shown that the same claim holds also in 
region $\RB$, where the values that $\Db$ can take are larger than those that $\Db$ can take 
in $\RA$.
\end{proof}
	
Depending on its starting position, the protuberance can be re-attached in two possible
ways: either the particle of type $\tb$ shares two particles of type $\ta$ with the other 
particles of type $\tb$ in the cluster (see Fig.~\ref{fig:possible-reattachment}(a)),
or it shares only one particle of type $\ta$ (see Fig.~\ref{fig:possible-reattachment}(b)). 
In both cases we consider the evolution of the path to a $\btiled$ configuration that is 
equivalent to the one reached the moment the particle of type $\tb$ joins the main cluster.
In the first case, the configuration is again in the class $\CA$, and hence the same kind 
of argument can be repeated. In the second case, the following statement allow us to conclude 
the proof of Lemma~\ref{lemma-proto-below-gamma-RA}.	

\begin{figure}[htbp]
\centering
\subfigure[]{\includegraphics[height=4\tilesize]{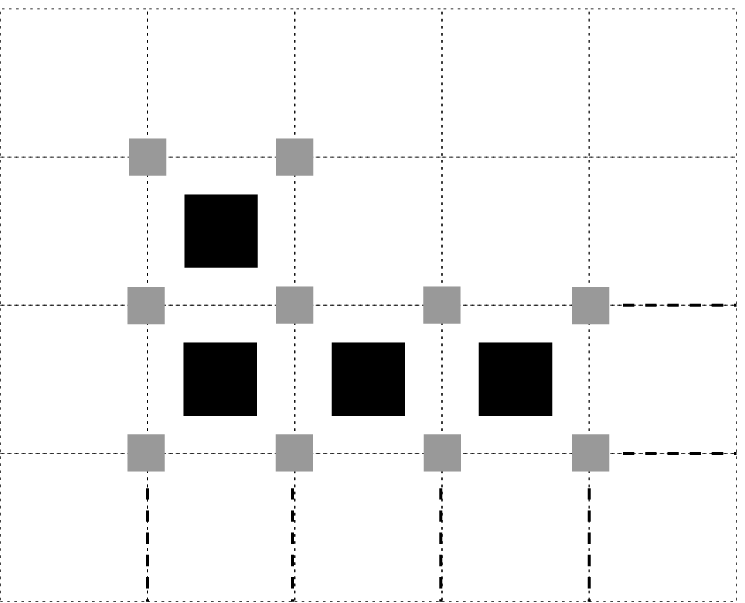}}
\qquad
\subfigure[]{\includegraphics[height=4\tilesize]{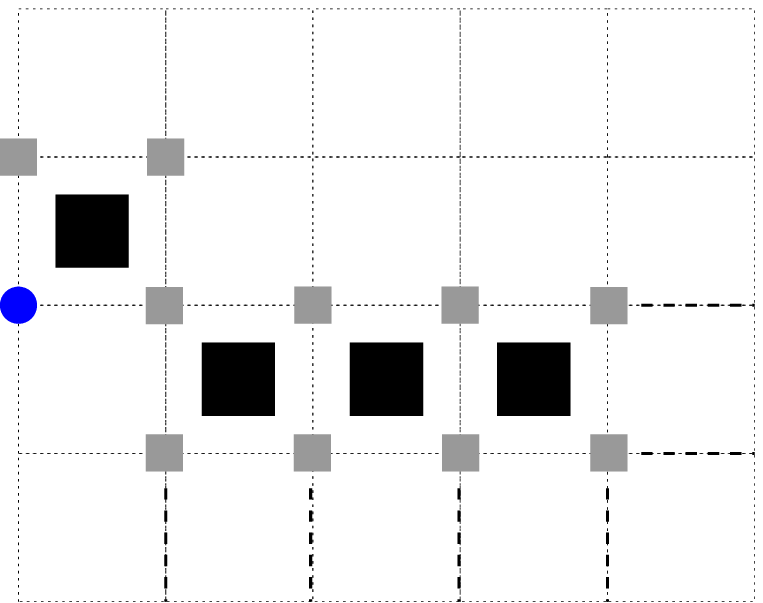}}
\caption{Modifying paths may create hanging protuberances.} 
\label{fig:possible-reattachment}
\end{figure}
	
In the first case, the configuration is again in the class $\CA$ and, hence, the same kind 
of analysis can be repeated. In the second case, the following statement allow us to 
conclude the proof of the lemma.	


\step{3}

\begin{claim} 
From a configuration consisting of a $\btiled$ rectangle plus a hanging protuberance 
it is not possible to reach a configuration in $\minnbis{\protonumber}\backslash\CA$ 
without exceeding energy level $\Gamma\starred$.
\end{claim}

\begin{proof}
Let $\eta_{0}$ be the configuration of Fig.~\ref{fig:possible-reattachment}(b), $\eta_{1}$ 
the configuration reached when the first heavy-step from $\eta_{0}$ is completed, and 
$\eta_{1-}$ the configuration visited by the path just before $\eta_{1}$. Let $p_{1}$ 
be the particle of type $\tb$ in the hanging protuberance and $p_{2}$ the particle of type 
$\tb$ at dual distance $\sqrt{2}$ in the South-East direction from $p_{1}$. Let $q_{1}$ 
denote the particle of type $\ta$ shared by $p_{1}$ and $p_{2}$. We will show that 
if the first heavy-step from $\eta_{0}$ is not completed by moving $p_{1}$ one step 
North-East of one step South-West (the two cases are analogous), then the path exceeds 
energy level $\Gamma\starred$. From Lemma~\ref{lemma-first-2step-from-corner} and Step 1 
it follows that, in order to prove the claim, we need to consider only the heavy-steps 
completed by moving $p_{1}$ North-West or South-East and $p_{2}$ North-West.

\begin{itemize}	
\item
Assume that $p_{1}$ is moved one step South-East. Observe that $B(p_{1}, \eta_{1}) \le 2$ 
and $B(p_{2}, \eta_{1}) \le 3$. Clearly, since $\Da < U$, the choice that is most favorable 
from the point of view of energy is when $p_{1}$ has $2$ active bonds and $p_{2}$ has $3$ 
active bonds. It is clear that, since $\eta_{1}$ is reached from $\eta_{1-}$ via the motion 
of a particle of type $\tb$, the two configurations have the same particles of type $\ta$ 
placed at the same sites. Since $B(p_{1}, \eta_{1}) = 2$, there is no advantage in having 
more than two particles of type $\ta$ adjacent to $p_{1}$ in $\eta_{1-}$, since one bond 
will be lost anyway with the motion of $p_{2}$. It follows that $H(\eta_{1}) - H(\bar{\eta}) 
= 3U + \Da > \Db$ (see Fig.~\ref{fig:hanging-ra}(a)). Arguing in the same way, we see that 
in the case where the first heavy-step from $\eta_{0}$ is completed by moving $p_{2}$ one 
step North-West, we have $H(\eta_{1}) - H(\bar{\eta}) = 3U + 2\Da > \Db$ (see 
Fig.~\ref{fig:hanging-ra}(b)). (Note that $3U + \Da > \Db$ also in region $\RB$, and so
these moves will be forbidden there as well.)
\item	
Assume that the first heavy-step is in the North-West direction. Then the only possibility 
without exceeding energy level $\Gamma\starred$ is when $B(p_{1}, \eta_{1}) = 2$. This is 
achieved, for instance, by moving the two particles of type $\ta$ on the West side of $p_{1}$ 
one step North-West before completing the heavy-step. Since $\eta_{1}$ has two broken bonds, 
it is not possible to break any extra bond, and hence no extra heavy-step is possible as 
long as the particle of type $\ta$ does not reach a site adjacent to $p_{1}$. Let $\eta_{2}$ 
be such a configuration (see Fig.~\ref{fig:hanging-ra}(d)). Since $\eta_{2}$ has one broken 
bond and one extra particles of type $\ta$, the next heavy-step (completed by moving either 
$p_{1}$ or $p_{2}$) cannot break more than one bond, but clearly this is impossible.
\item
Assume that the first heavy-step is completed by moving $p_{1}$ North-East (see 
Fig.~\ref{fig:hanging-ra}(c)). Then $\eta_{1}$ is a configuration that can be reached 
with one heavy-step from a configuration in $\CA$, and hence the claim in Step 2 holds.
This can be done by moving the two particles of type $\ta$ North of $p_{1}$ one step 
North-East (breaking two active bonds), moving $p_{1}$ one step North-East ($\D H=0$) 
and removing the free particle of type $\ta$ froma $\Lambda$ ($\D H=-\Da$).
\end{itemize}
\end{proof}

\end{proof}

\begin{figure}[htbp]
\centering
\subfigure[]{\includegraphics[height=4\tilesize]{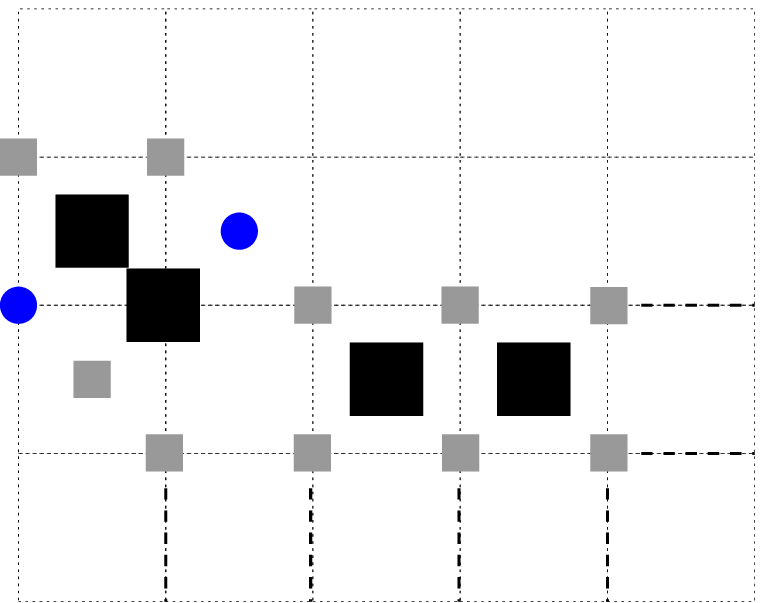}}
\qquad
\subfigure[]{\includegraphics[height=4\tilesize]{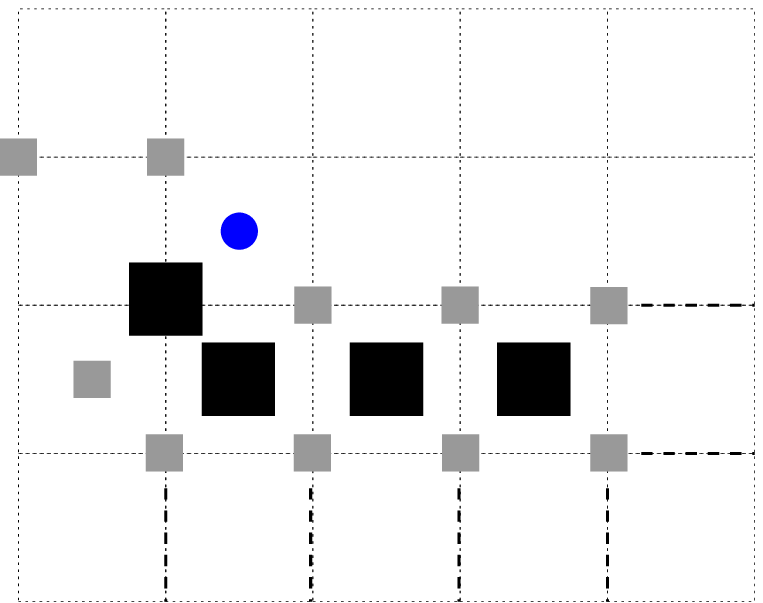}}
\qquad
\subfigure[]{\includegraphics[height=4\tilesize]{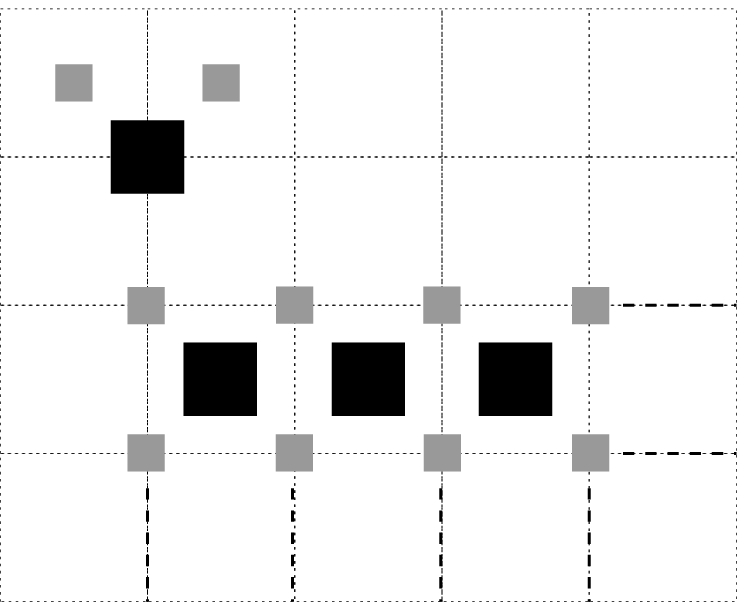}}
\qquad
\subfigure[]{\includegraphics[height=4\tilesize]{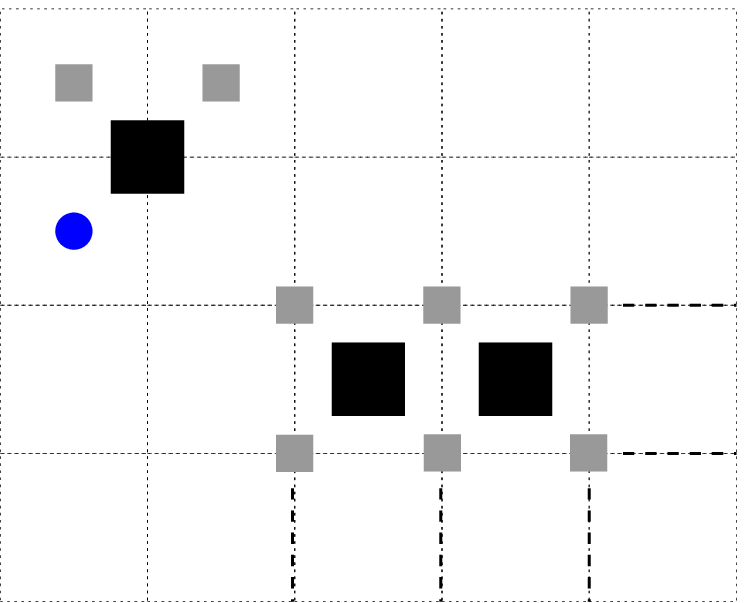}}
\caption{Region $\RA$, first heavy--step along edge $e$.} 
\label{fig:hanging-ra}
\end{figure}

\subsubsection{Existence of a good site}

\begin{lemma}
\label{lemma-proto-is-good-RA}
For all $\hat{\eta} \in \CA$, there exists an $x \in F(\hat{\eta})$ such that 
$\comlev((\hat{\eta},x),\boxplus) < \Gamma\starred$.
\end{lemma}

\begin{proof}
Configurations in $\CA$ consist of a dual $\btiled$ square and a protuberance on one 
of the four sides. If the protuberance is on the longest side, then the rectangle 
circumscribing the cluster is a square of side length $\ell\starred$. Conversely, if 
the protuberance is on one of the shortest sides, then the rectangle circumscribing 
the cluster has side lengths $\ell\starred - 1,\ell\starred + 1$. Without loss of 
generality, we assume that the longest side of the rectangle is horizontal and that
the protuberance is on the North side of the rectangle when it is attached to the 
longest side and on the East side when it is attached to the shortest side.
\begin{itemize}
\item
Assume that the protuberance is on the longest side. Let $x$ be the central site of 
a tile adjacent to the protuberance to the North side of the rectangle. After a particle 
of type $\tb$ has entered $\Lambda$ and has been moved to $x$, a configuration with 
energy $\Gamma\starred - 3U$ is reached. The particle at site $x$ can be saturated 
within the energy barrier $\Da$, to reach a configuration with energy $\Gamma\starred 
- \Db - \epsi < H(\eta)$ for all $\eta \in \CA$, consisting of a rectangle of side 
lengths $\ell\starred, \ell\starred - 1$ plus a $\abbar$ of length $2$ on its (longest) 
north-side. Then, within energy barrier $\Db$, the $\abbar$ can be completed to obtain 
a $\btiled$ dual square of side length $\ell\starred$. The claim in the lemma follows 
from Remark~\ref{remark-supercritical-square}. Such a configuration is supercritical 
(and has energy smaller than the standard configuration with $\quasiqnumber$ particles 
of type $\tb$). Therefore it can grow until $\boxplus$ is reached while staying below 
energy level $\Gamma\starred$. 
\item
Assume that the protuberance is on the shortest side. Let $x$ be the central site of 
a tile adjacent to the protuberance to the East side of the rectangle. After a particle 
of type $\tb$ has entered $\Lambda$ and has been moved to $x$, a configuration with 
energy $\Gamma\starred - 3U$ is reached. The particle at site $x$ can be saturated 
within energy barrier $\Da$, to reach a configuration with energy $\Gamma\starred 
- \Db - \epsi < H(\eta)$ for all $\eta \in \CA$, consisting of a rectangle of side 
lengths $\ell\starred, \ell\starred - 1$ plus a $\abbar$ of length $2$ on its shortest 
(see Fig.~\ref{fig:small-bar-to-top-side}(a)).

\begin{figure}[htbp]
\centering
\subfigure[]{\includegraphics[height=5\tilesize]{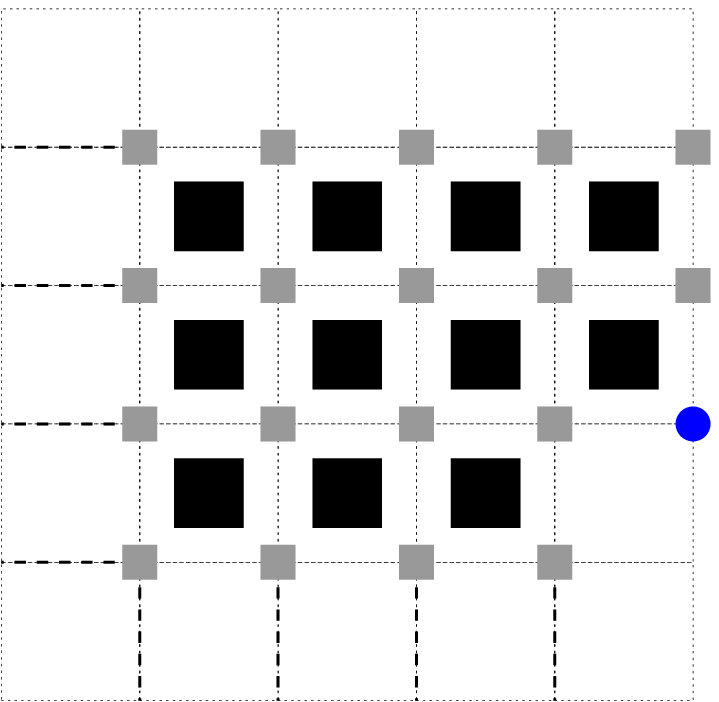}}
\qquad
\subfigure[]{\includegraphics[height=5\tilesize]{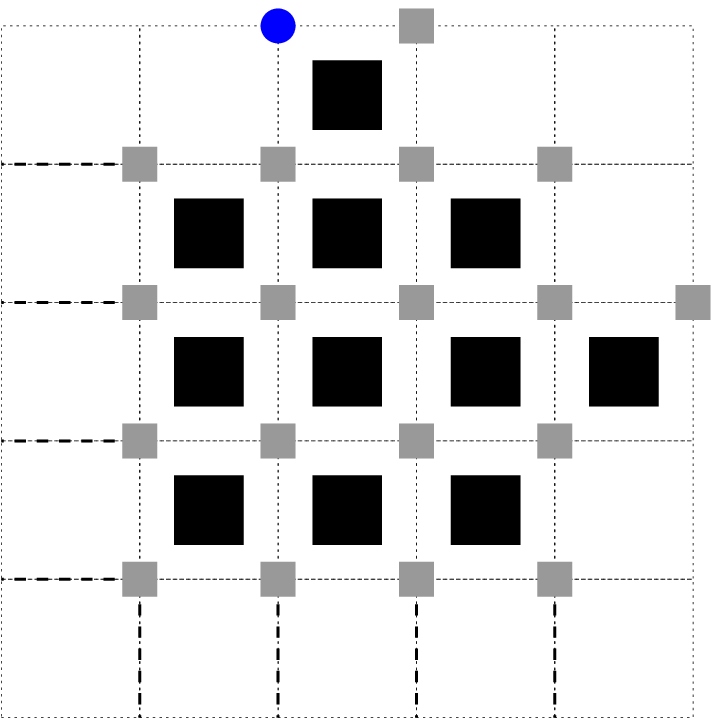}}
\qquad
\subfigure[]{\includegraphics[height=5\tilesize]{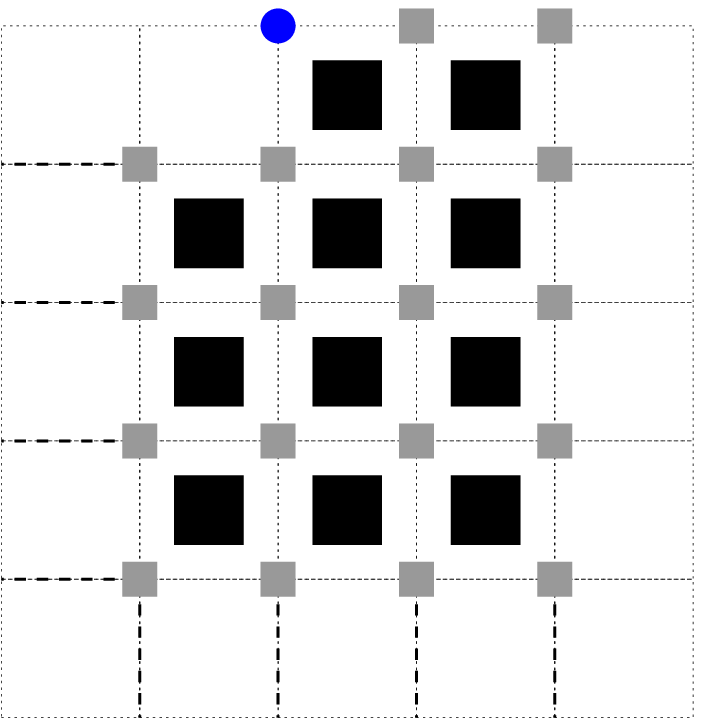}}
\caption{Region $\RA$, first heavy--step along edge $e$.} 
\label{fig:small-bar-to-top-side}
\end{figure}

\medskip
Let $y$ be the central site of a tile adjacent to the North side of the rectangle of side 
lengths $\ell\starred, \ell\starred - 1$. One of the dimers of the $\abbar$ of length $2$ 
can be moved within energy barrier $3U$ (note that $3U - \Db - \epsi < 0$) to the north 
Side of the rectangle (as described in Section~\ref{sec-moving-dimers}), with the particle 
of type $\tb$ at site $y$, to reach a configuration with energy $\Gamma\starred - \Db - 
\epsi + U$. The particle of type $\tb$ can be saturated using one of the corner particle of 
type $\ta$ adjacent to the protuberance that is left on the East side. This is done within 
energy barrier $U$, and the resulting configuration again has energy $\Gamma\starred - \Db 
- \epsi + U$ (see Fig.~\ref{fig:small-bar-to-top-side}(b)). Finally, move the pending dimer 
that is left on the East side to a corner adjacent to the tile centered at site $y$ (see 
Fig.~\ref{fig:small-bar-to-top-side}(c)). As described in Section~\ref{sec-moving-dimers}, 
this can be done within energy barrier $2U$ (since the particle of type $\tb$ of the dimer 
has only one neighboring corner particle of type $\ta$). The configuration obtained in 
this way consists of a rectangle of side lengths $\ell\starred,\ell\starred-1$ plus a 
$\abbar$ of length $2$ on its longest North-side. It is now possible to continue as in 
the previous case. Note that, to move the $\abbar$ of length 2 from the East side to the 
North side, we implicitly assumed that these sides of the rectangle are far from the 
boundary of $\Lambda$. 

Now suppose that the protuberance was originally on a side of the rectangle close to 
$\partial^{-}\Lambda$. Then we can proceed in the following way. After the $\abbar$ of 
length $2$ has been completed, it is possible to complete the $\abbar$ below energy level 
$\Gamma\starred$ and obtain a $\btiled$ rectangle of side lengths $\ell\starred+1, 
\ell\starred-1$. From this configuration it is possible to iteratively remove $\ell\starred 
- 3$ corner $\btiles$ from the shorthest side of the rectangle that is far from 
$\partial^{-}\Lambda$, in order to obtain a $\btiled$ configuration where the $\abbar$ 
of length $2$ is far from $\partial^{-}\Lambda$ as well. Clearly, this can be achieved 
below energy level $\Gamma\starred$. It is now possible to proceed as in the previous case, 
observing that either the North side or the South side of the rectangle is far from
$\partial^{-}\Lambda$.
\end{itemize}
\end{proof}


\subsubsection{Identification of $\proto$ and $\entgate$}

\begin{lemma}
\label{lemma-identification-of-entgate-RA}
$\proto = \CA$ and $\entgate = \CA\boub$.
\end{lemma}

\begin{proof}
All configurations in $\CA\boub$ have energy $\Gamma\starred$. We will show that all 
configurations in $\CA\boub$ are essential saddles, and therefore belong to $\cG(\Box,
\boxplus)$ by Lemma~\ref{lemma-mnos-essential-saddles}. Pick $\eta = (\hat{\eta},x)$ 
in $\CA\boub$. 

\medskip
Let $\cU(\eta)$ be the set of optimal paths entering $\nbis{\critinumber}$ via configuration 
$\eta$. Pick $\omega \in \cU(\eta)$ such that $H(\xi) < \Gamma\starred$ for all $\xi \in 
P_{\omega}(\eta)$ and $S_{\omega}(\eta) = \{(\hat{\eta},y_{1}),\ldots,(\hat{\eta}, y_{m}),
\ldots,\boxplus\}$, with $y_{i} \notin \partial^{-}\Lambda$ for all $i \in 1,\dots,m$ and 
$y_{m}$ in $F(\hat{\eta})$ such that $H(\xi)<\Gamma\starred$ for all $\xi \in S_{\omega}
((\hat{\eta}, y_{m}))$. Note that such a path $\omega$ exists because:
\begin{mytextlikelist}
\item[(i)] 
$\hat{\eta} \in g(\Box,\minnbis{\protonumber})$ by Lemma~\ref{lemma-proto-below-gamma-RA};
\item[(ii)]
there exists an $y_{m} \in F(\hat{\eta})$ such that $\comlev((\hat{\eta}, y_{m}))
< \Gamma\starred$ by Lemma~\ref{lemma-proto-is-good-RA};
\item[(iii)]
in configuration $\eta$ there exists a lattice path of empty sites $y_{1}, \ldots, y_{m}$ 
with $y_{1} \sim x$ and $y_{i} \notin \partial^{-}\Lambda$ for all $i \in 1, \dots, m$, 
since for all $x \in \partial^{-}\Lambda$ there exists an $y \sim x$ such that $y \notin 
\partial^{-}\Lambda$ (by the shape of $\Lambda$), and $[\hat{\eta}]$ is ``sufficiently 
far'' from $\partial^-\Lambda$ (i.e., there is a border of width $3$ in $\Lambda$ where 
particles do no interact).
\end{mytextlikelist}
Next, note that all saddles in $S_{\omega}(\eta)$ are configurations of the type $\eta\prm 
= (\hat{\eta},u)$, $u \notin \partial^{-}\Lambda$. Recall Definition~\ref{def3}(e). We have 
to show that $S(\omega\prm) \nsubseteq S(\omega) \backslash \{\eta\}$ for all $\omega\prm 
\in (\eta, \eta\prm)_{\mathrm{opt}}$. Consider the partition $\optpaths = (\cU(\eta),
\cU^{c}(\eta))$ with $\cU^{c}(\eta) = \optpaths \backslash \cU(\eta)$. If $\omega\prm \in 
\cU(\eta)$, then $S(\omega\prm) \nsubseteq S(\omega) \backslash \{\eta\}$ because $\eta 
\in \omega\prm$. If, on the other hand, $\omega\prm \in \cU^{c}(\eta)$, then, by 
Lemma~\ref{lemma-entrance-of-set-protonumber}(1), $\omega\prm$ enters the set 
$\nbis{\critinumber}$ via some configuration $\zeta = (\hat{\zeta},z)$ with $H(\zeta) 
= \Gamma\starred$ and $z \in \partial^{-}\Lambda$ that does not belong to $S(\omega)$ 
(by the construction of $\omega$). Since the choice of $\eta\in\CA\boub$ was arbitrary, 
we conclude that all configurations in $\CA\boub$ are essential saddles, and hence that 
$\CA\boub \subset \cG(\Box,\boxplus)$. 
	
To prove that $\CA\boub = \entgate$, i.e., $\CA\boub$ is the entrance of the essential 
gate, we will show that, for any path $\omega \in \optpaths$, any configuration $\zeta \in 
\cS(\Box,\boxplus)$ that is visited by $\omega$ before some configuration in $\CA\boub$ 
is an unessential saddle and therefore does not belong to $\cG(\Box,\boxplus)$.

\begin{itemize}
\item
We show that all saddles in $\nbis{\le \protonumber}$ are unessential. To that end, we 
pick $\omega \in \optpaths$ and we let $\zeta \in \omega$ be such that $H(\zeta) = \Gamma
\starred$ and $n_{2}(\zeta) \le \protonumber$. By Lemma~\ref{lemma-entrance-of-set-protonumber}(1), 
all optimal paths enter $\nbis{\critinumber}$ via a configuration in $\bar{g}(\Box,
\minnbis{\critinumber})\boub$. Hence, after visiting $\zeta$, $\omega$ must visit 
$\bar{g}(\Box, \minnbis{\critinumber})$ before entering $\nbis{\critinumber}$. By 
Lemma~\ref{lemma-proto-below-gamma-RA}, combined with Lemma~\ref{lemma-reachable-from-standard},
we have $\bar{g}(\Box, \minnbis{\critinumber}) = g(\Box, \minnbis{\critinumber})$. The 
claim now follows via Lemma~\ref{lemma-unessential-saddles}(1).
\item
For all $\omega \in \optpaths$ such that $\xi \in S(\omega)$ and $\xi \in 
\nbis{\ge\critinumber} \backslash \CA\boub$, there exists an $\eta \in \CA\boub$ such that 
$\eta \in P_{\omega}(\xi)$. Indeed, by Lemmas~\ref{lemma-entrance-of-set-protonumber}(1) 
and \ref{lemma-proto-below-gamma-RA}, all optimal paths enter $\nbis{\protonumber}$ via 
a configuration in $\CA\boub$.
\item
To conclude, we need to show that $\cP = \CA$. The inclusion $\cP \supseteq \CA$ is immediate, 
since all $\eta = (\hat{\eta}, x) \in \CA\boub = \entgate$ are obtained, in a single step,
by adding a particle of type $\tb$ at the boundary of $\Lambda$ to a configuration 
$\hat{\eta} \in g(\Box,\minnbis{\protonumber}) = \CA$. To see that $\cP \subseteq \CA$, 
suppose that there is a configuration $\hat{\eta} \in \cP \backslash \CA$. Let $\hat{\eta}\prm$ 
be the configuration in $\entgate = \CA\boub$ obtained in a single step from $\hat{\eta}$.
Observe that $\CA\boub \subset \entrance{\nbis{\critinumber}}$. Since, by 
Lemma~\ref{lemma-entrance-of-set-protonumber}(1), all optimal paths enter $\nbis{\critinumber}$ 
by adding a particle of type $\tb$ at $\partial^-\Lambda$ to a configuration in $g(\Box,
\minnbis{\protonumber}) = \CA$, it follows that $\hat{\eta} \in \nbis{\ge\critinumber}$.
Indeed, by assumption, this configuration cannot be in $\g=\CA$. In particular, it follows 
that $\hat{\eta} \in \nbis{\ge \critinumber}\backslash\CA\boub$. But this means that there 
is a path $\hat{\omega} \in \optpaths$ that reaches $\hat{\eta} \in \nbis{\ge\critinumber}
\backslash\CA\boub$ and does not contain any configuration in $\cG(\Box,\boxplus)$. In 
particular, $\hat{\omega}$ does not contain any configuration in $\CA\boub$. But this is 
a contradiction, since by Lemmas~\ref{lemma-entrance-of-set-protonumber}(1) and 
\ref{lemma-proto-below-gamma-RA}, all optimal paths enter $\nbis{\protonumber}$ via a 
configuration in $\CA\boub$.	
\end{itemize}
\end{proof}


\subsection{Region $\RB$: proof of Theorem \ref{th-RB}}
\label{sec-RB}

Let $\CB$ be the set of $\btiled$ configurations with $\protonumber$ particles of type 
$\tb$ whose dual tile support is a monotone polyomino and whose circumscribing rectangle 
has side lengths either $\ell\starred, \ell\starred$ or $\ell\starred + 1, \ell\starred - 1$
(see Fig.~\ref{fig:paradigm_class_B}). Note that $\CB \supseteq \CA$.

\begin{figure}[htbp]
\centering
\subfigure[]
{\includegraphics[height=0.21\textwidth]{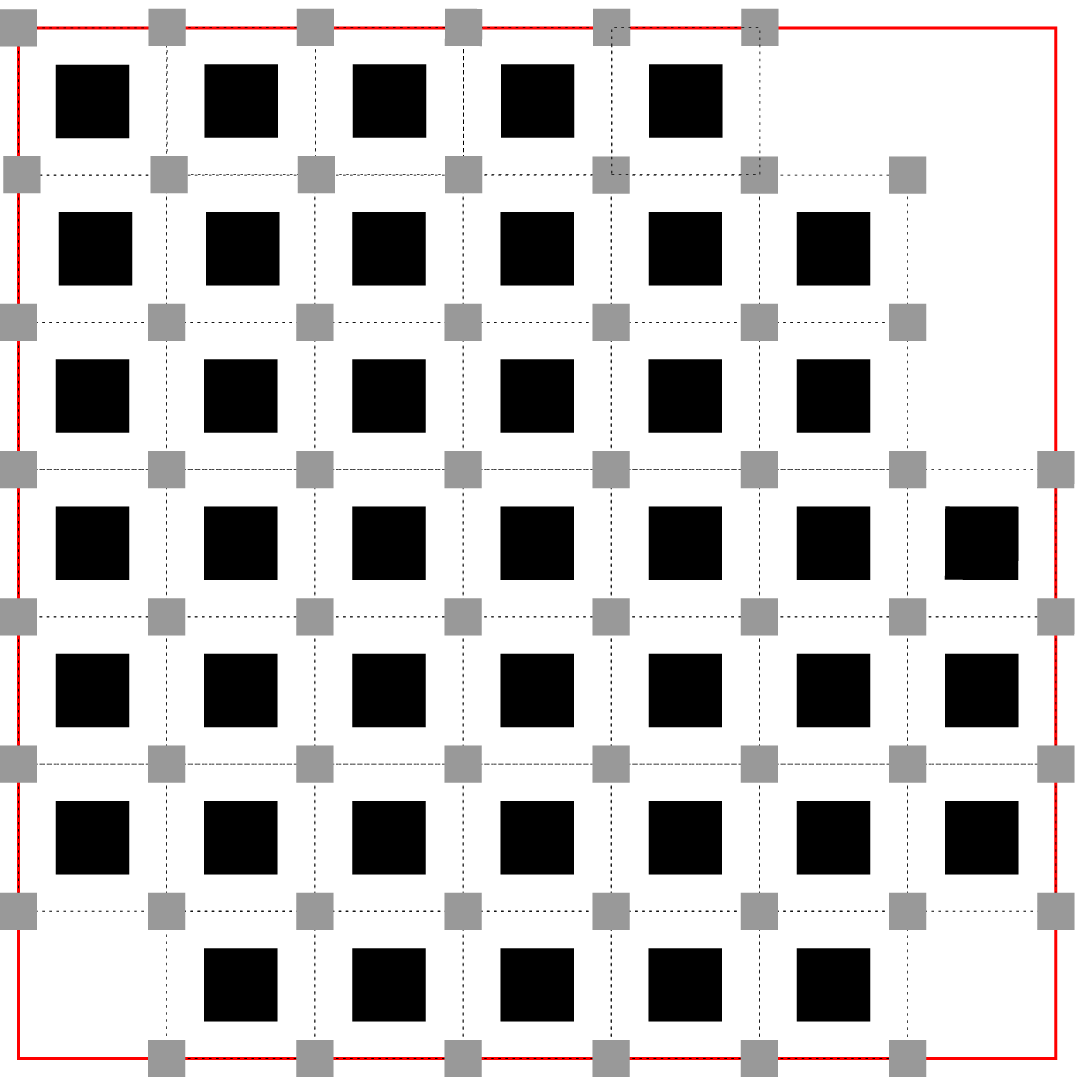}}
\qquad
\subfigure[]
{\includegraphics[height=0.18\textwidth]{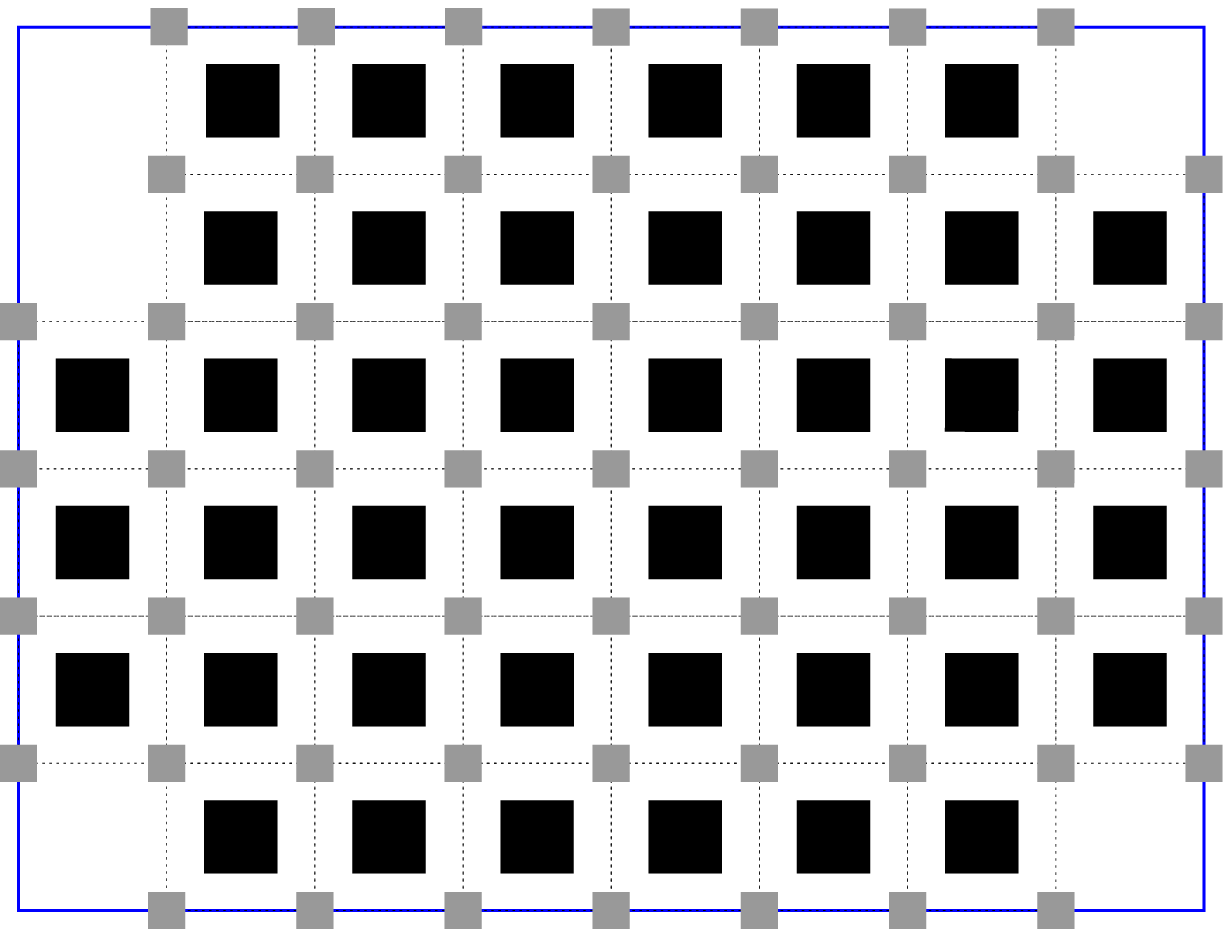}}
\caption{Two examples of configurations in \CB for $\ell\starred = 7$.}
\label{fig:paradigm_class_B}
\end{figure}


\subsubsection{Identification of $g(\{\bar{\eta}\},\minnbis{\protonumber})$ and 
$\bar{g}(\{\bar{\eta}\},\minnbis{\protonumber})$}

There are configurations in $\CB$ that cannot be reached within energy level $\Gamma\starred$. 
These configurations have support near the boundary of $\Lambda$ and for $\Lambda \to \Z^{2}$ 
form a negligible fraction of $\CB$.

\begin{lemma}
\label{lemma-proto-below-gamma-RB}
$g(\{\bar{\eta}\},\minnbis{\protonumber})=\bar{g}(\{\bar{\eta}\},\minnbis{\protonumber}) 
\subseteq \CB$.	
\end{lemma}

\begin{proof}
Note that in $\CB$, for $\xi\in\nbis{\protonumber}$ the following conditions are satisfied
\begin{mytextlikelist}
\item $\broken{\sigma} = 1$: $\extrapart{\sigma} \le 3$;
\item $\broken{\sigma} = 2$: $\extrapart{\sigma} \le 1$.
\item $\broken{\sigma} = 3$: $\extrapart{\sigma} \le 0$.
\end{mytextlikelist}
Observe that $\Db \le 3U + \Da$ throughout region $\RB$, and therefore
Lemma~\ref{lemma-first-2step-from-corner} applies. This means that any heavy-step, completed 
without exceeding energy levl $\Gamma\starred$ and starting from a configuration in 
$\minnbis{\protonumber}$, necessarily involves a corner particle of type $\tb$ that is 
moved along the edge where it shares a bond with a corner particle of type $\ta$.

From Remark~\ref{rem-visit-all-monotone} it follows that all $\btiled$ configurations with 
a monotone dual support, a circumscribed rectangle of side lengths $(\ell\starred,\ell\starred)$ 
or $(\ell\starred-1, \ell\starred+1)$, and with a fixed lower left corner far enough from 
$\partial^{-}\Lambda$, belong to $g(\bar{\eta}, \minnbis{\protonumber})$. It remains to show 
that all other configurations in $\minnbis{\protonumber}$ cannot be reached without exceeding 
energy level $\Gamma\starred$. As in the case of the analogous lemma for region $\RA$, the 
proofs comes in various steps.
	
\step{1} 
Let the first step of a modifying path starting from a configuration in $\CA$ involve the 
particle $p$ of type $\tb$ in the protuberance. If $p$ is not re-attached to the main 
cluster before the next heavy-step is completed, then the path exceeds energy level
$\Gamma\starred$ before reaching a configuration in $\minnbis{\protonumber}$. Therefore 
we may consider a path whose first heavy-step involves the protuberance and whose second 
heavy-step does not. 

\begin{figure}[htbp]
\centering
\subfigure[$\eta_{1}$]{\includegraphics[height=6\tilesize]{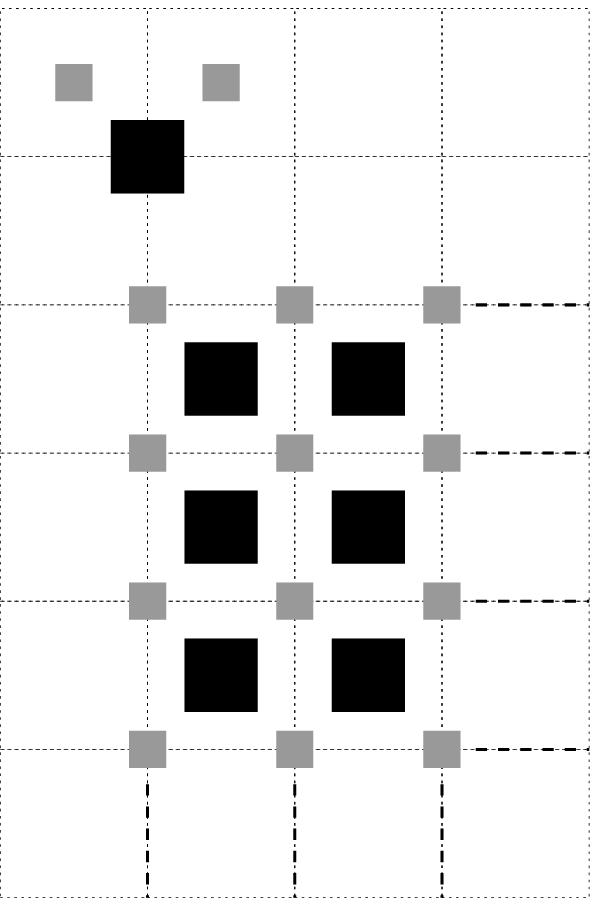}}
\qquad
\subfigure[$\eta_{2}$]{\includegraphics[height=6\tilesize]{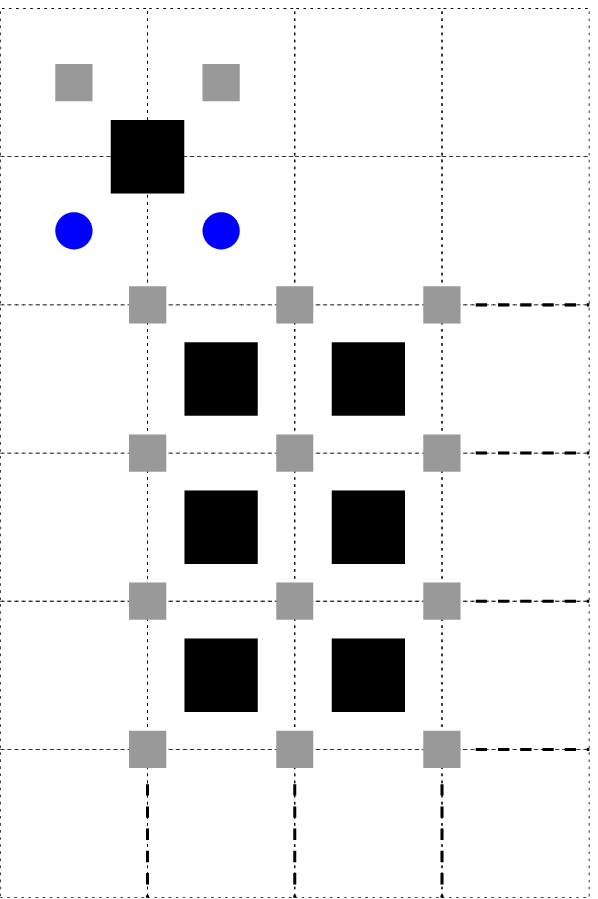}}
\qquad
\subfigure[$\eta_{3}$]{\includegraphics[height=6\tilesize]{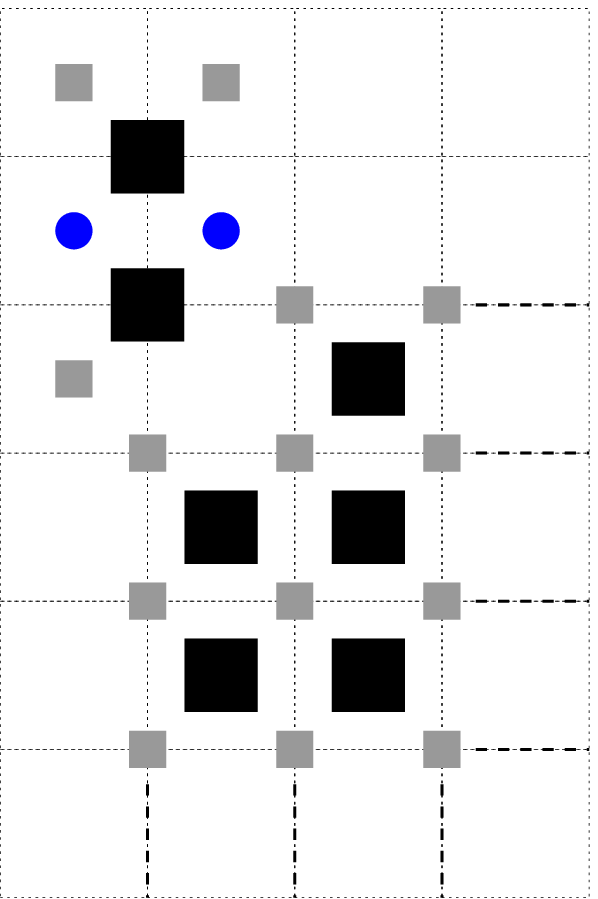}}
\qquad
\subfigure[$\eta_{4}$]{\includegraphics[height=6\tilesize]{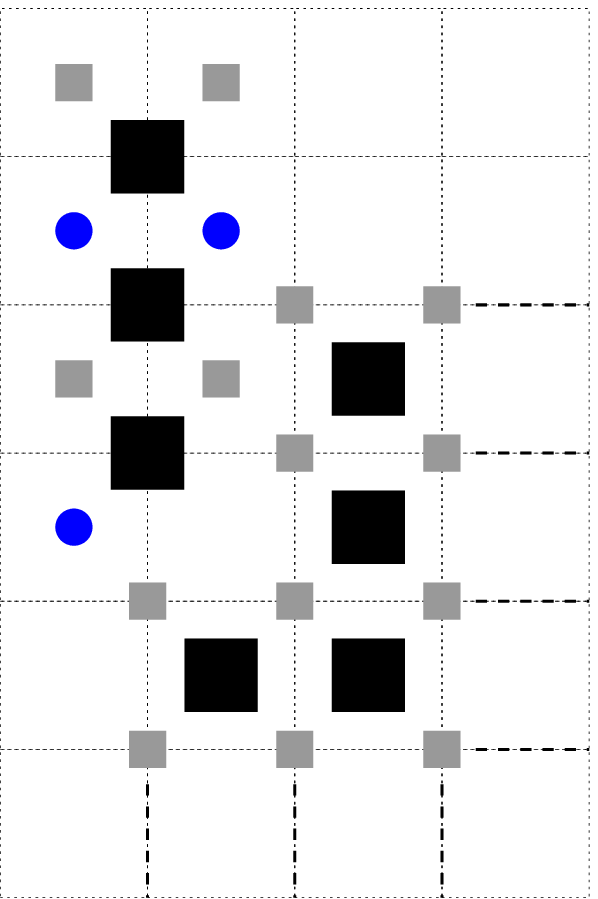}}
\qquad
\subfigure[$\eta_{5}$]{\includegraphics[height=6\tilesize]{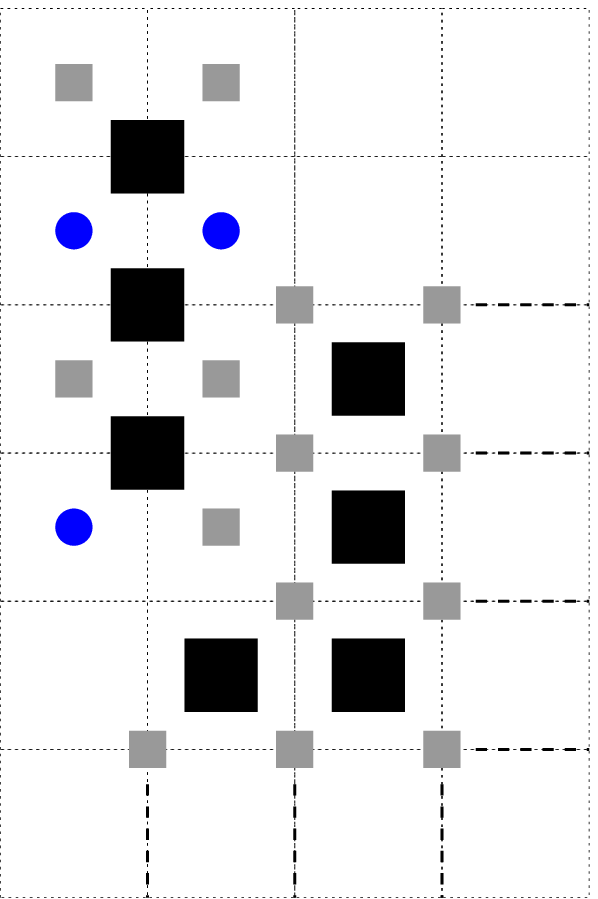}}
\caption{The presence of a detached $\btile$ precludes the motion of other $\btiles$.} 		
\label{fig:reattach-protuberance}
\end{figure}

Let $\bar{\eta}$ be a standard configuration whose protuberance on the North side belongs 
to the Western-most bar, and let $\eta_{1}$ (see Fig.~\ref{fig:reattach-protuberance}(a)) 
be the configuration reached by completing a heavy-step that moves the particle $p_{1}$ in 
the protuberance one step North-West. Since the second heavy-step of the path does not 
involve a motion of $p_{1}$, we may assume that from configuration $\eta_{2}$ (see 
Fig.~\ref{fig:reattach-protuberance}(b)) the path saturates particle $p_{1}$ with two extra 
particles of type $\ta$. It will become clear later on that this choice for $\eta_{0}$ 
and $\eta_{1}$ is the most interesting, since after the next heavy-step is completed the 
particle of type $\tb$ that is moved can share two particles of type $\ta$ with $p_{1}$.

Let $p_{2}$ be the North-West particle of type $\tb$ in the $\btiled$ rectangle of $\eta_{2}$,
$p_{3}$ the particle of type $\tb$ below $p_{2}$, and $p_{4}$ the particle of type $\tb$ East 
of $p_{2}$. It is easy to check that any heavy-step completed by moving any other corner 
particle of type $\tb$ other than $p_{2}$ leads to an energy above $\Gamma\starred$. Since 
$\eta_{2}$ has two extra particle of type $\ta$, the next heavy-step must be completed by 
breaking at most one bond. (Note that in $\eta_{2}$ all external particles of type $\ta$ 
are saturated and hence at least one bond must be broken.) This means that the next particle 
of type $\tb$ must be moved from one good dual corner to another, and the corner must be 
created by breaking exactly one bond. It follows that the only heavy-step from $\eta_{2}$ 
that can be completed is the one obtained by moving $p_{2}$ one step North-West with a 
particle of type $\ta$ sitting South-West of it (configuration $\eta_{3}$; see 
Fig.~\ref{fig:reattach-protuberance}(c)). From $\eta_{3}$ it is not allowed to break 
other bonds. As a consequence, the only heavy-steps that are possible are completed by 
moving $p_{4}$ South-East or $p_{3}$ North-East. These two cases are analogous. We describe 
the second one.
	
When $p_{2}$ is moved North-East, the total number of bonds cannot be decreased. This can 
be achieved by moving one step North-West the particle of type $\ta$ adjacent to $p_{3}$ 
in $\eta_{3}$, and with the help of one extra particle of type $\ta$ reaching the South-West 
of the ``destination'' of $p_{3}$ (configuration $\eta_{4}$; see 
Fig.~\ref{fig:reattach-protuberance}(d)). From $\eta_{4}$ it is not allowed to increase 
the energy. Since in $\eta_{4}$ all particles of type $\tb$ have at least three active bonds, 
the motion of the next particle of type $\tb$ must be from a good dual to another. But from 
$\eta_{4}$ only the motion of corner particles of type $\ta$ is allowed (configuration 
$\eta_{5}$; see Fig.~\ref{fig:reattach-protuberance}(e)), and it is not possible to create 
a good dual corner without bringing a further extra particles of type $\ta$ inside $\Lambda$. 
Hence the path cannot be extended with another heavy-step without exceeding energy level
$\Gamma\starred$. 
	
\step{2}
In this step we will consider the evolution of those paths that visit a $\btiled$ 
configuration having a hanging protuberance. For definiteness, we take $\eta$ to be 
a configuration consisting of a $\btiled$ rectangle plus a hanging protuberance at 
the North-West corner (see Fig.~\ref{fig:possible-reattachment}(b)). Let $p_{1}$ be the 
particle of type $\tb$ in the hanging protuberance, and let $p_{2}$ be the particle of 
type $\tb$ South-East of $p_{1}$. 
	 
If the first heavy-step from $\eta$ is not completed by moving $p_{1}$ in the North-West, 
South-West or North-East direction, then the path exceeds energy level $\Gamma\starred$ 
before reaching a configuration in $\minnbis{\protonumber}$. From Step 3 in the proof of 
Lemma~\ref{lemma-identification-of-entgate-RA}, we already know that it is not possible 
to move $p_{1}$ South-East nor $p_{2}$ North-West. We will investigate what happens when
some other corner particle of type $\tb$ is moved to complete the first heavy-step from 
$\eta$. 

Suppose that the first heavy-step is completed by moving particle $p_{3}$ of type $\tb$ 
in the North-East corner tile of the rectangle (configuration $\eta_{1}$). It is easy to 
see that this move is only possible below energy level $\Gamma\starred$ when, in $\eta_{1}$, 
$p_{3}$ has three adjacent particles of type $\ta$. This can be achieved by bringing 
inside $\Lambda$ two extra particles of type $\ta$ (see Fig.~\ref{fig:hanging-move-corner}). 
From $\eta_{1}$ it is not possible to further increase the energy without exceeding 
energly level $\Gamma\starred$, and it is immediate that no other heavy-step is allowed.

\begin{figure}[htbp]
\centering
{\includegraphics[height=5\tilesize]{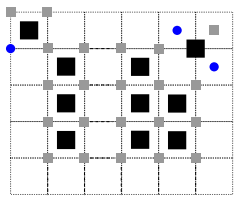}}
\caption{The presence of a hanging protuberance precludes the motion of other
$\btiles$.} 		
\label{fig:hanging-move-corner}
\end{figure}
	
Step 1 and 2 imply the following. A modifying path that does not exceed energy level 
$\Gamma\starred$ and starts with any heavy-step involving the particle of the 
protuberance must ``go back'' to a configuration equivalent to a configuration 
in $\CA$ before some other particle of type $\tb$ can be moved.
	
\step{3} 
Let the first heavy-step from a configuration $\eta_{0} \in \CB$ be completed by moving 
a corner particle $p$ of type $\ta$ belonging to a $\abbar$ of dual length $m \ge 3$.
If $p$ is not re-attached to the main cluster before the next heavy-step is completed, 
then the path exceeds energy level $\Gamma\starred$ before reaching a configuration in 
$\minnbis{\protonumber}$.
	
\begin{figure}[htbp]
\centering
\subfigure[$\eta_{0}$]{\includegraphics[height=5\tilesize]{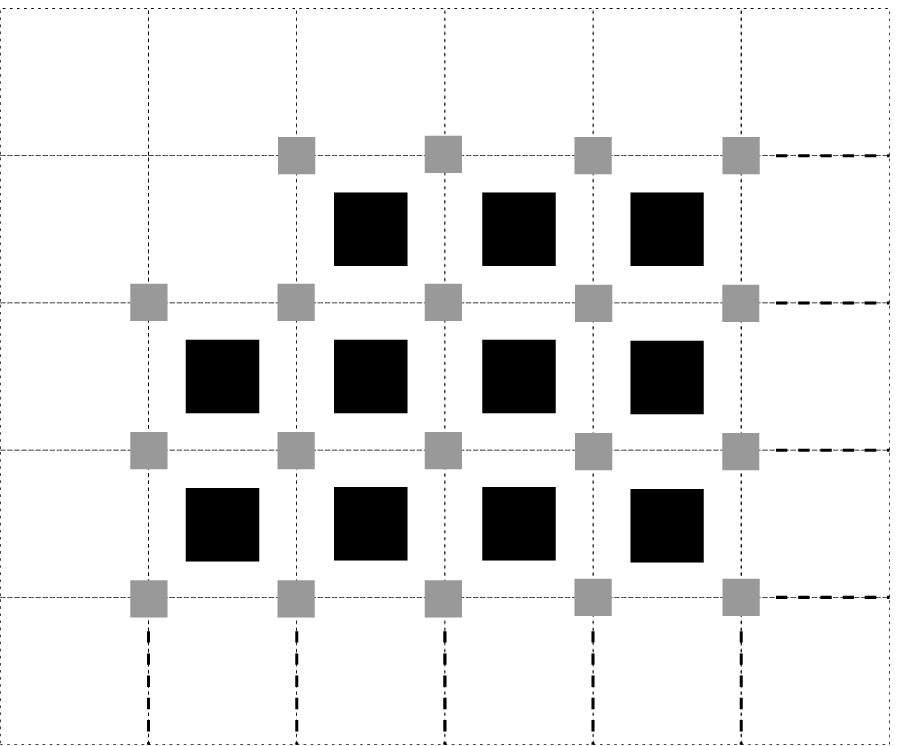}}
\qquad
\subfigure[$\eta_{1}$]{\includegraphics[height=5\tilesize]{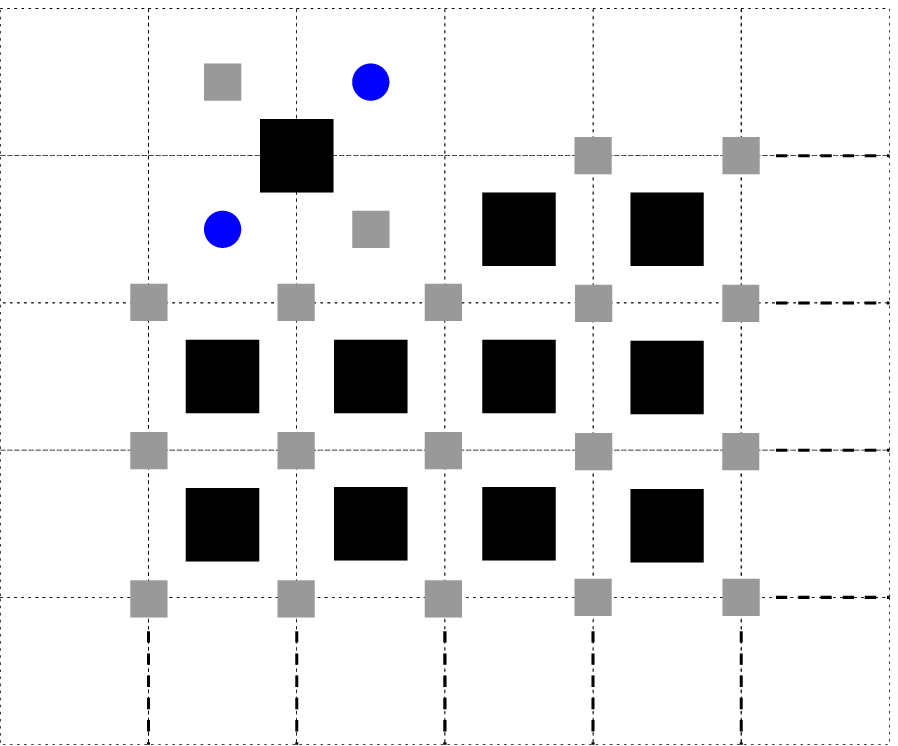}}
\qquad
\subfigure[$\eta_{2}$]{\includegraphics[height=5\tilesize]{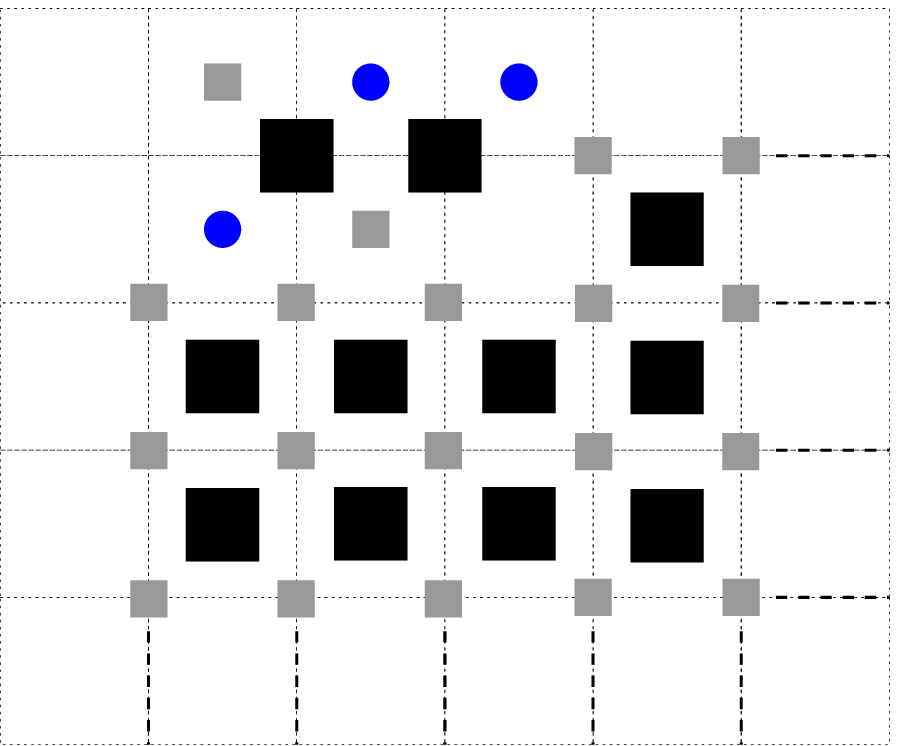}}
\qquad
\subfigure[$\eta_{3}$]{\includegraphics[height=5\tilesize]{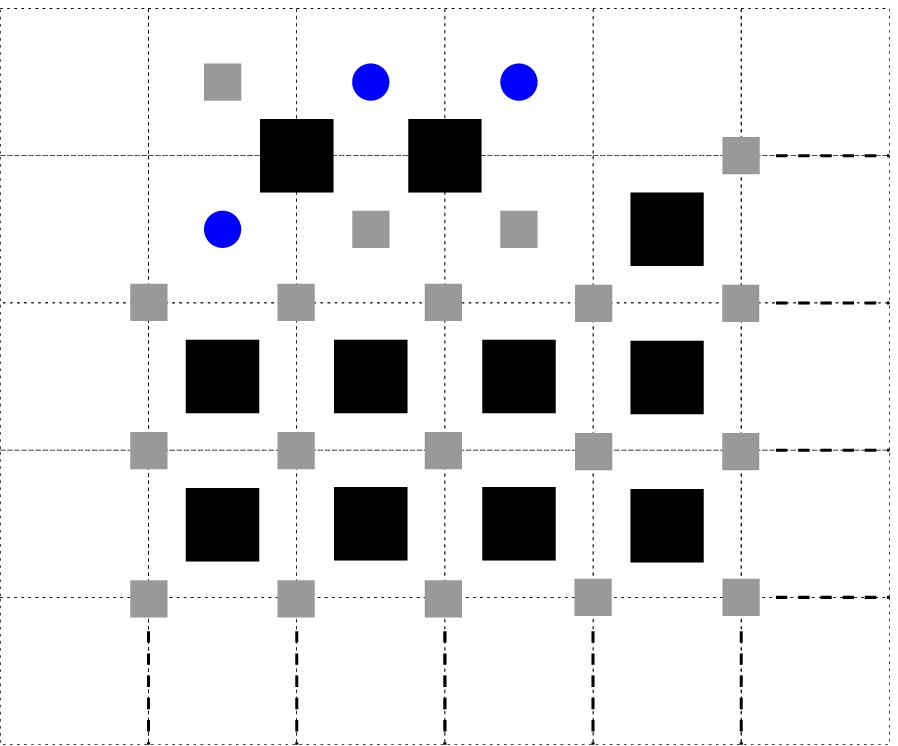}}
\caption{First heavy-step from a bar of length $\ge 3$.} 		
\label{fig:bar3}
\end{figure}

Let $\eta_{0}$ be as in Fig.~\ref{fig:bar3}(a), and let $p_{1}$, $p_{2}$ and 
$p_{3}$ be the first three particles of type $\tb$, starting from the West, of 
the Northern-most $\abbar$. Let the first heavy-step from $\eta_{0}$ be completed by 
moving $p_{1}$ North-West. Since we are assuming that the second heavy-step is not 
completed by moving $p_{1}$, we may consider the path from $\eta_{1}$ that is
obtained by saturating $p_{1}$ with two extra particles of type $\ta$ and by moving 
the particle North-West of $p_{2}$ one step South-West (see Fig.~\ref{fig:bar3}(b)). 
Note that $\eta_{1}$ can be reached within energy barrier $U + 3\Da$.
	
Since from $\eta_{1}$ it is not allowed to break another extra bond, the only possible
heavy-step is the one obtained by moving $p_{2}$ North-West after an extra particle of 
type $\ta$ has entered $\Lambda$ and has reached a site adjacent to the destination 
of $p_{2}$ (configuration $\eta_{2}$; see Fig.~\ref{fig:bar3}(c)). From $\eta_{2}$ it 
is not possible to further increase the energy without exceeding level $\Gamma\starred$, 
and a further heavy-step is therefore not possible. 

\emph{Remark}:
Note that the key observations here are the following. After $p_{1}$ has moved, two 
extra particles of type $\ta$ are required to saturate it. The presence of two extra 
particles of type $\ta$ forces the path to evolve without breaking extra bonds. Since 
all the external particles of type $\tb$ other than $p_{1}$ have at least three active 
bonds, the motion must necessarily be towards a good dual corner (a site with three 
neighbors occupied by a particles of type $\ta$). Note that when a particle is moved, 
it changes its parity and only particles with the same parity can interact. When the 
second particle $p_{2}$ of type $\tb$ is moved, it takes the parity of the particles 
in the tile $p_{1}$ belongs to. But after $p_{2}$ is moved, it can share at most two 
particles of type $\ta$ with $p_{1}$, and hence another particle of type $\ta$ is 
needed. This leads to a configuration with 3 extra particles of type $\ta$ and one broken 
bond, and hence it is not allowed to make moves that increase the energy. In particular, 
it is not allowed to bring inside $\Lambda$ other particles of type $\ta$, and only 
particles of type $\ta$ with one active bond can be moved. Note that, again, all particles 
of type $\tb$ have at least three active bonds. Hence, a further heavy-step would only 
be possible if a good dual corner can be created close to a particle of type $\tb$ 
without decreasing the number of bonds. These observations are key in order to explore 
what configurations can be reached by a modifying paths without exceeding $\Gamma\starred$.	
	
\step{4} 
Let the first heavy-step from a configuration $\eta_{0} \in \CB$ be completed by moving 
a corner particle $p_{1}$ of type $\tb$ belonging to a $\abbar$ of dual length $m = 2$.
If the path does not exceed energy level $\Gamma\starred$, then one of the following 
must happen:
\begin{itemize}
\item 
Particle $p$ is re-attached to the cluster before any other particle of type $\tb$ is 
moved.
\item 
If $p_{1}$ denotes the other particle of type $\tb$ in the bar, then the path reaches a 
configuration $\eta_{d}$ where $p_{1}$ and $p_{2}$ are saturated (with the help of 
three extra particles of type $\tb$), belong to the same cluster and are at dual 
distance $\sqrt{2}$ from its location in $\eta_{0}$ (see Fig.~\ref{fig:bar2}(c)).
From $\eta_{d}$, the first heavy-step (without backtracking) must be completed
by re-attaching $p_{1}$ to the cluster. The configuration that is reached in this way
is again a configuration that can be reached with a single heavy-step completed by 
moving a corner particle of type $\tb$ belonging to a $\abbar$ of dual length $m = 2$
(see Fig.~\ref{fig:bar2}(d)).
\end{itemize}

\begin{figure}[htbp]
\centering
\subfigure[$\eta_{0}$]{\includegraphics[height=5\tilesize]{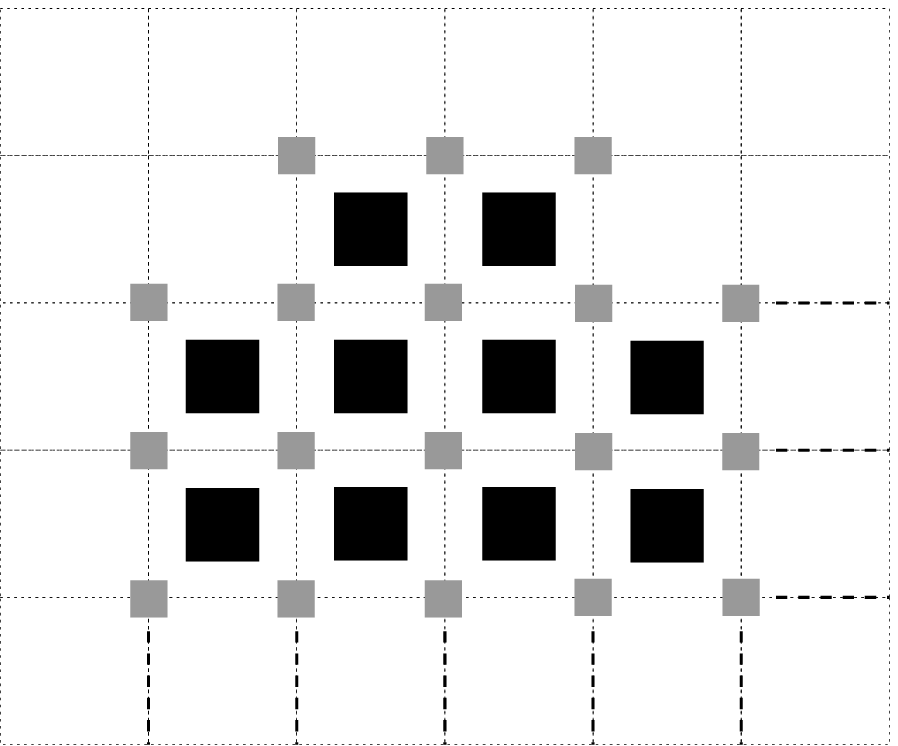}}
\qquad
\subfigure[$\eta_{1}$]{\includegraphics[height=5\tilesize]{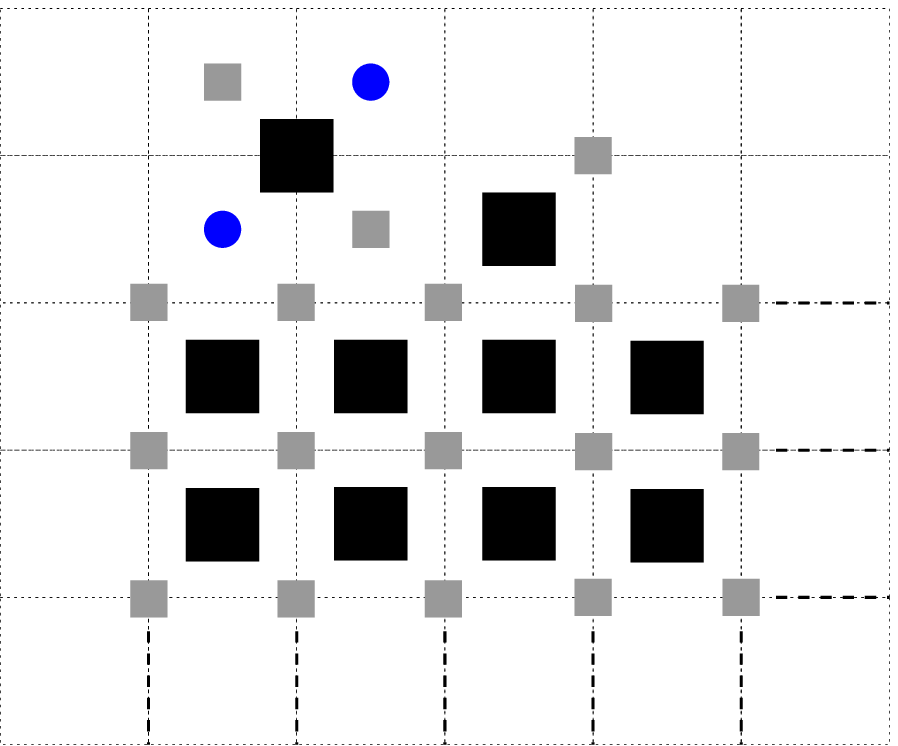}}
\qquad
\subfigure[$\eta_{d}$]{\includegraphics[height=5\tilesize]{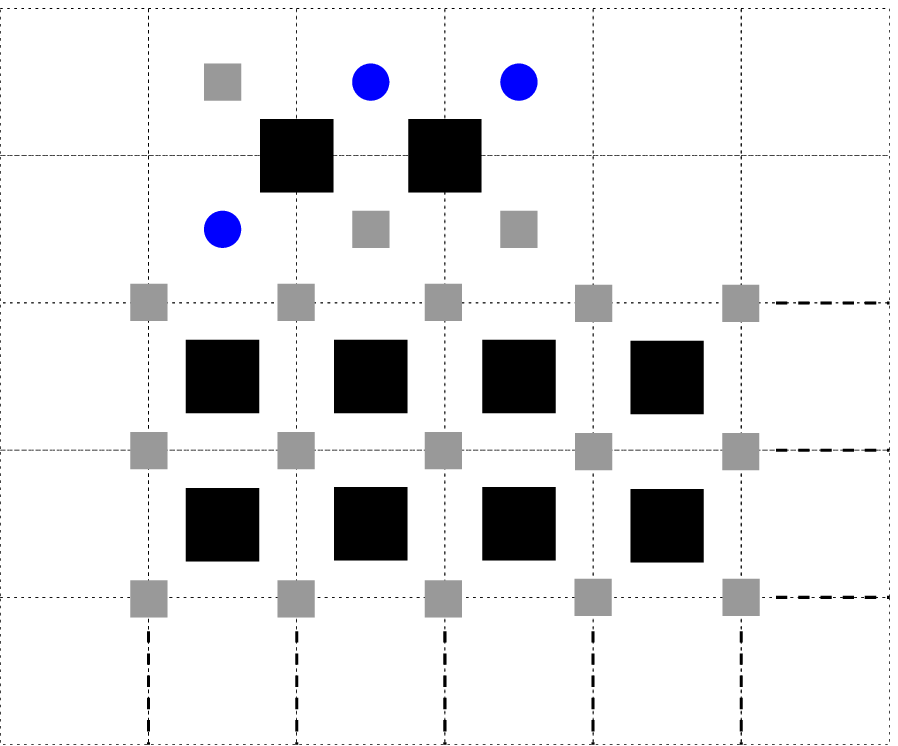}}
\qquad
\subfigure[$\eta_{2}$]{\includegraphics[height=5\tilesize]{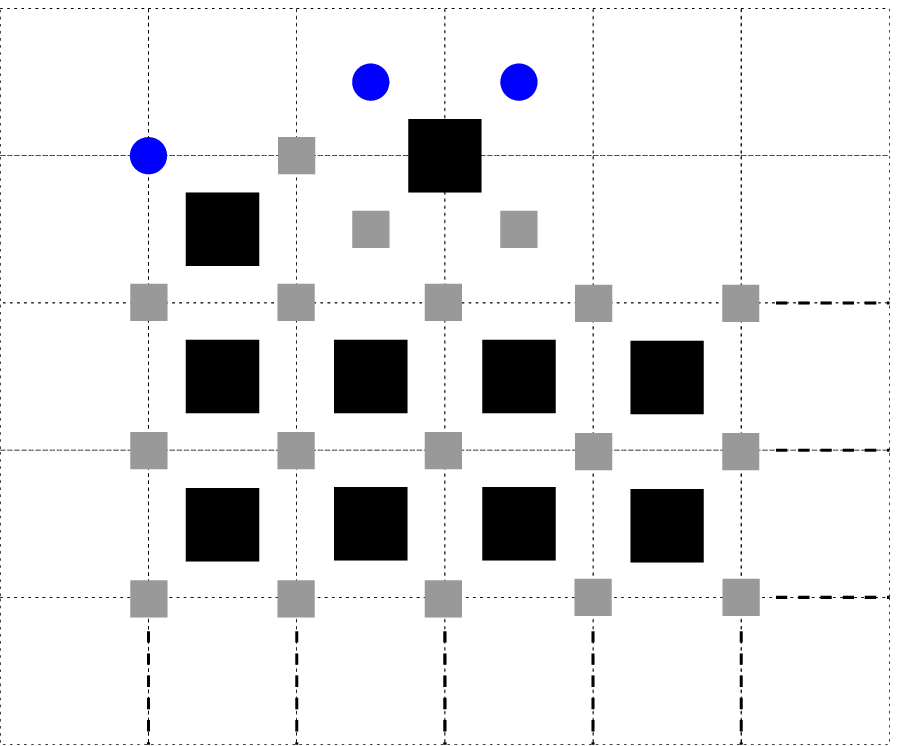}}
\caption{First heavy-step from a bar of length $2$.} 		
\label{fig:bar2}
\end{figure}

Let $p_{1}$ be the Western-most particle of type $\tb$ in the bar of length $2$ in 
configuration $\eta_{0}$, and $p_{2}$ the particle of type $\ta$ East of $p_{1}$ 
(see Fig.\ref{fig:bar2}(a)). Let the first heavy-step from $\eta_{0}$ be completed 
by moving $p_{1}$ North-West. Suppose that $p_{1}$ is not re-attached to the cluster 
before the next heavy-step is completed. As in Step 3, we can consider the path 
from $\eta_{1}$ obtained by saturating $p_{1}$ with two extra particles of type $\ta$ 
and by moving the particle North-West of $p_{2}$ one step South-West (see
Fig.~\ref{fig:bar2}(b)). Note that $\eta_{1}$ can be reached within energy barrier 
$U + 3\Da$.
	
Since from $\eta_{1}$ it is not allowed to break another bond, the only possible
heavy-step is the one obtained by moving $p_{2}$ North-West after an extra particle of 
type $\ta$ has entered $\Lambda$ and has reached a site adjacent to the destination 
of $p_{2}$. After that, since it is not possible to further increase the energy, before 
the next heavy-step is completed it is necessary to move a particle of type $\ta$ to 
the empty site adjacent to $p_{2}$, reaching configuration $\eta_{d}$. The first 
heavy-step from $\eta_{d}$, since it contains three extra particles of type $\ta$, must 
be completed breaking at most one bond. This means that a particle of type $\ta$ must 
be moved from a good dual corner to another. This is only possible, without backtracking, 
by moving $p_{1}$ one step South-West to site $x$ after the particle of type $\ta$ sitting 
at $x$ in $\eta_{d}$ has been moved one step North-West. Particle $p_{1}$ can be saturated 
with a free particle of type $\ta$, to reach configuration $\eta_{2}$ of Fig.~\ref{fig:bar2}(d). 
Note that $\eta_{2}$ is the ``mirror image'' of $\eta_{1}$, and hence the same arguments 
can be repeated.
	
For further reference, let us consider also a configuration with a vertical $\abbar$ of 
length 2. Such a configuration $\eta_{v}$ does not belong to $\CB$, but could be reached 
by a modifying path below energy level $\Gamma\starred$ (see Fig.~\ref{fig:prot-on-prot}(a)).

\begin{figure}[htbp]
\centering
\subfigure[$\eta_{v}$]{\includegraphics[height=5\tilesize]{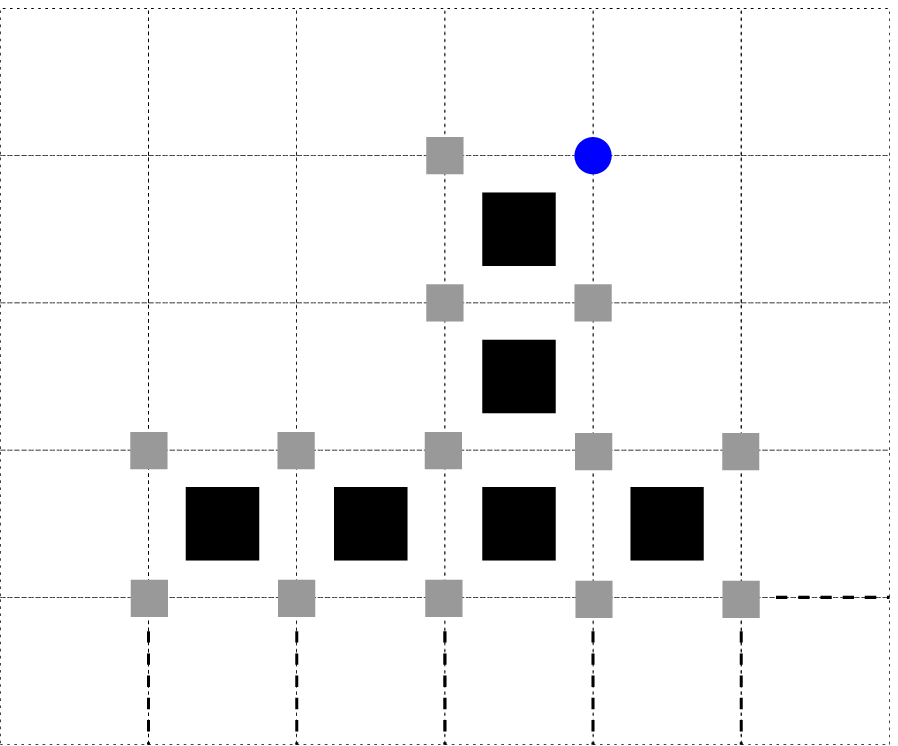}}
\qquad
\subfigure[$\eta_{w}$]{\includegraphics[height=5\tilesize]{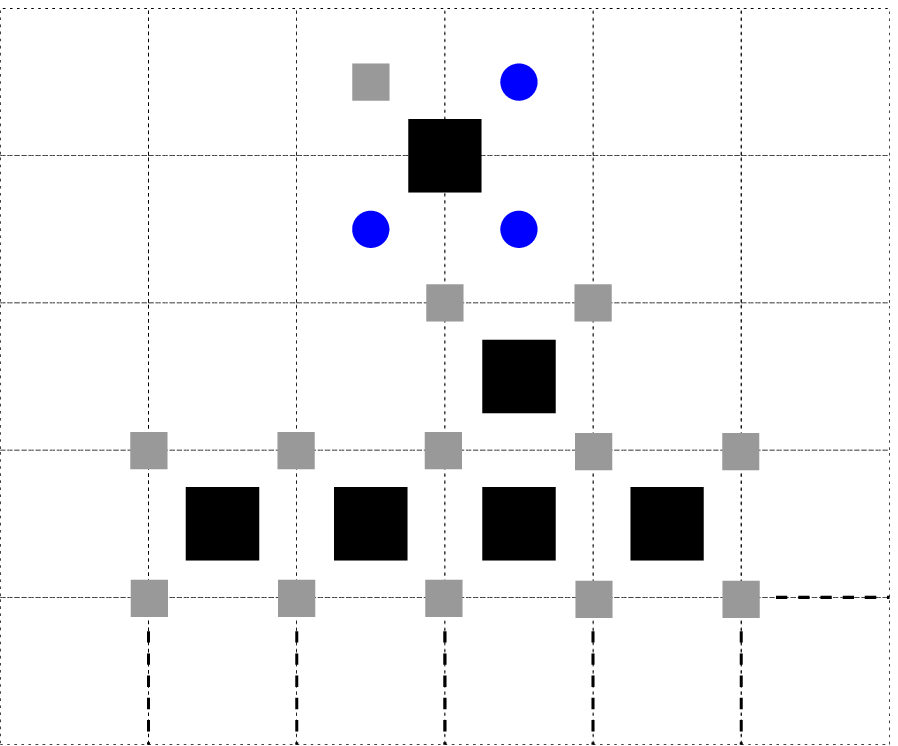}}
\qquad
\subfigure[$\eta_{x}$]{\includegraphics[height=5\tilesize]{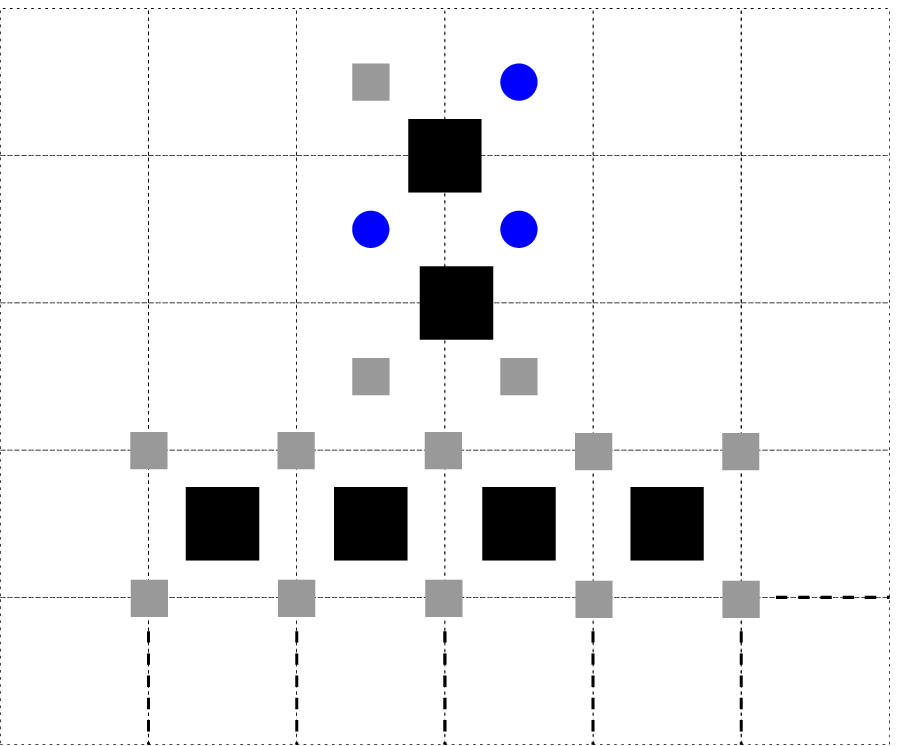}}
\qquad
\subfigure[$\eta_{y}$]{\includegraphics[height=5\tilesize]{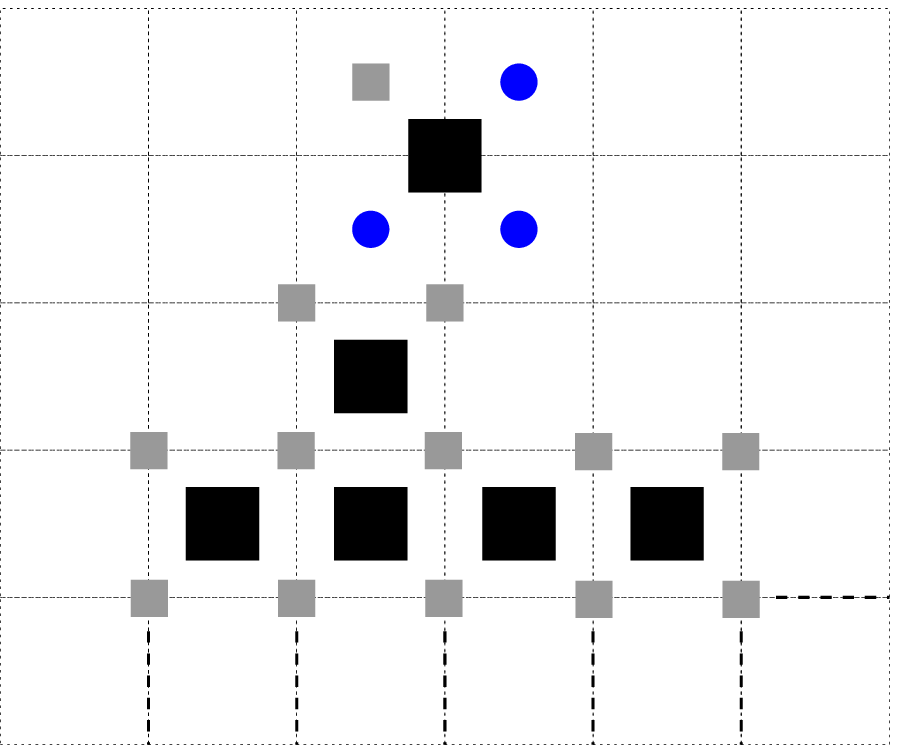}}
\caption{Protuberances attached to $\abbars$ can be treated as $\abbars$ of length 2.} 		
\label{fig:prot-on-prot}
\end{figure}

The analysis is completely analogous to the case of a horizontal $\abbar$ of length 2, 
and the claim is that if the first heavy-step is completed by moving the particle $p_{1}$ 
of type $\tb$ on top of the vertical $\abbar$ and the path does not exceed the energy 
level $\Gamma\starred$, then either $p_{1}$ is re-attached to the cluster before the next 
heavy-step is completed, or after $p_{1}$ has been saturated (configuration $\eta_{w}$; 
see Fig.~\ref{fig:prot-on-prot}(b)) the next heavy-step must be completed by moving the 
other particle $p_{2}$ of type $\tb$ originally in the $\abbar$ to reach a $\btiled$ 
configuration $\eta_{x}$ with three extra particles of type $\ta$ where the $\abbar$ 
is ``floating'' on a side of the cluster (see Fig.~\ref{fig:prot-on-prot}(c)). From 
$\eta_{x}$, the only heavy-step that can be completed is the one obtained by re-attaching 
$p_{2}$ to the cluster. Then $p_{2}$ can be saturated to obtain configuration $\eta_{y}$ 
(see Fig.~\ref{fig:prot-on-prot}(d)). Note that $\eta_{y}$ is the mirror image of 
$\eta_{w}$ and the same argument can be repeated. Note that when $p_{1}$ is re-attached, 
either a configuration with again a vertical $\abbar$ of length $2$ is reached, or the 
path visits a configuration where $p_{1}$, after being saturated, belongs to a hanging 
protuberance. We will see below that in this case the next heavy-step must be completed 
by moving $p_{1}$ again. 
	
Thus, after the first heavy-step from a configuration $\eta_{0}$ in $\CB$ is completed 
by moving a particle $p_{1}$ of type $\tb$, this particle must be re-attached to the 
cluster. This particle can be saturated to reach a configuration $\eta_{1}$ in which
$p_{1}$ belongs to a (possibly hanging) protuberance. Note that $\eta_{1}$ contains one 
extra particle of type $\ta$, and therefore $H(\eta_{1}) = H(\eta_{0}) + \Da$. Note
that in the case $p_{1}$ belonged to a bar of length $\tb$, and configuration $\eta_{1}$ 
contains two protuberances.
	
We claim that, from $\eta_{1}$, the first heavy-step must be completed by moving a particle 
of type $\ta$ in one of the protuberances. This can be seen as follows.

\begin{figure}[htbp]
\centering
\subfigure[$\eta_{1}$]{\includegraphics[height=5\tilesize]{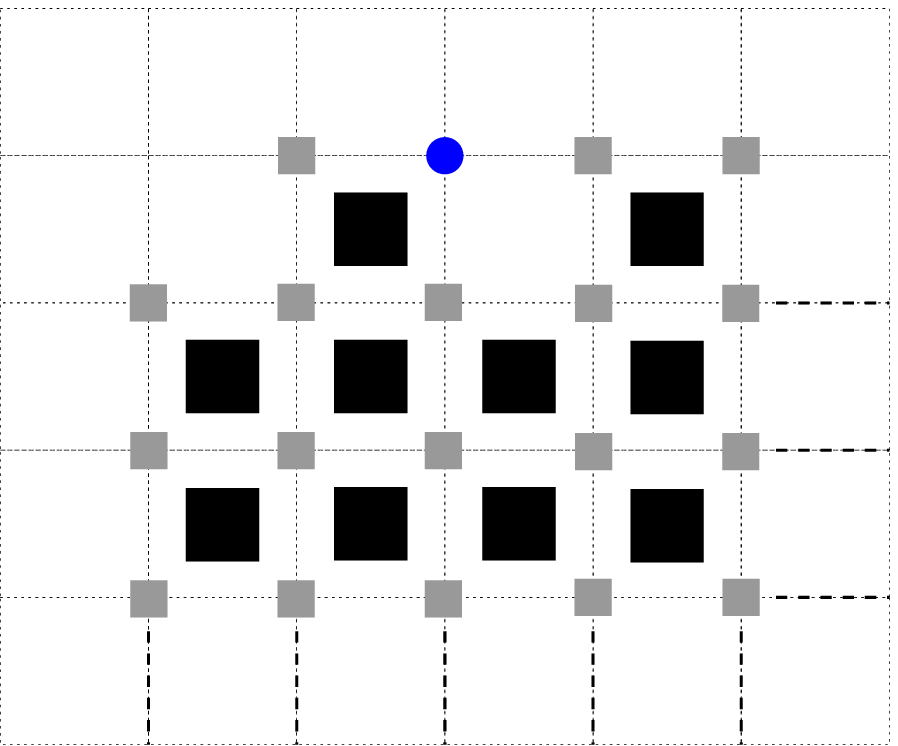}}
\qquad
\subfigure[$\eta_{2}$]{\includegraphics[height=5\tilesize]{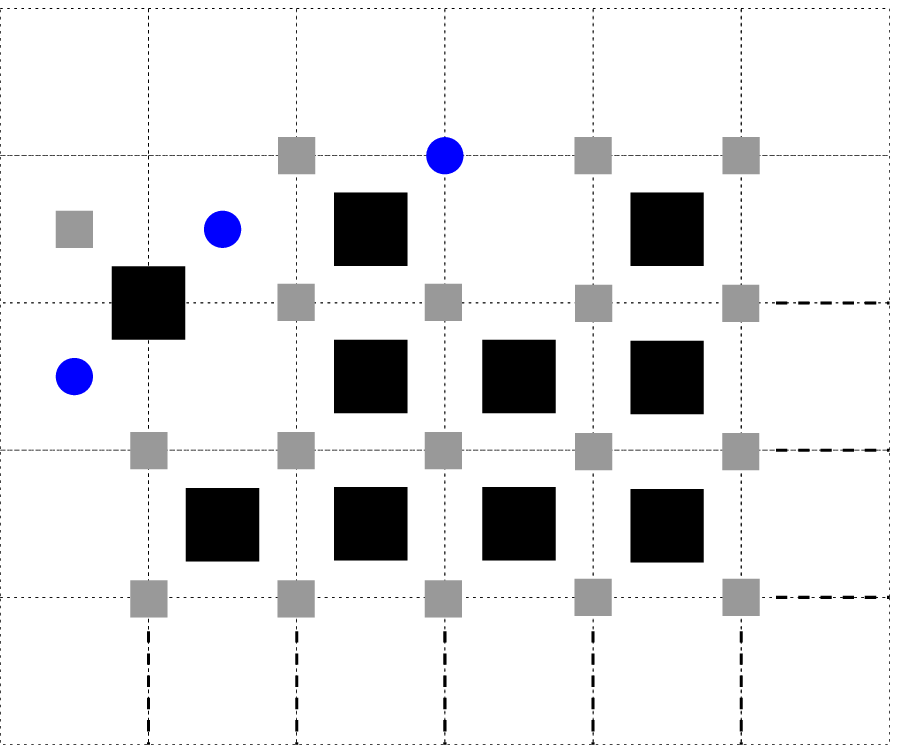}}
\qquad
\subfigure[$\eta_{2}\prm$]{\includegraphics[height=5\tilesize]{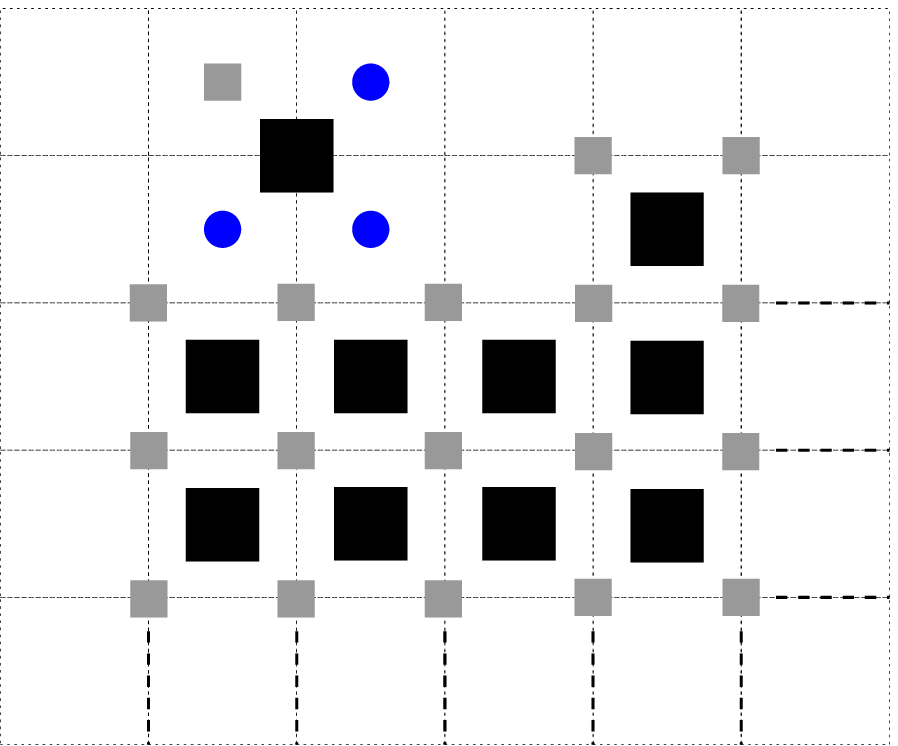}}
\caption{Heavy-step from a configuration with protuberances non in $\CA$} 		
\label{fig:protuberances}
\end{figure}

\medskip\noindent
{\bf 1.}
Let $\eta_{1}$ be a configuration like the one in Fig.~\ref{fig:protuberances}(a), and 
let the first heavy-step from $\eta_{1}$ be completed by moving the particle $p_{2}$ of 
type $\tb$ sitting at the Northern-most site of the West bar of $\eta_{1}$ one step 
North-West. Let $\eta_{2}$ be the configuration that is reached when this heavy-step 
is completed (see Fig.~\ref{fig:protuberances}(b)). Since $\eta_{p}$ has already one 
extra particle of type $\ta$, it is not possible to have $B(p_{2}, \eta_{2}) = 1$
(i.e., three broken bonds). Similarly, if we consider the case $B(p_{2},\eta_{2}) 
= 2$, then we must require one extra particle of type $\ta$, but this is incompatible 
with a configuration with two broken bonds. Therefore the only possibility is $B(p_{2},
\eta_{2}) = 1$, where the three bonds of $p_{1}$ in $\eta_{2}$ are obtained with 
the help of two extra particles of type $\ta$. From $\eta_{2}$ it is not possible to 
increase the energy further. Since in $\eta_{2}$ all configurations have at least three 
active bonds, the next heavy-step must be completed by moving a particle of type $\tb$ 
from a good dual corner	to another without increasing the number of broken bonds. But 
this is impossible from configuration $\eta_{2}$.

\medskip\noindent
{\bf 2.}
Clearly the same conclusion can be reached also if $p_{1}$ belongs to a hanging 
protuberance (see Step 2). Note that, by completing a heavy-step from $\eta_{1}$ 
by moving $p_{1}$, the configuration that is reached is either one heavy-step away 
from being attached in a good dual corner to form a bar of length $m \ge 2$ (and 
we know that it must be re-attached by Step 3 and 4), or it is far from good dual 
corners. In this case, if we assume that the next heavy-step is not completed by 
moving again $p_{1}$, then we can consider the path from the configuration 
$\eta_{2}\prm$ obtained by saturating $p_{1}$ (see Fig.\ref{fig:protuberances}(c)). 
This configuration has three extra particles of type $\ta$, and hence no heavy-step 
can be completed without exceeding energy level $\Gamma\starred$ by moving a particle 
of type $\tb$ different from $p_{1}$. It is straightforward to see in the next heavy-step 
$p_{1}$ must indeed be re-attached to the cluster.

\step{5}	
It follows from Steps 1--4 that the set of single cluster configurations (and, consequently, 
the set of configurations in $\minnbis{\protonumber}$) that can be visited by a modifying 
path that does not exceed energy level $\Gamma\starred$ coincides with the set of 
configurations that can be reached by a modifying path whose configurations do not have 
more than one particle of type $\tb$ not belonging to the main cluster such that if a 
configuration has a particle of type $\tb$ that is not connected to the cluster, 
then this particle is re-attached to the cluster in the next heavy-step, and afterwards 
is saturated. This means that the configurations in $\minnbis{\protonumber}$ that are 
reached without exceeding energy level $\Gamma\starred$ are obtained by iteratively moving 
a corner $\btile$ around the cluster (with the help of one or two extra particles of type 
$\ta$ used to saturate the particle of type $\tb$). We already saw  in Steps 1--2 that a 
modifying path starting from a configuration in $\CA$ with a heavy-step involving the 
particles of type $\tb$ in the protuberance cannot leave the set $\CA$. In the other cases, 
if the $\btile$ that is moved reaches a corner, then a new configuration in $\CB$ is 
reached. Otherwise, a configuration with one or two protuberances not belonging to $\CA$ 
is reached. But in this case, the path can only proceed by moving one of the protuberances, 
which eventually reach a corner and again producing a configuration in $\CB$. 

\step{6}
From what has been seen so far it follows that the set of single cluster $\btiled$ 
configurations that can be visited by a modifying path starting from a standard 
configuration $\bar{\eta}$ without exceeding energy level $\Gamma\starred$ consists 
of those configurations that can be reached by either moving the protuberance of a 
configuration in $\CA$, or by iteratively moving corner $\btile$ to some other corner 
possibly created by adding an extra particle of type $\ta$. This observation implies 
that $\gbar \subset \CB$. Furthermore, if $\eta\in\CB$ and has support far enough from 
$\partial^{-}\Lambda$, then $\eta \in \g$. It is straightforward to see that lattice 
distance $2$ from the annulus where no interaction is present is already far enough. 
In order to prove that $\g = \gbar$, we will show that the set $\gbar \backslash \g$ 
indeed is empty.

\medskip\noindent
\emph{Remark}:	
Note that variations in the energy are only possible when a particle of type $\ta$ enters 
or leaves $\Lambda$ or when the number of active bonds changes as a consequence of the 
motion of a particle inside $\Lambda$. This implies that, for all $\eta \in \CB$, 
$\comlev(\bar{\eta}, \eta)$ for some standard configuration $\bar{\eta}\in\nbis{\protonumber}$
can take only a discrete set of values. In region $\RB$, $\comlev(\bar{\eta}, \eta) 
= \Gamma\starred$ can only happen when $\comlev(\bar{\eta}, \eta) = H(\bar{\eta}) + U + 3\Da$ 
and $\Db = U + 3\Da$.  

\medskip	
Consider the set of $\btiled$ configurations consisting of a single cluster that can be
reached with the moves considered at the various stages of the previous analysis. Clearly,
this set contains $\bar{g}(\bar{\eta}, \minnbis{\protonumber})$. Let $\eta\prm$ belong 
to this set, and let $\omega:\bar{\eta}\to\eta\prm$. Assume that $\max_{\xi\in\omega} 
H(\xi) = U + 3\Da$ and let $\supp(\omega) = \cup_{\xi \in \omega} \supp(\xi)$. Then, 
from the previous analysis, it follows that there is a path $\omega\prm:\bar{\eta} 
\to \eta\prm$ such that $\max{\xi \in \omega\prm} H(\xi) =  H(\bar{\eta} + 3U$
and $\supp(\omega\prm) \subset \supp(\omega)$. In words, if a configuration in $\CB$ 
can be reached within energy barrier $U + 3\Da$, then it can also be reached within 
energy barrier $3U$, irrespective of the distance from the boundary of $\Lambda$.
\end{proof}	

\begin{remark}
\label{rem-filling-circumscribed-rectangle}
Let $\eta$ be a $\btiled$ conconfiguration of minimal energy with a fixed number of 
particles of type $\tb$, and let $R(\eta)$ be the rectangle circumscribing its dual 
support. Then the $\btiled$ configuration $\eta\prm$ with dual support equal to 
$R(\eta)$ can be obtained by iteratively bringing a particle of type $\tb$ to a good 
dual corner and saturating it with a particle of type $\ta$ within energy barrier $\Db$. 
Clearly, $H(\eta\prm) < H(\eta)$. We say that $\eta\prm$ is obtained by filling the 
rectangle circumscribing the dual support of $\eta$.
\end{remark}


\subsubsection{Existence of a good site}

\begin{lemma}
\label{lemma-proto-is-good-RB}
For all $\hat{\eta} \in \CB$, there exists an $x \in F(\hat{\eta})$ such that 
$\comlev((\hat{\eta},x),\boxplus) < \Gamma\starred$.
\end{lemma}

\begin{proof}
The dual support of $\hat{\eta} \in \CB$ is either a square of side length $\ell\starred$
or a rectangle of side lengths $\ell\starred + 1, \ell\starred - 1$. Let $x$ be a site 
in a good dual corner of $\hat{\eta}$. After a particle of type $\tb$ has reached site 
$x$ ($H((\hat{\eta}, x)), \Gamma\starred - 3U$), it is possible to saturate this particle 
with an extra particle of type $\ta$ within energy barrier $\Da$, reaching the configuration 
$\eta\prm$ with energy $H(\eta\prm) = \Gamma\starred - \Db - \epsi$. Note that the dual 
support of $\eta\prm$ has the same circumscribing rectangle as $\hat{\eta}$. If the 
rectangle circumscribing the dual support of $\eta\prm$ is a square, then by filling
this rectangle we obtain a dual $\btiled$ square of side length $\ell\starred$. By 
Remark~\ref{remark-supercritical-square} this is enough. If, on the other hand, the 
rectangle circumscribing the dual support of $\eta\prm$ has side length $\ell\starred + 1,
\ell\starred - 1$, then by filling this rectangle we obtain a dual $\btiled$ rectangle 
of side lengths $\ell\starred + 1, \ell\starred - 1$. From this point on, it is possible 
to argue as in the final part of the proof of Lemma~\ref{lemma-proto-is-good-RA}.
\end{proof}


\subsubsection{Identification of $\proto$ and $\entgate$}

\begin{lemma}
\label{lemma-identification-of-entgate-RB}
$\proto = \CB$ and $\entgate = \CB\boub$.
\end{lemma}

\begin{proof}
Same as the proof of Lemma~\ref{lemma-identification-of-entgate-RA}.
\end{proof}


\subsection{Region $\RC$: proof of Theorem \ref{th-RC}}\label{sec-RC}

Let $\CC$ be the set of $\btiled$ configurations with $\protonumber$ particles of 
type $\tb$ whose dual tile support is a monotone polyomino and whose circumscribing 
rectangle has perimeter $4\ell\starred$ (see Fig.~\ref{fig:paradigm_class_F}).

\begin{figure}[htbp]
\centering
{\includegraphics[height=0.15\textwidth]{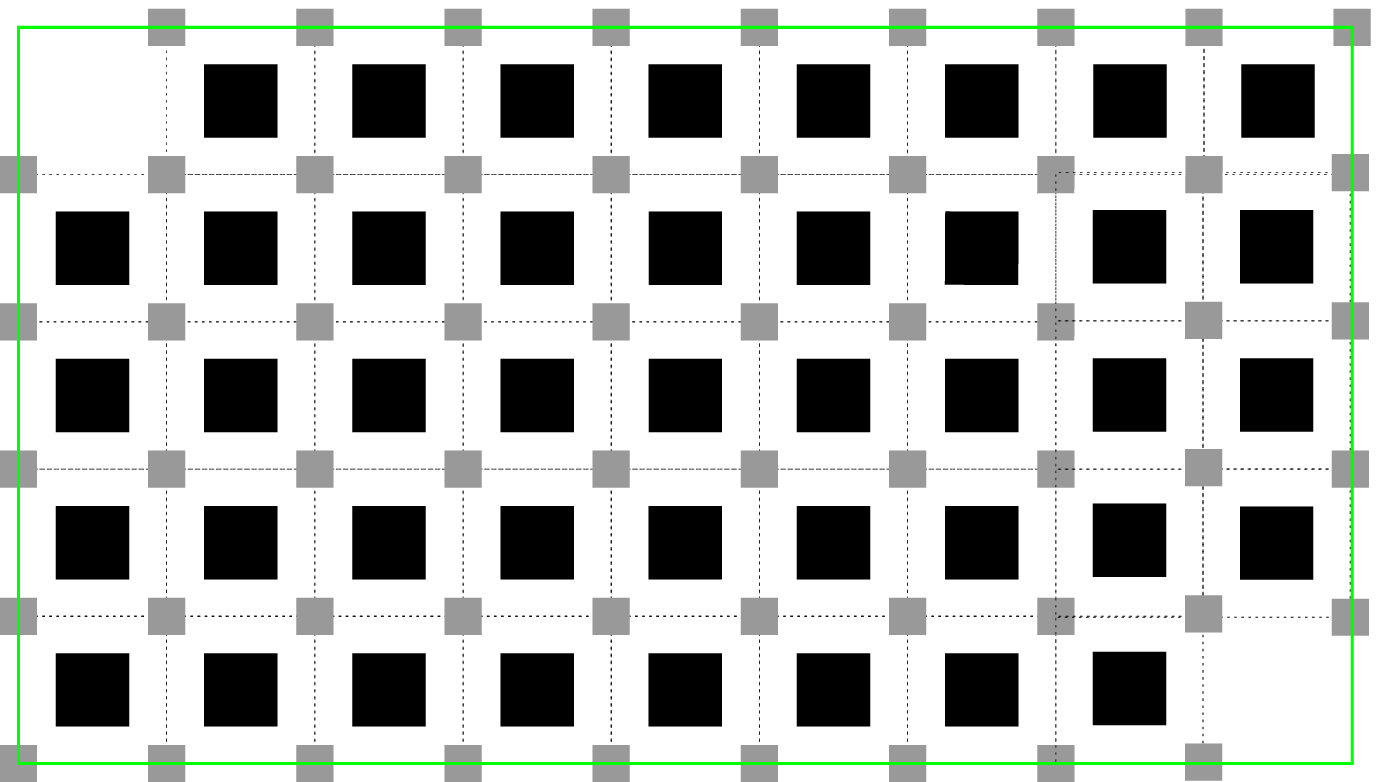}}
\caption{An example of configuration in \CC for $\ell\starred = 7$.}
\label{fig:paradigm_class_F}
\end{figure}


\subsubsection{Identification of $g(\{\bar{\eta}\},\minnbis{\protonumber})$ and 
$\bar{g}(\{\bar{\eta}\},\minnbis{\protonumber})$}

\begin{lemma}
\label{lemma-proto-below-gamma-RF}
$\g = \gbar = \CC$.	
\end{lemma}

\begin{proof}
First observe that the sets $\minnbis{\protonumber}$ and $\CC$ coincide by
Lemma~\ref{lemma-minimal-energy-coincides-with-minimal-perimeter}. Therefore it 
remains to prove that for any configuration $\eta \in \CC$ there exist a standard
configuration $\bar{\eta}$ and a path $\omega: \bar{\eta} \to \eta$ such that 
$\max_{\xi \in \omega} H(\xi) < \Gamma\starred$ or, equivalently, a path
$\omega\prm\colon\,\eta \to \bar{\eta}$ such that $\max_{\xi \in \omega\prm} 
H(\xi) < \Gamma\starred$.

\medskip\noindent
{\bf 1.} 
We know from Remark~\ref{rem-visit-all-monotone} that, starting from a $\btiled$
configuration $\eta\prm \in \minnbis{\protonumber}$ whose dual tile support has a 
circumscribing rectangle of side lengths $L,l$ with $L \ge l$ and $L\times l 
> \protonumber$, it is possible to reach below energy level $\Gamma\starred$ 
all configurations in $\minnbis{\protonumber}$ whose dual support has the same 
circumscribing rectangle (provided that the cluster is sufficiently far from 
$\partial^{-}\Lambda)$. We next show that, from $\eta\prm$, it is also possible 
to reach a configuration $\eta\dprm \in \minnbis{\protonumber}$ whose dual support 
has a circumscribing rectangle with side lengths $L+1, l-1$ (whenever $(L+1) \times 
(l-1) \ge \protonumber$).

\medskip\noindent	
{\bf 2.} 
From $\eta\prm$ it is possible to reach below energy level $\Gamma\starred$ a 
configuration $\tilde{\eta}$ whose tile support, in dual coordinates, is a rectangle 
of side lengths $L-1,l$ plus a bar of length $k = \protonumber - (L-1) l$ on top of 
the longest side of the rectangle. There are two cases. Either $k > l$ or $k \le l$.
If $k > l$, then $(L+1)(l-1) < \protonumber$. If $k \le l$, then we will show that 
it is possible to obtain within energy barrier $\Db$ the configuration $\eta\dprm$ 
whose dual tile support is obtained from the dual tile support of $\tilde{\eta}$
by moving the bar of length $k$ from the top of the rectangle to one of its sides
as follows. 

\medskip\noindent
{\bf 3.}	
Suppose we want to move the bar onto the East side of the rectangle. Let a particle of 
type $\ta$ enter $\Lambda$ and reach the site at dual distance $1$ in the East direction 
from the Southern-most particle of type $\ta$ on the East side of the rectangle 
(configuration $\tilde{\eta}\prm$) in order to create a good dual corner (see 
Fig.~\ref{fig:change-bar-side}). Note that $H(\tilde{\eta}\prm) = H(\tilde{\eta}) + \Da$.
From $\tilde{\eta}\prm$, using the mechanism described in Section~\ref{sec-moving-dimers},
we can iteratively move all the $\btiles$ originally on the top $\abbar$ to the East side 
of the rectangle within energy barrier $3U$. The free particle of type $\ta$ that is left 
afterwards is removed from $\Lambda$. Therefore the task can be achieved within energy 
barrier $3U + \Da$. Note that it is sufficient that only the North side and the East side 
of the dual rectangle circumscribing the cluster of $\tilde{\eta}$ are far from 
$\partial^{-}\Lambda$.

\begin{figure}[htbp]
\centering
{\includegraphics[height=6\tilesize]{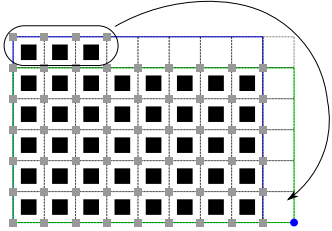}}
\caption{The top bar can be moved to the East side of the rectangle within a 
energy barrier $3U + \Da$.}
\label{fig:change-bar-side}
\end{figure}

\medskip\noindent	
{\bf 4.} 
It is clear that this argument is sufficient to show that from a $\btiled$ configuration  
in $\minnbis{\protonumber}$ with a dual support consisting of a rectangle of side lengths
$L,(l-1)$ plus a bar attached to one of the shortest sides far from $\partial^{-}\Lambda$
(note that at least one of the shortest side is far from $\partial^{-}\Lambda$) it is possible 
to return to a configuration of minimal energy with a dual tile support an $L \times l$ 
rectangle and, eventually, to some standard configuration $\bar{\eta}$ strictly below energy 
level $\Gamma\starred$. This implies (see Lemma~\ref{lemma-reachable-from-standard}) that 
$\comlev(\Box, \eta) < \Gamma\starred$ for all $\eta \in \CF$ whose dual support is far 
from $\partial^{-}\Lambda$. To complete the proof we will have to consider those configurations 
in $\CF$ with a dual tile support consisting of an $L \times l$ rectangle that is close to 
$\partial^{-}\Lambda$ and see that also for these configurations it is possible to 
reach below energy level $\Gamma\starred$ some standard configuration in 
$\minnbis{\protonumber}$. A particle is said to be close to $\partial^{-}\Lambda$ if it 
is adjacent to a site in the region of $\Lambda$ where interaction between particles is 
not possible. Dimers, $\btiles$, $\abbars$ and cluster are said to be close to 
$\partial^{-}\Lambda$ if they contain at least one particle that is close to 
$\partial^{-}\Lambda$.
	
\begin{figure}[htbp]
\centering
\subfigure[]{\includegraphics[height=5\tilesize]{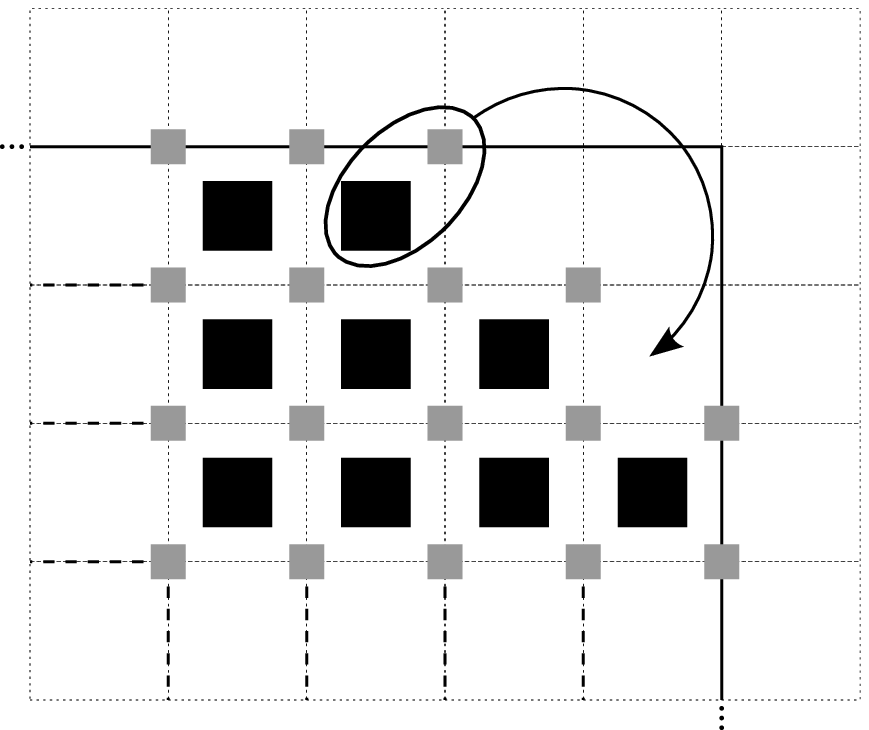}}
\qquad
\subfigure[]{\includegraphics[height=5\tilesize]{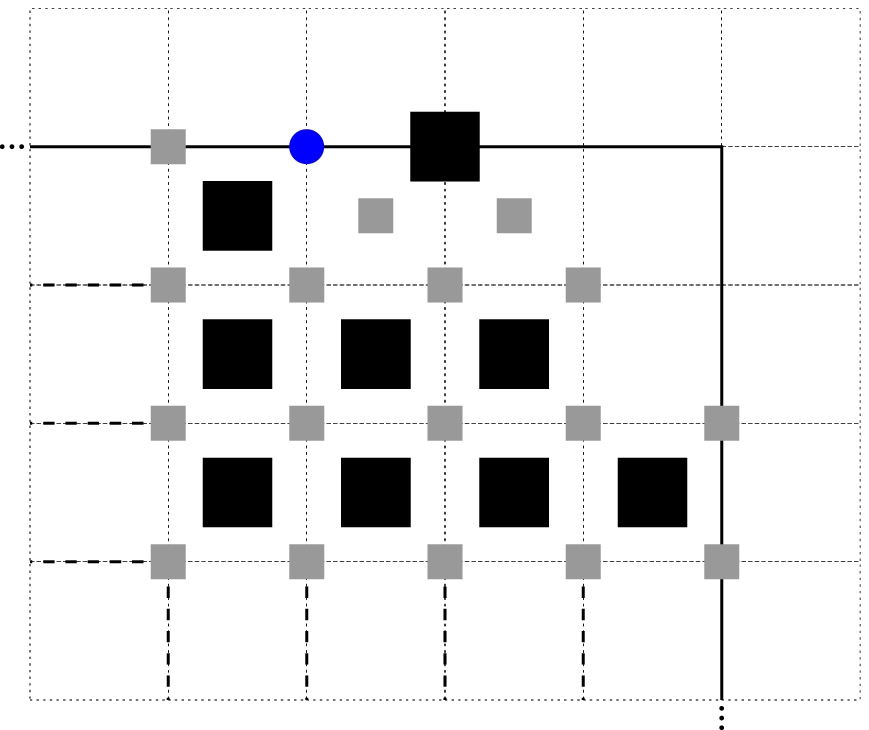}}
\qquad
\subfigure[]{\includegraphics[height=5\tilesize]{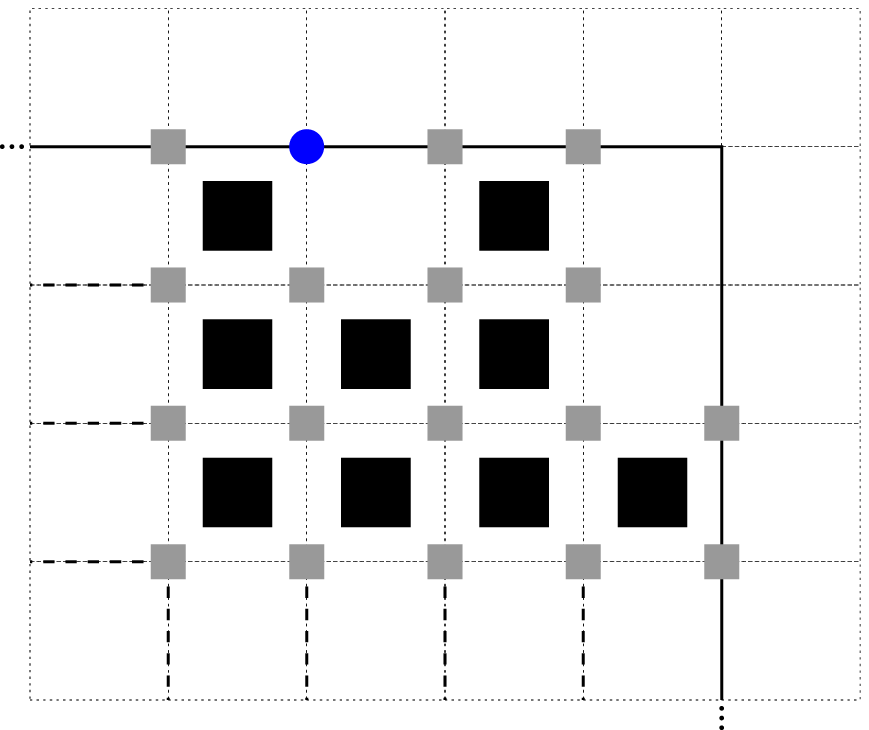}}
\qquad
\subfigure[]{\includegraphics[height=5\tilesize]{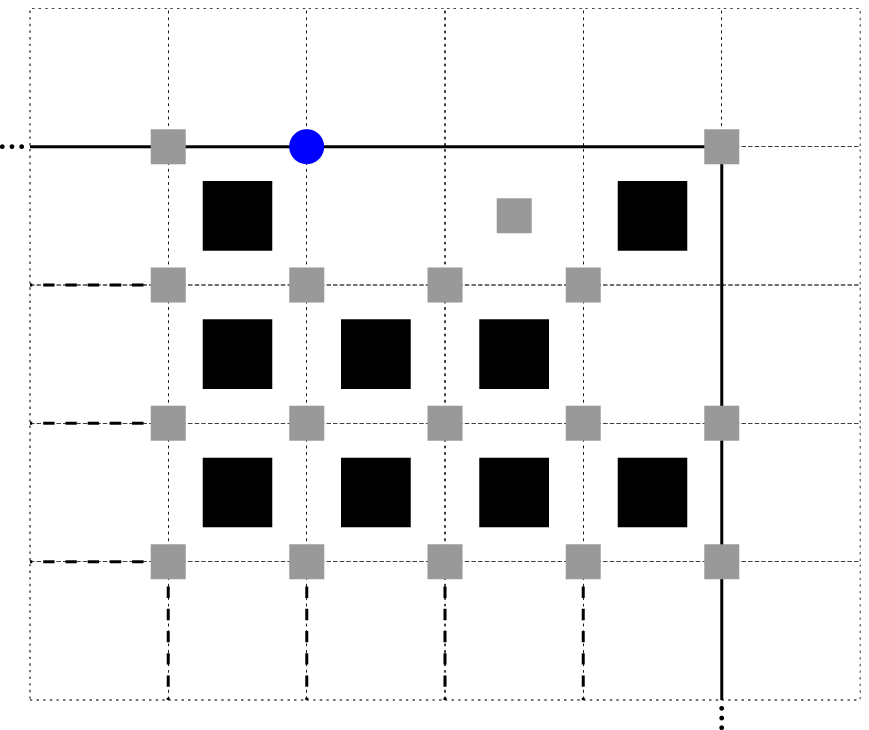}}
\caption{Dimers that are close to $\partial^{-}\Lambda$ can be moved within a  
$3U + \Da$ energy barrier. A particle that is ``beyond'' the black linec an not
have active bonds.}
\label{fig:move-boundary-dimers}
\end{figure}

\medskip\noindent
{\bf 5.}
We will show how it is possible to move those dimers that are close to $\partial^{-}
\Lambda$ within energy barrier $3U + \Da$ using a modification of the argument presented 
in Section~\ref{sec-moving-dimers}. For this purpose, we refer to 
Fig.~\ref{fig:move-boundary-dimers}. Let $p$ and $q_{1}$ be  the particle of type 
$\tb$, respectively, type $\ta$ of the dimer that is encircled in configuration 
$\eta$ in Fig.~\ref{fig:move-boundary-dimers}(a), and let $q_{2}$ be the particle 
of type $\ta$ adjacent to $p$ in the North-East direction. We will construct a path 
$\omega$ that moves the dimer to a different $\abbar$, as follows. Move $q_{1}$ one 
step South-East ($\D H (\omega) = U$), and $p$ one step North-East ($\D H(\omega) = 3U$). 
Then move $q_{2}$ one step South-East ($\D H(\omega) = 3U$), and let a new particle of 
type $\ta$ enter $\Lambda$ ($\D H(\omega) = 3U + \Da$) and reach the site originally 
occupied by $q_{2}$ ($\D H(\omega) = 2U + \Da$; see Fig.~\ref{fig:move-boundary-dimers}(b)).
Afterwards, move $q_{1}$ one step North-East ($\D H(\omega) = 3U + \Da$), $p$ one step 
South-East ($\D H(\omega) = 3U + \Da$), and $q_{2}$ one step North-East ($\D H(\omega) 
= 2U + \Da$; see Fig.~\ref{fig:move-boundary-dimers}(c)). The same procedure described 
so far can be repeated (now it is not necessary to let a new particle of type $\ta$ enter 
$\Lambda$, since $p$ is adjacent to two corner particles of type $\ta$), to reach the 
configuration represented in Fig.~\ref{fig:move-boundary-dimers}(d) (as in the 
case of Section~\ref{sec-moving-dimers}), which is the key configuration to see that 
the dimer can moved to a corner without further increasing the energy difference 
with the original configuration.
 
\medskip\noindent
{\bf 6.}
Let $\eta \in \CC$ consist of a cluster close to $\partial^{-}\Lambda$, and let 
$L, l$ be the side lengths of the rectangle circumscribing its dual tile support. Using 
the mechanism described above, we see that also from $\eta$ it is possible to reach a 
configuration $\eta\prm$ with a dual support consisting of a rectangle of side lengths 
$L-1,(l)$ plus a $\abbar$ (possibly still close to $\partial^{-}\Lambda$) attached to 
one of its sides. 

\begin{itemize}
\item
If $L=l = \ell\starred$, then we can reach a standard configuration and we are done. 
\item
If $L = \ell\starred + 1$ and $l = \ell\starred - 1$ (and hence $L - l = 2$), then 
the same procedure can be used to reach a configuration with a protuberance that can 
be easily moved below energy level $\Gamma\starred$ to obtain a standard configuration 
(detach and remove the two particles of type $\ta$, detach and move the particle of 
type $\tb$, and saturate the particle of type $\tb$ with two new particles of type $\ta$).
\item
If $L - l >  2$, then  it is possible to proceed as follows. Move the short external 
$\abbar$ that is far from $\partial^{-}\Lambda$ onto the longest side of the dual 
rectangle far from $\partial^{-}\Lambda$, as described above within energy barrier 
$3U + \Da$, to obtain a configuration $\eta\prm$ such that $H(\eta\prm) = H(\eta)$ 
and with a dual tile support that has a circumscribing rectangle of side lengths 
$L-1, l+1$. From $\eta\prm$ it is possible to iterate the above procedure until a 
configuration with dual support with a circumscribing rectangle of side lengths 
$\ell\starred + 1, \ell\starred - 1$ is reached. But this case has already been 
treated.  
\end{itemize}
\end{proof}


\subsubsection{Existence of a good site}

\begin{lemma}
\label{lemma-proto-is-good-RF}
For all $\hat{\eta} \in \CC$, there exists an $x \in F(\hat{\eta})$ such that 
$\comlev((\hat{\eta},x),\boxplus) < \Gamma\starred$.
\end{lemma}

\begin{proof}
If $\hat{\eta} \in \CB$, then the claim follows from the proof of 
Lemma~\ref{lemma-proto-is-good-RB}. We will show that, for all $\hat{\eta}$ whose dual 
support has a circumscribing rectangle with side lengths $L,l$, there is a site $x$ 
such that from configuration $(\hat{\eta}, x)$ there is a path $\omega$ to a configuration 
$\tilde{\eta}\prm$ whose dual support has a circumscribing rectangle of side lengths 
$L-1, l + 1$ such that $H(\tilde{\eta}\prm) \le H(\hat{\eta})$ and $H(\xi) < 
\Gamma\starred$ for all $\xi \in \omega$. The procedures that we present only require 
that two sides of the circumscribing rectangle of the dual support of $\eta$ are far 
from $\partial^{-}\Lambda$, which is always the case. Without loss of generality we 
may assume that these two sides are the North side and the East side.

\begin{figure}[htbp]
\centering
{\includegraphics[height=6\tilesize]{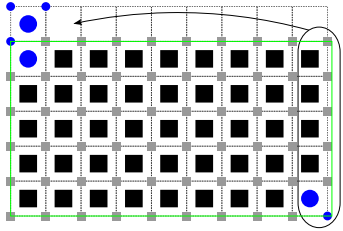}}
\caption{The configuration obtained from  $\hat{\eta}$ by first filling the rectangle 
circumscribing its dual support and then adding a $\btiled$ on the North side; 
circles represent particles added to $\hat{\eta}$.}
\label{fig:reconnect-rf}
\end{figure}	

\begin{itemize}
\item
$Ll > \protonumber$. In this case, configuration $\hat{\eta}$ contains at least one 
good dual corner. Let $x$ be one of these dual corners. From $\hat{\eta}$, first 
saturate the particle of type $\tb$ at $x$ with a new particle of type $\tb$, to 
obtain configuration $\eta\prm$. Let $\eta\dprm$ be the configuration obtained from 
$\eta\prm$ by filling $R(\hat{\eta})$ (note that $\eta\prm = \eta\dprm$ is possible). 
Clearly, $H(\eta\dprm) < H(\hat{\eta})$ and $\comlev(\hat{\eta}, \eta\dprm) < 
\Gamma\starred$. Let $\tilde{\eta}$ be the configuration obtained from $\eta\dprm$ 
within energy barrier $\Db$ by adding a $\btile$ on the North side of the $\btiled$ 
dual rectangle with side lengths $L,l$. $H(\tilde{\eta}) \le \Gamma\starred - 4U + 2\Da$. 
As in the proof of Lemma~\ref{lemma-proto-below-gamma-RF}, all the tiles of the 
Eastern-most $\abbar$ of $\tilde{\eta}$ can be itereatively moved to a corner on 
the North side of the cluster within energy barrier $3U$ (see Fig.~\ref{fig:reconnect-rf}), 
to obtain a $\btiled$ configuration $\tilde{\eta}\prm$ whose dual support has a 
circumscribing rectangle with side lengths $L-1, l+1$ and such that $\tilde{\eta}\prm 
= H(\eta\dprm)$. Hence $\comlev((\hat{\eta},x), \tilde{\eta}\prm) le \Gamma\starred 
- 4U + 2\Da + 3U < \Gamma\starred$ as soon as $\Da < \tfrac{1}{2}$, which is satisfied 
in $\RF$.
\item	
$Ll > \protonumber$. In this case the dual support of $\hat{\eta}$ is already a rectangle 
with side lengths $L, l$. Let $x$ be the central site of a tile adjacent to the North 
side of the rectangle. We have $H((\hat{\eta},x)) = \Gamma\starred - 2U$. Let $\tilde{\eta}$ 
be the configuration obtained within energy barrier $\Da$ by saturating the particle of 
type $\tb$ at $x$. $H(\tilde{\eta}) = \Gamma\starred - 4U + 2\Da$. From $\tilde{\eta}$ 
proceed as in the previous case.
\end{itemize}
\end{proof}


\subsubsection{Identification of $\proto$ and $\entgate$}

\begin{lemma}
\label{lemma-identification-of-entgate-RF}
$\proto = \CC$ and $\entgate = \CC\boub$.
\end{lemma}

\begin{proof}
Same as the proof of Lemma~\ref{lemma-identification-of-entgate-RA}.
\end{proof}


\appendix
\section{Computation of $N\starred$}
\label{appA}

Kurz~\cite{K08} shows how to construct all polyominoes of minimal perimeter 
with fixed area, and gives an expression for their number in terms of 
generating functions. Two basic generating functions
\begin{equation}
s(x)=1+\sum_{k=1}^\infty x^{k^2}\prod\limits_{j=1}^k\frac{1}{1-x^{2j}},
\qquad 
a(x)=\prod\limits_{j=1}^\infty\frac{1}{1-x^{j}}
\end{equation}
are used to define two composite generating functions 
\begin{equation}
r(x)=\tfrac{1}{4}\left[a(x)^4+3a(x^2)^2\right],
\qquad
q(x)= \tfrac{1}{8}\left[a(x)^4+3a(x^2)^2+2s(x)^2a(x^2)+2a(x^4)\right],
\end{equation}
whose coefficients $r_{k}$, $q_{k}$ of $x^k$ count the polyominoes as follows.
The number of polyominoes of minimal perimeter with area $n$ equals
\begin{equation}
e(n) = \left\{\begin{array}{ll}
1 
&\text{ if } n=s^2,\\[0.2cm]
\sum_{c=0}^{\left\lfloor-\tfrac12+\tfrac12\sqrt{1+4s-4t}\right\rfloor} r_{s-c-c^2-t}
&\text{ if } n=s^2+t \text{ with } 0<t<s,\\[0.2cm]
1
&\text{ if } n=s^2+s+t,\\[0.2cm]
q_{s+1-t} + \sum_{c=1}^{\left\lfloor\sqrt{s+1-t}\right\rfloor} r_{s+1-c^2-t}
&\text{ if } n=s^2+s+t \text{ with } 0<t\leq s,
\end{array}
\right.
\end{equation}
where $s=\lfloor\sqrt{n}\rfloor$.

We need to count the number of polyominoes of minimal perimeter with area 
$n = \protonumber$ for $\ell\starred \ge 4$, i.e., we are only interested 
in $n$ of the form $s^{2} + s + t$ with $s = \ell\starred - 1$ and $t = 1$. 
Kurz~\cite{K08} counts polyominoes modulo translations, rotations and reflections. 
We need the number modulo translations only. Therefore we must put in correction 
factors: $4$ for the rotations and $2$ for the reflections. 

In region $\RC$ we retain all $c$-terms. In region $\RB$ we only retain the term 
with $c = 1$. Indeed, $q_{\ell\starred - 1}$ is the number of polyominoes of 
minimal perimeter when the circumscribing rectangle is a square of side length 
$\ell\starred$, and $r_{\ell\starred - c^{2} - 1}$ is the number of polyominoes 
of minimal perimeter when the circumscribing rectangle has side lengths 
$\ell\starred + 1,\ell\starred - 1$. Thus, modulo rotations and reflections, we 
have
\begin{equation}
N\starred = 
\left\{\begin{array}{ll}
8\left[q_{\ell\starred - 1} + r_{\ell\starred - 1 - 1}\right]
&\text{ in region } \RB,\\[0.2cm]
8\left[q_{\ell\starred - 1} 
+ \sum_{c=1}^{\left\lfloor\sqrt{\ell\starred - 1}\right\rfloor} 
r_{\ell\starred - c^{2} - 1}\right]
&\text{ in region } \RC.
\end{array}
\right.
\end{equation}


\section{Clarification of some statements in \cite{dHNT12}}


\subsection{Proof of Lemma~1.18 in \cite{dHNT12}}
\label{appB}

The following statement was used in the proof of Lemma~2.2 in \cite{dHNT12}.

\begin{lemma}
\label{lemma-entgate-is-minimal-gate}
$\entgate$ is a minimal gate.
\end{lemma}

\begin{proof}
Let $\cA = \{\eta\in\entgate\colon\,\exists\,\omega \in \Omega(\eta)\colon\, 
\omega\cap\entgate=\{\eta\}\}$, and let $\tilde{\cA} = \entgate \backslash \cA$.
In words, for each $\eta \in \cA$ there is a path in $\optpaths$ entering 
$\cG(\Box,\boxplus)$ via $\eta$ and reaching $\boxplus$ without hitting $\entgate$ 
again, while if a path in $\optpaths$ enters $\entgate$ via a configuration in 
$\tilde{\cA}$, then it must go back to $\entgate$ before reaching $\boxplus$.
We will show that $\tilde{\cA}$ is empty. The proof is by contradiction.
	
Assume that $\eta \in \tilde{\cA}$ and let $\omega \in \Omega(\eta)$. Let $\zeta$ 
be the last configuration in $\entgate$ visited by $\omega$ before reaching $\boxplus$. 
Then there is an $\omega\prm \in \optpaths$ entering $\cG(\Box,\boxplus)$ via $\zeta$.
The path $\omega\dprm$ obtained by joining the part of $\omega\prm$ from $\Box$ to 
$\zeta$ and the part of $\omega$ from $\zeta$ to $\boxplus$ belongs to $\optpaths$, 
and $\eta \notin \omega\dprm$. Therefore $\eta$ is unessential and, by Theorem~5.1 
in \cite{MNOS04} (Lemma~\ref{lemma-mnos-essential-saddles} in this paper), it does 
not belong to $\cG(\Box,\boxplus)$. This contradicts the assumption $\eta \subset 
\entgate \subset \cG(\Box, \boxplus)$.
\end{proof}

\noindent
{\bf Lemma~1.18 in \cite{dHNT12}} 
\emph{{\rm (H3a), (H3-c)} and Definition~{\rm\ref{defdroplets-b}(a)} imply that for every 
$\eta\in\gate_\mathrm{att}$ all paths in $(\eta\to\Box)_\mathrm{opt}$ pass through 
$\entgate$}.

\medskip
\begin{proof}
The proof is by contradicion. Let $\gateatt \ni \eta = (\hat{\eta}, x)$, and assume that 
there is a path $\omega_{1}\colon\,\Box \to \eta$ that does not visit $\entgate$ and such 
that $H(\sigma) \le \Gamma\starred$ for all $\sigma \in \omega_{1}$. By the definition of 
$\gateatt$, there exists a configuration $\entgate \ni \zeta = (\hat{\eta}, z)$ such that 
$\zeta$ is obtained from $\eta$ by moving the particle of type $\tb$ from $x$ to $z$.
Consequently, there exists a path $\omega_{2}$ from $\eta$ to $\zeta$ consisting of a 
sequence of configurations of the type $(\hat{\eta}, y_{i})$ with $y_{i} \in \Lambda$ 
for all $i$. Note that, since $H(\zeta) = \Gamma\starred$, the particle of type $\tb$ 
at site $z$ in $\zeta$ has no active bond (it is in $\partial^{-}\Lambda$) and all 
configurations in $\omega_{2}$ have the same number of particles of both types, we have 
$H(\sigma) \le \Gamma\starred$ for all $\sigma$ in $\omega_{2}$. Since $\zeta$ belongs 
to the minimal gate $\entgate$, there is a path $\omega_{3}\colon\,\zeta \to \boxplus$ 
such that $\omega_{3} \cap \entgate = \{ \zeta \}$ and such that $H(\sigma) \le \Gamma
\starred$ for all $\sigma \in \omega_{3}$. 

Now consider the following two cases.
\begin{itemize}
\item[(1)] 
$\eta \in \omega_{3}$. Let $\omega_{4}$ be the part of $\omega_{3}$ from $\eta$ 
to $\boxplus$. Then the path obtained by joining $\omega_{1}$ and $\omega_{4}$ is a path 
in $\optpaths$ that does not visit $\entgate$. This contradicts the definition of $\entgate$.
\item[(2)] 
$\eta \notin \omega_{3}$. Let $\omega = \omega_{1} + \omega_{2} + \omega_{3}$. 
Then, by construction, $\omega \in \optpaths$. Let $\pi = (\hat{\eta}, w)$ be the configuration 
in $\omega$ visited just before $\zeta$. By definition, $\pi \in \proto$. By (the first part 
of) (H3-a), $\pi$ consists of a single droplet, and hence $w \notin \partial^{-}\Lambda$.
Therefore $\zeta$ is obtained by ``breaking'' a droplet of a configuration in $\proto$ and 
not by adding a particle of type $\tb$ in $\partial^{-}\Lambda$ to a configuration in 
$\proto$. This contradicts (the second part of) (H3-a).
\end{itemize}
\end{proof}


\subsection{Proof of Lemma~2.2 in \cite{dHNT12}}

In the proof of Lemma~2.2 in \cite{dHNT12}, the sentence 

\noindent
``Denote by $\Omega(\eta)$ the set of all optimal paths from $\Box$ to $\boxplus$
that enter $\cG(\Box,\boxplus)$ via $\eta$ (note that this set is
non-empty because $\entgate$
is a minimal gate by Definition~\ref{defdroplets-a}(a)). By
Definition~\ref{defdroplets-a}(b),
$\omega_i \in \Omega(\eta)$ visits $\hat\eta$ before $\eta$ for all $i
\in 1,\ldots,|\Omega(\eta)|$.''

\noindent
must be replaced by

\noindent
``Denote by $\Omega(\eta)$ the set of all optimal paths from $\Box$ to $\boxplus$
such that $\omega\cap\entgate = \{ \eta \}$ (note that this set is
non-empty because $\entgate$
is a minimal gate). By Definition~\ref{defdroplets-a}(b) and by the more precise form 
of (H3-a), $\omega_i \in \Omega(\eta)$ visits $\hat\eta$ before $\eta$ for all $i
\in 1,\ldots,|\Omega(\eta)|$.''


\subsection{Proof of Theorem~1.8 in \cite{dHNT12}}\label{appb3}

In Step~3, the definition of $\rm{CS^{++}}(\hat{\eta})$ should read:
$\rm{CS^{++}}(\hat{\eta}) = \partial^{+}CS^{+}(\hat{\eta})\cap\Lambdaminus$.

In Lemma~2.11, the estimate of $\Theta_1$ in formula (2.48) should be replaced by
\begin{align}
\Theta_{1} &= [1+o(1)] \sum_{\hat{\eta}\in\proto\prm} \CAPA^{\,\Lambda^+} 
\left(\partial^+\Lambda,\CS(\hat{\eta})\right),
\end{align}
where $\proto\prm$ is the set of configurations in $\proto$ whose support has lattice 
distance at least two from $\partial^{-}\Lambda$. This allow us to always construct a 
lattice path connecting every two points adjacent to $\supp(\hat{\eta})$ and avoiding 
$\partial^{+}\Lambda$, as in the argument used to derive bounds for formula (2.56).
The value of $\Theta_{1}$ is derived by restricting the sum in (2.50) to $\hat{\eta} 
\in \proto\prm$. This still produces a lower bound for $\Theta$, because $\proto\prm 
\subset \proto$. $\proto$ should be changed to $\proto\prm$ accordingly in formulas~(2.53) 
in Step~3 and (2.61) in Step~4. The bounds $\Theta_{1}$ and $\Theta_{2}$ will still merge
asymptotically, as shown in Step~4, since $|\proto| \sim N\starred |\Lambda| \sim |\proto\prm|$ 
as $\Lambda\to\Z^{2}$.

In Step~3, where $\Theta_{2}$ is computed, the sentence

\noindent
``The only transitions in $\cX\starred$ between $\cC^{++}$ and $\cX\starred\backslash
[\cX_\Box\cup\cC^{++}]$ are those where the free particle moves from distance $2$ to 
distance $1$ of the protocritical droplet.''

\noindent
should be replaced by

\noindent
``The only transitions in $\cX\starred$ between $\cC^{++}$ and $\cX\starred\backslash
[\cX_\Box\cup\cC^{++}]$ are those where the free particle enters $\rm{CS^{++}}(\hat{\eta})$.''


\end{document}